\documentclass[a4paper,12pt]{amsart}
\usepackage{amscd,amsmath,amssymb,amsfonts,xypic}
\usepackage[all]{xy}

\numberwithin{subsection}{section}
\numberwithin{equation}{subsection}
\newcommand{\eq}[2]{\begin{equation}\label{#1}#2 \end{equation}}
\newcommand{\eqn}[1]{\begin{equation*}#1\end{equation*}}
\newcommand{\eqa}[2]{\begin{eqnarray}\label{#1}#2 \end{eqnarray}}
\newcommand{\eqna}[1]{\begin{eqnarray*}#1\end{eqnarray*}}
\newcommand{\ml}[2]{\begin{multline}\label{#1}#2 \end{multline}}
\newcommand{\mln}[1]{\begin{multline*}#1\end{multline*}}
\newcommand{\ga}[2]{\begin{gather}\label{#1}#2 \end{gather}}
\newcommand{\gan}[1]{\begin{gather*}#1\end{gather*}}

\theoremstyle{plain}
\newtheorem{thm}[subsection]{Theorem}
\newtheorem{lem}[subsection]{Lemma}
\newtheorem{cor}[subsection]{Corollary}
\newtheorem{prop}[subsection]{Proposition}

\theoremstyle{definition}
\newtheorem{defn}[subsection]{Definition}
\newtheorem{rmk}[subsection]{Remark}

\newcommand{\DS}{\displaystyle}

\newcommand{\num}{n$^{\circ}\,$}

\newcommand{\gr}{{\rm gr}}

\newcommand{\surj}{\twoheadrightarrow}
\newcommand{\inj}{\hookrightarrow}

\DeclareMathOperator{\Spec}{Spec}

\newcommand{\can}{\mathrm{can}}

\newcommand{\sA}{{\mathcal A}}

\newcommand{\sD}{{\mathcal D}}
\newcommand{\sE}{{\mathcal E}}
\newcommand{\sF}{{\mathcal F}}

\newcommand{\sH}{{\mathcal H}}
\newcommand{\sI}{{\mathcal I}}
\newcommand{\sJ}{{\mathcal J}}
\newcommand{\sK}{{\mathcal K}}

\newcommand{\sN}{{\mathcal N}}
\newcommand{\sO}{{\mathcal O}}
\newcommand{\sP}{{\mathcal P}}

\newcommand{\sR}{{\mathcal R}}

\newcommand{\sT}{{\mathcal T}}

\renewcommand{\H}{{\mathbb H}}

\renewcommand{\L}{{\mathbb L}}

\newcommand{\N}{{\mathbb N}}
\renewcommand{\P}{{\mathbb P}}
\newcommand{\Q}{{\mathbb Q}}
\newcommand{\RR}{\mathbb{R}}

\newcommand{\U}{{\mathbb U}}
\newcommand{\V}{{\mathbb V}}
\newcommand{\W}{{\mathbb W}}

\newcommand{\Y}{{\mathbb Y}}
\newcommand{\Z}{{\mathbb Z}}

\newcommand{\fU}{\mathfrak{U}}

\newcommand{\fa}{\mathfrak{a}}

\newcommand{\tu}{\tilde{u}}

\newcounter{romain}[subsection]
\newcommand{\romain}
{\stepcounter{romain}
\noindent\makebox[1.5cm][r]{{\normalfont(\roman{romain})}\hspace{0.3cm}}}

\newcounter{numero}[subsection]
\newcommand{\numero}
{\stepcounter{numero}%\vspace{0.2cm}
\noindent\makebox[1.4cm][r]{\arabic{numero})\hspace{0.3cm}}}

\newcounter{alphab}[subsection]
\newcommand{\alphab}
{\stepcounter{alphab}
\noindent\hspace{0.5cm}{\normalfont\alph{alphab})}\hspace{0.3cm}}
\newcommand{\resp}{resp.\ }
\newcommand{\Id}{\mathrm{Id}}
\newcommand{\Frac}{\mathrm{Frac}}
\newcommand{\Ker}{\mathrm{Ker}}

\newcommand{\Fil}{\mathrm{Fil}}

\newcommand{\tor}{\mathcal{T}\!or}
\newcommand{\End}{\mathcal{E}\mspace{-1mu}nd}
\DeclareMathOperator{\Spf}{Spf}
\DeclareMathOperator{\Spm}{Spm}
\newcommand{\spm}{\mathrm{sp}}

\newcommand{\Zar}{_{\mathrm{Zar}}}
\newcommand{\cris}{_{\mathrm{crys}}}
\newcommand{\Cris}{\mathrm{Crys}}
\newcommand{\rig}{_{\mathrm{rig}}}
\newcommand{\rigc}{_{\mathrm{rig,c}}}
\newcommand{\tot}{_{\mathrm{t}}}
\newcommand{\sbul}{\raisebox{.1mm}{\tiny$\mspace{2mu}\bullet$}}
\newcommand{\Omd}[1]{\Omega^{\sbul}_{#1}}
\newcommand{\WOm}[3]{W_{#1}\Omega^{#2}_{#3}}
\newcommand{\WOmd}[2]{W_{#1}\Omega^{\sbul}_{#2}}

\newcommand{\Db}{D^{\mathrm{b}}}
\newcommand{\sPh}{\widehat{\sP}{}}
\newcommand{\sEh}{\widehat{\sE}}

\newcommand{\sPb}{\overline{\sP}{}}
\newcommand{\sPt}{\widetilde{\sP}{}}
\newcommand{\sAt}{\widetilde{\sA}{}}
\newcommand{\sJm}{\sJ^{(m)}}

\newcommand{\hbul}{^{\sbul}}

\newcommand{\lbul}{_{\sbul}}
\newcommand{\otimesh}{\widehat{\otimes}}
\newcommand{\otimesL}{\stackrel{\L}{\otimes}}

\newcommand{\ot}{\overline{t}}
\renewcommand{\tt}{\tilde{t}}

\newcommand{\uf}{\underline{f}}
\newcommand{\uk}{\underline{k}}
\newcommand{\up}{\underline{p}}
\newcommand{\ualpha}{\underline{\alpha}}
\newcommand{\uGamma}{\underline{\Gamma}}

\newcommand{\lra}{\longrightarrow}
\newcommand{\xra}[1]{\xrightarrow{#1}}
\newcommand{\riso}{\xrightarrow{\ \sim\ \,}}

\newcommand{\sPbp}[2]{\sPb\mspace{1mu}'{}^{\mspace{2mu}(#1)}_{\hspace{-1.5mm}#2}}

\newcommand{\rhob}{\overline{\rho}{\mspace{1mu}}}
\newcommand{\rhoh}{\hat{\rho}}
\newcommand{\rhot}{\tilde{\rho}}
\newcommand{\pih}{\hat{\pi}}
\newcommand{\sigmat}{\tilde{\sigma}}
\newcommand{\oF}{\overline{F}}
\newcommand{\EWOm}[3]{\widehat{#1}^W\mspace{1mu}\otimesh_{W\sO_#3}\WOm{}{#2}{#3}}
\newcommand{\cc}{\mathrm{c}}

\renewcommand{\)}{{\normalfont{)}}}
\newcommand{\nref}[1]{{\normalfont{\ref{#1}}}}

\begin{document}

\title[Rigid Cohomology and de Rham-Witt complexes]{
Rigid Cohomology and de Rham-Witt complexes}
\author{Pierre Berthelot}
\address{IRMAR, Universit\'e de Rennes 1,
Campus de Beaulieu,
35042 Rennes cedex, France
}
\email{pierre.berthelot@univ-rennes1.fr}

\date{\today}  % \normalfont{May 21, 2012}

\vspace{-1cm}

\begin{abstract}
Let $k$ be a perfect field of characteristic $p > 0$, $W_n = W_n(k)$. For separated
$k$-schemes of finite type, we explain how rigid cohomology with compact supports can
be computed as the cohomology of certain de Rham-Witt complexes with coefficients. This
result generalizes the classical comparison theorem of Bloch-Illusie for proper and
smooth schemes. In the proof, the key step is an extension of the Bloch-Illusie theorem
to the case of cohomologies relative to $W_n$ with coefficients in a crystal that 
is only supposed to be flat over $W_n$.
\end{abstract}


\maketitle
\vspace{-0.2cm}

\begin{flushright}
\textit{Dedicated to Francesco Baldassarri,\\
on his sixtieth birthday}
\end{flushright}

\setcounter{tocdepth}{1}

\tableofcontents
\vspace{-0.7cm}

%%%%%%%%%%%%%%%%%%%%%%%%%%%%%%%%%%%%%%%%%%%%%%%%%%%%%%%%%%%%%%%%
\section{Introduction}
\medskip\setcounter{subsection}{1}
%%%%%%%%%%%%%%%%%%%%%%%%%%%%%%%%%%%%%%%%%%%%%%%%%%%%%%%%%%%%%%%%

Let $k$ be a perfect field of characteristic $p > 0$, $W_n = W_n(k)$ (for all $n \geq
1$), $W = W(k)$, $K = \Frac(W)$. If $X$ is a proper and smooth scheme over $k$, and if 
$\WOmd{}{X} = \varprojlim_n \WOmd{n}{X}$ denotes the de Rham-Witt complex of $X$, the
classical comparison theorem between crystalline and de Rham-Witt cohomologies
(\cite[III, Th. 2.1]{Bl77}, \cite[II, Th. 1.4]{Il79}) provides canonical isomorphisms
%%%%%%%%%%%%%%%%
\ga{fCrisndRW}{H^{\ast}\cris(X/W_n)\ \riso\ H^{\ast}(X, \WOmd{n}{X}), \\
H^{\ast}\cris(X/W)\ \riso\ H^{\ast}(X, \WOmd{}{X}), \label{fCrisdRW} \\
H^{\ast}\cris(X/W)_K\ \riso\ H^{\ast}(X, \WOmd{}{X,K}), \label{fCrisKdRW} }
%%%%%%%%%%%%%%%%
where the subscript ${}_K$ denotes tensorization with $K$. In particular, the latter 
allows to study the slope decomposition of crystalline cohomology under the Frobenius 
action, thanks to the degeneracy of the spectral sequence defined by the filtration of
$\WOmd{}{X}$ by the subcomplexes $\WOm{}{\geq i}{X}$ \cite[II, Cor.\ 3.5]{Il79}.

In this article, our main goal is to generalize the isomorphism \eqref{fCrisKdRW} to
the case of a separated $k$-scheme of finite type. Then $H^*\cris(X/W)_K$ is no longer
in general a good cohomology theory, but we may use instead rigid cohomology with
compact supports, which coincides with $H^*\cris(X/W)_K$ when $X$ is proper and smooth,
while retaining all the standard properties of a topological cohomology theory. The key
case is the case of a proper $k$-scheme, as the non proper case can be reduced to the
proper one using a mapping cone description. We will therefore assume for the rest of
the introduction that $X$ is proper.

When $X$ is singular, the classical theory of the de Rham-Witt complex can no longer be
directly applied to $X$. Instead, we will consider a closed immersion of $X$ into a
smooth $k$-scheme $Y$ of finite type, and we will use on $Y$ de Rham-Witt complexes
with coefficients in crystals, as introduced by \'Etesse in \cite{Et88}. More
precisely, we will construct on $Y$ a $W\sO_{Y,K}$-algebra $\sA^W_{X,Y}$, supported in
$X$ and closely related to the algebra of analytic functions on the tube of $X$ in
$\Y$ when $Y$ admits a lifting as a smooth $p$-adic formal scheme $\Y$ over $W$. This
algebra is endowed with a de Rham-Witt connection, which extends so as to define a
complex $\sA^W_{X,Y} \otimesh_{W\sO_Y} \WOmd{}{Y}$ (the tensor product being completed
for the canonical topology of the de Rham-Witt complex). Our main result, which is part
of Theorem \ref{ssMain}, is then the existence of a functorial isomorphism
%%%%%%%%%%%%%%%%
\eq{fMain0}{ H^{\ast}\rig(X/K) \riso H^{\ast}(Y, \sA^W_{X,Y} \otimesh_{W\sO_Y}
\WOmd{}{Y}), }
%%%%%%%%%%%%%%%%
which coincides with \eqref{fCrisKdRW} when $X$ is smooth and $Y = X$. 

A key ingredient in the proof of Theorem \ref{ssMain} is the use of Theorem
\ref{ssCrysnEtodRWE} and of its Corollary \ref{ssCrysEtodRWE}, which provide
generalizations of the comparison isomorphisms \eqref{fCrisndRW} and \eqref{fCrisdRW}
to the case of cohomologies of a smooth $k$-scheme $Y$ with coefficients in a crystal
in $\sO_{Y/W_n}$-modules (resp.\ $\sO_{Y/W}$-modules) $\sE$. Such results have been
previously obtained by \'Etesse \cite[Th\'eor\`eme 2.1]{Et88} and Langer-Zink
\cite[Theorem 3.8]{LZ04} when $\sE$ is flat over $\sO_{Y/W_n}$ (resp.\ $\sO_{Y/W}$).
However, this assumption is not verified in our situation, since we work on the smooth
scheme $Y$ with crystals supported in a closed subscheme $X \subset Y$. We give here a
proof of the comparison theorem with coefficients in $\sE$ that only requires $\sE$ to
be flat over $W_n$, in the sense of Definition \ref{ssFlatCryst} (resp.\ quasi-coherent
and flat relative to $W$). It provides under these assumptions functorial isomorphisms
%%%%%%%%%%%%%%%%
\ga{fCrisndRWE}{H^{\ast}\cris(Y/W_n, \sE)\ \riso\ 
H^{\ast}(Y, \sE_n^W \otimes_{W_n\sO_Y} \WOmd{n}{Y}), \\
H^{\ast}\cris(Y/W, \sE)\ \riso\ H^{\ast}(Y, \sEh^W \otimesh_{W_n\sO_Y} \WOmd{}{Y}), 
\label{fCrisdRWE} }
%%%%%%%%%%%%%%%%
where $\sE_n^W$ is the evaluation of the crystal $\sE$ on the PD-thickening 
$(Y,W_nY)$ (resp.\ $\sEh^W = \varprojlim_n \sE_n^W$). As a first application, we get that,
when the immersion $X \inj Y$ is regular (in particular, when $X$ is smooth), there 
exists functorial isomorphisms
%%%%%%%%%%%%%%%%
\ga{fCrisndRWY}{ H^{\ast}\cris(X/W_n) \riso H^{\ast}(Y, \sP_{X,Y,n}^W \otimes_{W_n\sO_Y}
\WOmd{n}{Y}), \\
H^{\ast}\cris(X/W) \riso H^{\ast}(Y, \sPh_{X,Y}^W \otimesh_{W\sO_Y} \WOmd{}{Y}), }
%%%%%%%%%%%%%%%%
where $\sP_{X,Y,n}^W$ is the divided power envelope of $\Ker(W_n\sO_Y \surj W_n\sO_X)$
with compatibility with the canonical divided powers of $VW_{n-1}\sO_Y$, and $\sPh_{X,Y}^W =
\varprojlim_n \sP_{X,Y,n}^W$.

Note that, even if $X$ is smooth, the possibility of computing the crystalline
cohomology of $X$ using the de Rham-Witt complex of a smooth embedding was not
previously known. Our results imply in particular that, as objects of the derived
category $\Db(X,K)$, the complexes $\sP_{X,Y,n}^W \otimes_{W_n\sO_Y} \WOmd{n}{Y}$,
$\sPh_{X,Y}^W \otimesh_{W\sO_Y} \WOmd{}{Y}$, and $\sA^W_{X,Y} \otimesh_{W\sO_Y}
\WOmd{}{Y}$ do not depend, up to canonical isomorphism, on the embedding $X \inj Y$. It
would be interesting to have a direct proof of this fact in the form of appropriate
Poincar\'e Lemmas.
\medskip

We now briefly describe the content of each section. 

In section \ref{Torind}, we prove some Tor-independence properties of the sheaves of
Witt differentials $\WOm{n}{j}{Y}$ on a smooth $k$-scheme $Y$, and, more generally, of
the graded modules associated to their canonical filtration or to their $p$-adic
filtration. Specifically, we show in Theorems \ref{ssTorind2} and \ref{ssTorind3} the
vanishing of the higher Tor's involving such sheaves and the evaluation on the
PD-thickening $(Y,W_nY)$ of a crystal on $Y/W_n$ that is flat over $W_n$ -- a result
that may seem surprising at first sight, given the intricate structure of the sheaves
of Witt differentials. The flatness assumption allows to reduce to similar statements
for the reduction modulo $p$ of the crystal. The key point is then the existence of
a filtration of the crystal such that the corresponding graded pieces have
$p$-curvature $0$. Thanks to Cartier's descent, this allows to write them as Frobenius
pullbacks, and it is then possible to conclude using the local freeness results proved 
in \cite{Il79} when the graded modules associated to $\WOm{n}{j}{Y}$ are viewed as
$\sO_Y$-modules through an appropriate Frobenius action.

We use these results in section \ref{dRWCoeff} to prove the comparison theorem between
crystalline and de Rham-Witt cohomologies with coefficients in a crystal that is flat
over $W_n$ (Theorem \ref{ssCrysnEtodRWE}). As in the constant coefficient case, the
proof proceeds by reduction to the associated graded complexes for the $p$-adic and 
for the canonical filtration. The Tor-independence results of the previous sections
allow to reduce to the fact that, in the constant coefficient case, one gets
quasi-isomorphisms between complexes that can be viewed as strictly perfect complexes
of $\sO_Y$-modules when $\sO_Y$ acts through an appropriate power of Frobenius. The
section ends with the above mentionned application to crystalline cohomology for
regularly embedded subschemes.

We begin section \ref{RigCoh} by recalling the definition of rigid cohomology for a
proper $k$-scheme $X$. Given a closed immersion of $X$ in a smooth $p$-adic formal
scheme $\P$ over $W$, let $\sA_{X,\P}$ be the direct image by specialization of the
sheaf of analytic functions on the tube of $X$ in $\P$. We explain how
$\sA_{X,\P}$ can be identified with the inverse limit of the completed divided power
envelopes (tensorized with $K$) of the ideals of the infinitesimal neighbourhoods of
$X$ in $\P$. We also give a variant of this result in which these envelopes are
replaced by their quotients by the ideal of $p$-torsion sections. If $X$ is proper over
$k$ and is embedded as a closed subscheme in a smooth $k$-scheme $Y$, we derive from
this construction an isomorphism between the rigid cohomology of $X$ and the derived
inverse limit of the crystalline cohomologies (tensorized with $K$) of the
infinitesimal neighbourhoods of $X$ in $Y$ (Theorem \ref{ssCompRigCrys}).

In section \ref{sProofMain}, we keep these last hypotheses, and we use the crystalline
nature of the previous constructions to define two inverse systems of
$W\sO_{Y,K}$-algebras canonically associated to $X \inj Y$, with isomorphic inverse
limits. By definition, this common inverse limit is the $W\sO_{Y,K}$-algebra
$\sA^W_{X,Y}$ entering in the isomorphism \eqref{fMain0}. To define \eqref{fMain0} and 
to prove Theorem \ref{ssMain},
we combine the isomorphism of the previous section, which identifies the rigid
cohomology of $X$ with the limit of the crystalline cohomologies of its infinitesimal
neighbourhoods in $Y$, with an isomorphism derived from the comparison theorem
\ref{ssCrysnEtodRWE}, which gives an identification of this limit with the cohomology
of the complex $\sA^W_{X,Y\,} \otimesh_{W\sO_Y} \WOmd{}{Y}$. The first identification
uses the description of $\sA^W_{X,Y}$ as a limit of PD-envelopes, while, because of the
$W_n$-flatness assumption in Theorem \ref{ssCrysnEtodRWE}, the second one uses the
second description given in section \ref{RigCoh}, based on $p$-torsion free quotients
of PD-envelopes. We end the section by extending \eqref{fMain0} to rigid cohomology
with compact supports for open subschemes of $X$.

To conclude this article, we explain the relation between the isomorphism 
\eqref{fMain0} and the isomorphism \cite[(1.3)]{BBE}, which identifies the slope $< 1$ 
part of rigid cohomology with Witt vector cohomology. As a consequence, we obtain that 
the part of slope $\geq 1$ of $H^*\rig(X/K)$ can be identified with the cohomology of a
subcomplex of $\sA^W_{X,Y\,} \otimesh_{W\sO_Y} \WOmd{}{Y}$. We hope that Theorem
\ref{ssMain} can also be used to provide a description of parts of higher slopes in
$H^*\rig(X/K)$, but this is still an open question at this point.
\medskip

\noindent \textbf{Acknowledgements}

A first version of some of the results of this article was obtained while working on
the joint article \cite{BBE} with S.\ Bloch and H.\ Esnault. I thank them heartily for
their inspiration, as well as for many stimulating discussions.

The first lecture I gave on these topics took place during a workshop on rigid
cohomology and $F$-isocrystals organized jointly with F.\ Baldassarri in Rennes in June
2005. It is a great pleasure for me to dedicate this article to him as an expression of
my gratitude for many years of fruitful collaboration.

\medskip 

\noindent \textbf{General conventions}

\numero In the whole article, we denote by $k$ a perfect field of characteristic $p$, by
$W_n = W_n(k)$ (for $n \geq 1$) and $W = W(k)$ the usual rings of Witt vectors with
coefficients in $k$, and by $K$ the fraction field of $W$. Formal schemes over $W$ 
are always supposed to be $p$-adic formal schemes.

\numero If $Y$ is a $k$-scheme, we call \textit{crystal on $Y/W_n$} (resp.\
$Y/W$) a crystal in $\sO_{Y/W_n}$-modules (resp.\ $\sO_{Y/W}$-modules) on the usual
crystalline site $\Cris(Y/W_n)$ (resp.\ $\Cris(Y/W)$) \cite[Def.\ 6.1]{BO}.

\numero From section \ref{dRWCoeff} till the end of the article, we assume for
simplicity that all schemes and formal schemes under consideration are quasi-compact
and separated.

%%%%%%%%%%%%%%%%%%%%%%%%%%%%%%%%%%%%%%%%%%%%%%%%%%%%%%%%%%%%%%%%
\medskip
\section{Tor-independence properties of the de~Rham-Witt complex}\label{Torind}
\smallskip
%%%%%%%%%%%%%%%%%%%%%%%%%%%%%%%%%%%%%%%%%%%%%%%%%%%%%%%%%%%%%%%%

Let $Y$ be a smooth $k$-scheme. We prove in this section Tor-independence results
between the sheaves of de Rham-Witt differential forms $\WOm{n}{j}{Y}$ and the
evaluation on the Witt thickenings $(Y,W_nY)$ of a crystal on $Y/W_n$ that is flat
over $W_n$ (see Definition \ref{ssFlatCryst}).

%%%%%%%%%%%%%%%%%%%%%%%%%%%%%%%%
\subsection{}\label{ssCrys}
%%%%%%%%%%%%%%%%%%%%%%%%%%%%%%%%
Let $\sE$ be a crystal on $Y/W_n$ for some $n \geq 1$ (\resp a crystal on $Y/W$).
Recall that:

\alphab For each divided powers thickening $U \inj T$, where $U \subset Y$ is an open
subset and $T$ a $W_n$-scheme (\resp a $W_n$-scheme for some $n$), $\sE$ defines an
$\sO_T$-module $\sE_T$, which we call the \textit{evaluation of} $\sE$ \textit{on} $T$. 

\alphab For each morphism of thickenings $(U',T') \to (U,T)$, defined by a 
$W_n$-PD-morphism $v : T' \to T$, the transition morphism 
%%%%%%%%%%%%%%%%
\eq{fIsoTrans}{ v^*\sE_T \to \sE_{T'}, }
%%%%%%%%%%%%%%%%
defined by the structure of $\sE$ as a sheaf of $\sO_{Y/W_n}$-modules (resp.\ 
$\sO_{Y/W}$-modules) on the crystalline site, is an isomorphism.

We will consider more specifically two families of thickenings. The first one is the 
family of thickenings $(U,U_n)$, where $U \subset Y$ is an open subset and $U_n$ is a 
smooth lifting of $U$ over $W_n$ (which always exist when $U$ is affine). When $U = Y$ 
and a smooth lifting $Y_n$ has been given, we will simply denote by $\sE_n$ the 
$\sO_{Y_n}$-module $\sE_{Y_n}$. From the crystal structure of $\sE$, $\sE_n$ inherits
an integrable connection $\nabla_n : \sE_n \to \sE_n \otimes \Omega^1_{Y_n/W_n}$, 
which is \textit{quasi-nilpotent} \cite[Def.\ 4.10]{BO}. This construction is
functorial in $\sE$, and defines an equivalence between the category of crystals on
$Y/W_n$ and the category of $\sO_{Y_n}$-modules endowed with an integrable and
quasi-nilpotent connection relative to $W_n$.

If $\sE$ is a crystal on $Y/W$ and $Y$ is lifted as a smooth formal 
scheme $\Y$ over $W$, with reduction $Y_n$ over $W_n$, $\sE$ can be viewed as a
compatible family of crystals on $Y/W_n$ for all $n$. Thus, $\sE$ defines for all $n$
an $\sO_{Y_n}$-module $\sE_n$ endowed with an integrable connection $\nabla_n$ such
that $\sE_n \simeq \sE_{n+1}/p^n\sE_{n+1}$ as a module with connection. Then the
$\sO_{\Y}$-module $\sEh = \varprojlim_n \sE_n$ has a connection $\nabla = \varprojlim_n
\nabla_n$, and one gets in this way an equivalence of categories between the category
of crystals on $Y/W$ and the category of $p$-adically separated and complete
$\sO_\Y$-modules endowed with an integrable and topologically quasi-nilpotent
connection.

The second family of thickenings we are going to use is provided by the immersions $Y
\inj W_nY := (|Y|, W_n\sO_Y)$, which are divided powers thickenings thanks to the
canonical divided powers of the ideal $VW_{n-1}\sO_Y$ (defined by $(Vx)^{[i]} =
(p^{i-1}/i!)V(x^i)$ for all $i \geq 1$). Thus, evaluating $\sE$ on $(Y,W_nY)$ defines a
$W_n\sO_Y$-module, which will be denoted by $\sE_n^W$. For $n' \leq n$, the closed
immersion $W_{n'}Y \inj W_nY$ defines a morphism of PD-thickenings of $Y$, hence the
crystal structure of $\sE$ provides a homomorphism $\sE^W_n \to \sE^W_{n'}$, the 
linear factorization of which is an isomorphism
%%%%%%%%%%%%%%%%
\eq{fTransEWn}{ W_{n'}\sO_Y \otimes_{W_n\sO_Y} \sE^W_n \riso \sE^W_{n'}. }
%%%%%%%%%%%%%%%%
If $\sE$ is a crystal on $Y/W$, these homomorphisms turn the family of
$W_n\sO_Y$-modules $(\sE^W_n)_{n\geq 1}$ into an inverse system of $W_n\sO_Y$-modules.

%%%%%%%%%%%%%%%%%%%%%%%%%%%%%%%%
\subsection{}\label{ssFroblift}
%%%%%%%%%%%%%%%%%%%%%%%%%%%%%%%%
Let $\sigma : W \to W$ be the Frobenius automorphism of $W$. We now assume that we are
given a smooth formal scheme $\Y$ lifting $Y$ over $W$, with reduction $Y_n$
over $W_n$. We assume in addition that $\Y$ is endowed with a $\sigma$-semi-linear
morphism $F : \Y \to \Y$ lifting the absolute Frobenius endomorphism of $Y$. As
$\sO_\Y$ is $p$-torsion free, the homomorphism $F : \sO_{\Y} \to \sO_{\Y}$ defines a
section $s_F : \sO_{\Y} \to W\sO_{\Y}$ of the reduction homomorphism $W\sO_{\Y} \to
\sO_{\Y}$, characterized by the fact that $w_i(s_F(x)) = F^i(x)$ for any $x \in
\sO_{\Y}$ and any ghost component $w_i$ \cite[0, 1.3]{Il79}. Composing $s_F$ with the
reduction homomorphisms and factorizing, we get for all $n \geq 1$ a homomorphism
%%%%%%%%%%%%%%%%
\eq{fDeftF}{ t_F : \sO_{Y_n} \to W_n\sO_Y. } 
%%%%%%%%%%%%%%%%
These homomorphisms are compatible with the reduction maps when $n$ varies, and
functorial in an obvious way with respect to the couple $(\Y,F)$.

The morphism $t_F$ is a PD-morphism, because the canonical divided powers of 
$VW_{n-1}\sO_Y$ extend the natural divided powers of $(p)$. Therefore, it defines a 
morphism of thickenings of $Y$, and, for any crystal $\sE$ on $Y/W_n$, we get 
a canonical isomorphism 
%%%%%%%%%%%%%%%%
\eq{fEW}{ W_n\sO_Y \otimes_{\sO_{Y_n}} \sE_n \riso \sE^W_n, }
%%%%%%%%%%%%%%%%
where the scalar extension is taken by means of the homomorphism $t_F$.

\medskip
We will need the following condition on $\sE$:

%%%%%%%%%%%%%%%%%%%%%%%%%%%%%%%%
\begin{defn}\label{ssFlatCryst}
Let $Y$ be a smooth $k$-scheme.\\
\romain If $\sE$ is a crystal on $Y/W_n$, we say that $\sE$ is \textit{flat over $W_n$}
if, for any smooth lifting $U_n$ over $W_n$ of an open subset $U \subset Y$, the
evaluation $\sE_{U_n}$ of $\sE$ on $U_n$ is flat over $W_n$. \\
\romain If $\sE$ is a crystal on $Y/W$, we say that $\sE$ is \textit{flat over $W$} if,
for all $n$, the induced crystal on $Y/W_n$ is flat over $W_n$.
\end{defn}
%%%%%%%%%%%%%%%%%%%%%%%%%%%%%%%%

For any open subset $U \subset Y$, two smooth liftings $U_n, U_n\!\!'$ of $U$ over
$W_n$ are locally isomorphic, and such local isomorphisms extend canonically to the
evaluations on $U_n$ and $U_n\!\!'$ of a crystal. Therefore, $\sE$ is flat
over $W_n$ if and only if there exists an open covering $U_{\alpha}$ of $Y$ and, for
all $\alpha$, a smooth lifting $U_{\alpha,n}$ of $U_{\alpha}$ over $W_n$ such that the
evaluation $\sE_{\alpha,n}$ of $\sE$ on $U_{\alpha,n}$ is flat over $W_n$.

Similarly, when $\sE$ is a crystal on $Y/W$, $\sE$ is flat over $W$ if and only
if there exists an open covering $U_{\alpha}$ of $Y$ and, for all $\alpha$, a lifting
$\U_{\alpha}$ of $U_{\alpha}$ as a smooth formal scheme over $W$ such that the
corresponding Zariski sheaf $\sEh_{\alpha} = \varprojlim_n \sE_{\alpha,n}$ on
$\U_{\alpha}$ is $p$-torsion free.

If $\sE$ is flat over $W_n$ (\resp $W$), then its 
restriction to $\Cris(Y/W_{i})$ is flat over $W_{i}$ for any $i \leq 
n$ (\resp any $i$). If $\sE$ is flat as an $\sO_{Y/W_n}$-module (\resp
as an $\sO_{Y/W}$-module), then $\sE$ is flat over $W_n$ (\resp over 
$W$).
\medskip

We now begin our study of the Tor-independence properties between the sheaves
$\WOm{n}{j}{Y}$ and the evaluation of a crystal that is flat over $W_n$. We recall
first from \cite[I, 3.1]{Il79} that, for $i \in \Z$, the $i$-th step of the canonical
filtration of $\WOm{n}{j}{Y}$ (resp.\ $\WOm{}{j}{Y}$) is the sub-$W_n\sO_Y$-module
(resp.\ sub-$W\sO_Y$-module) defined by
%%%%%%%%%%%%%%%%
\eqn{ \Fil^i\WOm{n}{j}{Y} = \begin{cases}
\WOm{n}{j}{Y} & \text{if $i \leq 0$,} \\
\Ker(\WOm{n}{j}{Y} \to \WOm{i}{j}{Y}) & \text{if $1 \leq i \leq n$,} \\
0 & \text{if $i > n$}
\end{cases} }
%%%%%%%%%%%%%%%%
\eqn{ \Big(\text{resp.\ } \Fil^i\WOm{}{j}{Y} = \begin{cases}
\WOm{}{j}{Y} & \text{if $i \leq 0$,} \\
\Ker(\WOm{}{j}{Y} \to \WOm{i}{j}{Y}) & \text{if $i \geq 1$}
\end{cases}\,\Big). }
%%%%%%%%%%%%%%%%
We denote  
%%%%%%%%%%%%%%%%
\eqna{ \gr^i\WOm{n}{j}{Y} & = & \Fil^i\WOm{n}{j}{Y}/\Fil^{i+1}\WOm{n}{j}{Y}, \\
\gr^i\WOm{}{j}{Y} & = & \Fil^i\WOm{}{j}{Y}/\Fil^{i+1}\WOm{}{j}{Y}. }
%%%%%%%%%%%%%%%%
By construction, the canonical homomorphism 
%%%%%%%%%%%%%%%%
\eqn{ \gr^i\WOm{}{j}{Y} \lra \gr^i\WOm{n}{j}{Y} }
%%%%%%%%%%%%%%%%
is an isomorphism for $i < n$. Note that, under the assumptions of \ref{ssFroblift},
the homomorphism $t_F$ defined in \eqref{fDeftF} endows the sheaves
$\Fil^i\WOm{r}{j}{Y}$ and $\gr^i\WOm{r}{j}{Y}$ with an $\sO_{Y_n}$-module structure for
$1 \leq r \leq n$ and all $i$, $j$.
\medskip

The following lemma allows to relate the reduction modulo $p$ of $t_F$ with the usual 
Frobenius endomorphism.

%%%%%%%%%%%%%%%%%%%%%%%%%%%%%%%%
\begin{lem}\label{ssRedtF}
Let $A$ be a commutative ring without $p$-torsion, $F : A \to A$ a ring homomorphism 
lifting the absolute Frobenius endomrphism of $A/pA$, and $s_F : A \to W(A)$ the ring 
homomorphism such that $w_i(s_F(x)) = F^i(x)$ for all $x \in A$ and all $i \geq 0$ 
\cite[0, 1.3]{Il79}. 
Let $n \geq 1$ be an integer, $A_n = A/p^nA$, $t_F : A_n \to W_n(A_1)$ the 
factorization of $s_F$ as in \eqref{fDeftF}, and $\ot_F : A_1 \to W_n(A_1)/pW_n(A_1)$ the 
reduction modulo $p$ of $t_F$. Denote by 
%%%%%%%%%%%%%%%%
\eq{fDefoF}{ \oF : A_1 \simeq W_n(A_1)/VW_{n-1}(A_1) \to W_n(A_1)/pW_n(A_1) }
%%%%%%%%%%%%%%%%
the homomorphism induced by the action of $F$ on $W_n(A_1)/pW_n(A_1)$. Then $\oF$ is 
equal to the composition 
%%%%%%%%%%%%%%%%
\eq{}{ A_1 \xra{F} A_1 \xra{\ot_F} W_n(A_1)/pW_n(A_1), }
%%%%%%%%%%%%%%%%
where the first homomorphism $F$ is the absolute Frobenius endomorphism of $A_1$. 
\end{lem}
%%%%%%%%%%%%%%%%%%%%%%%%%%%%%%%%

\begin{proof}
It suffices to prove the equality of the morphisms obtained by composing $\oF$ and 
$\ot_F \circ F$ with the surjection $A \surj A_1$. 
By construction, there is a commutative diagram 
%%%%%%%%%%%%%%%%
\eqn{ \xymatrix{
& A \ar[r]^-{F} \ar[d]^-{s_F} \ar[ld]_-{\Id} & A \ar@{>>}[r] \ar[d]^-{s_F} & A_1
\ar[d]^-{\ot_F} \\
A \ar@{>>}[dr] & W(A) \ar@{>>}[l] \ar[r]^-{F} \ar@{>>}[d] & W(A) \ar@{>>}[r]
\ar@{>>}[d] & W_n(A_1)/pW_n(A_1) \\
& A_1 \ar[r]^-{\oF} & W_n(A_1)/pW_n(A_1) \ar@{=}[ur] & 
} }
%%%%%%%%%%%%%%%%
in which the unlabelled arrows are the canonical surjections. Note that the commutativity of 
the left upper square (in which the upper $F$ is the given homomorphism and the lower
one the canonical Frobenius endomorphism of $W(A)$) is a consequence of the definition
of $s_F$ \cite[0, 1.3]{Il79}. The lemma follows from the equality of the exterior
paths of the diagram.
\end{proof}

%%%%%%%%%%%%%%%%%%%%%%%%%%%%%%%%
\begin{lem}\label{ssTorind1}
Let $Y$ be a smooth $k$-scheme, such that there exists a smooth formal scheme $\Y$
over $W$ lifting $Y$ and a semi-linear endomorphism $F$ of $\Y$ lifting the absolute
Frobenius endomorphism of $Y$. For some $n \geq 1$, let $\sE$ be a crystal on $Y/W_n$,
flat over $W_n$. Let $Y_n$ be the reduction of $\Y$ over $W_n$ and $\sE_n$ the
evaluation of $\sE$ on $Y_n$. Then, for the $\sO_{Y_n}$-module structure defined on
$W_n\sO_Y$-modules by $t_F$,
%%%%%%%%%%%%%%%%
\eq{fTorind1}{ \tor_q^{\sO_{Y_n}}(\sE_n,\Fil^i\WOm{r}{j}{Y}) = 
\tor_q^{\sO_{Y_n}}(\sE_n,\gr^i\WOm{r}{j}{Y}) = 0 }
%%%%%%%%%%%%%%%%
for $q \geq 1$, $1 \leq r \leq n$, and all $i$, $j$.
\end{lem}
%%%%%%%%%%%%%%%%%%%%%%%%%%%%%%%%

In particular, when $W_n\sO_Y$ is viewed as an $\sO_{Y_n}$-algebra through
$t_F$, $\sE_n$ and $W_n\sO_Y$ are Tor-independent over $\sO_{Y_n}$.

\begin{proof}
The canonical filtration of $\WOm{r}{i}{Y}$ is discrete and codiscrete, so it suffices
to prove the vanishing of $\tor_q^{\sO_{Y_n}}(\sE_n,\gr^i\WOm{r}{j}{Y})$. Moreover, we
may assume that $r = n$, since $\gr^i\WOm{r}{j}{Y} = 0$ for $i < 0$ or $i \geq r$, and
the natural map $\gr^i\WOm{n}{j}{Y} \to \gr^i\WOm{r}{j}{Y}$ is an isomorphism for $0 \leq i
< r \leq n$.

As $\gr^i\WOm{n}{j}{Y}$ is annihilated by $p$, its $\sO_{Y_n}$-module
structure is also given by the composition 
$\sO_{Y_n}\to\sO_Y \xrightarrow{\overline{t}_F} W_n\sO_Y/pW_n\sO_Y$,
where $\overline{t}_F$ is the reduction of $t_F$ mod $p$. Therefore, 
we obtain isomorphisms
%%%%%%%%%%%%%%%%
$$ \sE_n \otimesL_{\sO_{Y_n}} \gr^i\WOm{n}{j}{Y}\ \simeq\ 
(\sE_n \otimesL_{\sO_{Y_n}} \sO_Y) \otimesL_{\sO_Y} \gr^i\WOm{n}{j}{Y}.$$ 
%%%%%%%%%%%%%%%%
Because $\sE_n$ and $\sO_{Y_n}$ are flat relatively to $W_n$, we have
%%%%%%%%%%%%%%%%
$$ \sE_n \otimesL_{\sO_{Y_n}} \sO_Y\ \simeq\ \sE_n 
\otimesL_{\sO_{Y_n}} (\sO_{Y_n} \otimesL_{W_n} k)\ \simeq\ 
\sE_n \otimesL_{W_n} k
\ \xrightarrow{\ \sim\ \,}\ \sE_1,$$
%%%%%%%%%%%%%%%%
so it suffices to prove that
%%%%%%%%%%%%%%%%
$$ \tor_q^{\sO_Y}(\sE_1,\gr^i\WOm{n}{j}{Y}) = 0 $$
%%%%%%%%%%%%%%%%
for $q \geq 1$ and all $i$, $j$.

Since $\sE_1$ is the evaluation on $Y$ of a crystal, its connection $\nabla_1$
is quasi-nilpotent \cite[Def.\ 4.10]{BO}. Let $\sT_Y = \Omega_Y^{1\,\vee}$ 
be the sheaf of $k$-derivations on $Y$ and $\psi :
\sT_Y \to \End_{\sO_Y}(\sE_1)$ the $p$-curvature of
$\nabla_1$ \cite[5.0]{Ka70}. We define an increasing 
filtration of $\sE_1$ by horizontal submodules $\sE_1^m$, $m \geq 0$, by setting, 
for any affine open subset $U \subset Y$, 
%%%%%%%%%%%%%%%%
\ml{}{\Gamma(U, \sE_1^m) = \\
\{x \in \Gamma(U, \sE)\ |\ \forall 
D_1,\ldots,D_m \in \Gamma(U, \sT_Y), \psi(D_1)\cdots\psi(D_m)(x)=0\}}
%%%%%%%%%%%%%%%%
\cite[5.5]{Ka70}. As $\nabla_1$ is quasi-nilpotent, this
filtration is exhaustive, and it suffices to prove that, for $q \geq 1$, $m \geq 0$, 
and all $i$, $j$,
%%%%%%%%%%%%%%%%
$$ \tor_q^{\sO_Y}(\sE_1^m,\gr^i\WOm{n}{j}{Y}) = 0. $$
%%%%%%%%%%%%%%%%
By construction, $\sE_1^0 = 0$ and, for $m \geq 1$, each quotient
$\sE_1^m/\sE_1^{m-1}$ has $p$-curvature $0$. Therefore, for all $m$, there is an
$\sO_Y$-module $\sF^m$ such that $\sE_1^m/\sE_1^{m-1} \simeq F^{\ast}\sF^m$. Since $F$
is flat, it is enough to prove that, for all $q \geq 1$, $m \geq 0$, and all $i$, $j$,
%%%%%%%%%%%%%%%%
$$ \tor_q^{\sO_Y}(\sF^m,\gr^i\WOm{n}{j}{Y}) = 0, $$
%%%%%%%%%%%%%%%%
$\gr^i\WOm{n}{j}{Y}$ being now viewed as an $\sO_Y$-module thanks to 
the composition 
%%%%%%%%%%%%%%%%
$$ \sO_Y \xra{\ F\ } \sO_Y \xra{\ \overline{t}_F\ } 
W_n\sO_Y/pW_n\sO_Y.$$ 
%%%%%%%%%%%%%%%%
By Lemma \ref{ssRedtF}, this homomorphism is the factorization 
%%%%%%%%%%%%%%%%
\eqn{ \oF : \sO_Y \lra W_n\sO_Y/pW_n\sO_Y  }
%%%%%%%%%%%%%%%%
of the Frobenius action on $W_n\sO_Y/pW_n\sO_Y$. As $\gr^i\WOm{n}{j}{Y}$ is a locally
free $\sO_Y$-module for this structure \cite[I, 3.9]{Il79}, the lemma follows.
\end{proof}

%%%%%%%%%%%%%%%%%%%%%%%%%%%%%%%%
\begin{thm}\label{ssTorind2}
Let $Y$ be a smooth $k$-scheme. For some $n \geq 1$, let $\sE$ be a crystal on $Y/W_n$,
flat over $W_n$, and let $\sE^W_n$ be the evaluation of $\sE$ on the PD-thickening
$(Y,W_nY)$. Then
%%%%%%%%%%%%%%%%
\eq{fTorind2}{ \tor_q^{W_n\sO_Y}(\sE^W_n,\Fil^i\WOm{r}{j}{Y}) = 
\tor_q^{W_n\sO_Y}(\sE^W_n,\gr^i\WOm{r}{j}{Y}) = 0 }
%%%%%%%%%%%%%%%%
for $q \geq 1$, $1 \leq r \leq n$, and all $i$, $j$.
\end{thm}
%%%%%%%%%%%%%%%%%%%%%%%%%%%%%%%%

\begin{proof}
The statement is local on $Y$, hence we may assume that $Y$ is affine and has a 
smooth formal lifting $\Y$ over $W$ endowed with a lifting $F$ of the absolute 
Frobenius endomorphism of $Y$. Let $Y_n$ be the reduction of $\Y$ on $W_n$ and 
$\sE_n$ the evaluation of $\sE$ on $Y_n$. By \eqref{fEW}, $t_F$ provides an 
isomorphism 
%%%%%%%%%%%%%%%%
\eqn{ \sE_n \otimes_{\sO_{Y_n}} W_n\sO_Y \riso \sE^W_n. }
%%%%%%%%%%%%%%%%
Applying \eqref{fTorind1} for $i = j = 0$ and $r=n$, we obtain in $D^-(W_n\sO_Y)$ an
isomorphism
%%%%%%%%%%%%%%%%
\eq{fDerEWn}{ \sE_n \otimesL_{\sO_{Y_n}} W_n\sO_Y \riso \sE^W_n. }
%%%%%%%%%%%%%%%%
On the other hand, the transitivity isomorphism for the derived extension of scalars
yields
%%%%%%%%%%%%%%%%
\eqna{ \sE_n  \otimesL_{\sO_{Y_n}} \Fil^i\WOm{r}{j}{Y}& \riso & 
(\sE_n \otimesL_{\sO_{Y_n}} W_n\sO_Y) \otimesL_{W_n\sO_Y} \Fil^i\WOm{r}{j}{Y} \\
& \riso & \sE^W_n \otimesL_{W_n\sO_Y} \Fil^i\WOm{r}{j}{Y}. }
%%%%%%%%%%%%%%%%
By \eqref{fTorind1}, the left hand side is acyclic in degrees $\neq 0$. The first 
vanishing of \eqref{fTorind2} follows, and the second one is obtained by the same 
argument. 
\end{proof}
\medskip

We define the \textit{canonical filtration} of $\sE_n^W \otimes_{W_n\sO_Y}
\WOm{n}{j}{Y}$ by 
%%%%%%%%%%%%%%%%
\ml{fDefFilCanE}{ \Fil^i(\sE_n^W \otimes_{W_n\sO_Y} \WOm{n}{j}{Y}) = \\
\begin{cases}
\sE_n^W \otimes_{W_n\sO_Y} \WOm{n}{j}{Y} & \text{if $i \leq 0$,} \\
\Ker(\sE_n^W \otimes_{W_n\sO_Y} \WOm{n}{j}{Y} \to \sE_i^W \otimes_{W_i\sO_Y}
\WOm{i}{j}{Y}) & \text{if $1 \leq i \leq n$,} \\
0 & \text{if $i > n$}
\end{cases} }
%%%%%%%%%%%%%%%%
Using \eqref{fTransEWn}, the theorem implies:

%%%%%%%%%%%%%%%%%%%%%%%%%%%%%%%%
\begin{cor}\label{ssFilCanEnW}
%%%%%%%%%%%%%%%%%%%%%%%%%%%%%%%%
Under the assumptions of the theorem, the natural map 
%%%%%%%%%%%%%%%%
\eq{fFilCanEnW}{ \sE_n^W \otimes_{W_n\sO_Y} \Fil^i\WOm{n}{j}{Y} \lra 
\Fil^i(\sE_n^W \otimes_{W_n\sO_Y} \WOm{n}{j}{Y})  }
%%%%%%%%%%%%%%%%
is an isomorphism for all $i$. 
%%%%%%%%%%%%%%%%
\end{cor}
%%%%%%%%%%%%%%%%

%%%%%%%%%%%%%%%%%%%%%%%%%%%%%%%%
\subsection{}\label{ssupinj}
%%%%%%%%%%%%%%%%%%%%%%%%%%%%%%%%
Let $Y$ be a smooth $k$-scheme, and assume that $n \geq 2$. If $\sE$ is a crystal on
$Y/W_n$, multiplication by $p$ on $\sE^W_n \otimes \WOm{n}{j}{Y}$ vanishes on the image
of $\sE^W_n\otimes\Fil^{n-1}\WOm{n}{j}{Y}$. Factorizing and taking the isomorphism
\eqref{fTransEWn} into account, one gets a canonical homomorphism
%%%%%%%%%%%%%%%%
\eq{fDefup}{ \up : \sE^W_{n-1} \otimes_{W_{n-1}\sO_Y} \WOm{n-1}{j}{Y} \to 
\sE^W_n \otimes_{W_n\sO_Y} \WOm{n}{j}{Y}. }
%%%%%%%%%%%%%%%%

We recall that, when $\sE = \sO_{Y/W_n}$, the morphism $\up$ is an injection 
%%%%%%%%%%%%%%%%
\eqn{ \up : \WOm{n-1}{j}{Y} \inj \WOm{n}{j}{Y}  }
%%%%%%%%%%%%%%%%
for all $j$ \cite[Prop.\ 3.4]{Il79}. It follows that, for 
$i \geq 1$, there is an exact sequence
%%%%%%%%%%%%%%%%
\eq{fpseq}{ 0 \to \WOm{n-1}{j}{Y}/p^{i-1}\WOm{n-1}{j}{Y} \xra{\up} 
\WOm{n}{j}{Y}/p^i\WOm{n}{j}{Y} \to \WOm{n}{j}{Y}/p\WOm{n}{j}{Y} \to 0. }
%%%%%%%%%%%%%%%%

We will show later that, when $\sE$ is flat over $W_n$, the homomorphism \eqref{fDefup}
is still injective (see Corollary \ref{ssupinj2}). This is a consequence of another
Tor-independence property, which we prove next.

%%%%%%%%%%%%%%%%%%%%%%%%%%%%%%%%
\begin{lem}\label{ssLoclib}
Let $Y$ be a smooth $k$-scheme. For $n \geq 1$ and all $j$, the sheaf
$\WOm{n}{j}{Y}/p\WOm{n}{j}{Y}$ is a locally free $\sO_Y$-module of finite rank for the
structure defined by the homomorphism $\oF : \sO_Y \to W_n\sO_Y/pW_n\sO_Y$
\eqref{fDefoF}.
\end{lem}
%%%%%%%%%%%%%%%%%%%%%%%%%%%%%%%%

\begin{proof}
We proceed by induction on $n$, the claim being clear for $n = 1$. We set 
$\WOm{n}{j}{Y} = 0$ for $n = 0$. For $n \geq 1$, the
commutative diagram with exact rows
%%%%%%%%%%%%%%%%
\eqn{ \xymatrix{
0 \ar[r] & \gr^{n-1}\WOm{n}{j}{Y} \ar[r] \ar@{^{(}->}[d]_-{\up} & 
\WOm{n}{j}{Y} \ar[r] \ar@{^{(}->}[d]_-{\up} & 
\WOm{n-1}{j}{Y} \ar[r] \ar@{^{(}->}[d]_-{\up} & 0 \\
0 \ar[r] & \gr^n\WOm{n+1}{j}{Y} \ar[r] & \WOm{n+1}{j}{Y} \ar[r] & 
\WOm{n}{j}{Y} \ar[r] & 0
} }
%%%%%%%%%%%%%%%%
yields an exact sequence 
%%%%%%%%%%%%%%%%
\eqn{ \xymatrix{ 
0 \ar[r] & \gr^n\WOm{n+1}{j}{Y}/\up\,\gr^{n-1}\WOm{n}{j}{Y} \ar[d] & \\
& \WOm{n+1}{j}{Y}/\up\WOm{n}{j}{Y} \ar[r] & \WOm{n}{j}{Y}/\up\WOm{n-1}{j}{Y} \ar[r] & 0
 } }
%%%%%%%%%%%%%%%%
in which the morphisms are $\sO_Y$-linear for the module structure defined by $\oF$. As
$\up\WOm{n-1}{j}{Y} = p\WOm{n}{j}{Y}$ for all $n \geq 1$, the induction hypothesis
reduces to proving that $\gr^n\WOm{n+1}{j}{Y}/\up\,\gr^{n-1}\WOm{n}{j}{Y}$ is a locally
free finitely generated $\sO_Y$-module for the structure defined by $\oF$. By \cite[I,
(3.10.4)]{Il79}, there is a commutative diagram
%%%%%%%%%%%%%%%%
\eqn{ \xymatrix{
0 \ar[r] & \Omega^j_Y/B_{n-1}\Omega^j_Y \ar[r]^-{V^{n-1}} \ar@{^{(}->}[d]_-{C^{-1}} & 
\gr^{n-1}\WOm{n}{j}{Y} \ar[r] \ar@{^{(}->}[d]_-{\up} & 
\Omega^{j-1}_Y/Z_{n-1}\Omega^{j-1}_Y \ar[r] \ar@{^{(}->}[d]_-{C^{-1}} & 0 \\
0 \ar[r] & \Omega^j_Y/B_n\Omega^j_Y \ar[r]^-{V^n} \ar[r]^-{V^r} \ar[r] & 
\gr^n\WOm{n+1}{j}{Y} \ar[r] & \Omega^{j-1}_Y/Z_n\Omega^{j-1}_Y \ar[r] & 0
} }
%%%%%%%%%%%%%%%%
in which the exterior columns are defined by the inverse Cartier operator $C^{-1}$ and
we set $B_0\Omega^j_Y = 0$, $Z_0\Omega^j_Y = \Omega^j_Y$. All maps in this diagram
become $\sO_Y$-linear maps if we endow the terms of the middle column with the
structure defined by $\oF$, the exterior terms of the upper row with the structure
defined by $F^n$, and the exterior terms of the lower row with the structure defined by
$F^{n+1}$ (see \cite[I, Cor.\ 3.9]{Il79}). Therefore, the corresponding cokernel
sequence, which can be written as
%%%%%%%%%%%%%%%%
\eqn{ 0 \to B\Omega^{j+1}_Y \to
\gr^n\WOm{n+1}{j}{Y}/\up\,\gr^{n-1}\WOm{n}{j}{Y} \to B\Omega^j_Y \to 0,}
%%%%%%%%%%%%%%%%
is an $\sO_Y$-linear exact sequence when $B\Omega^j_Y$ and $B\Omega^{j+1}_Y$ are 
viewed as $\sO_Y$-modules thanks to $F^{n+1}$. By \cite[0, Prop.~2.2.8]{Il79}, they
are then locally free of finite rank over $\sO_Y$, which ends the proof.
\end{proof}

%%%%%%%%%%%%%%%%%%%%%%%%%%%%%%%%
\begin{thm}\label{ssTorind3}
Let $Y$ be a smooth $k$-scheme, and let $\sE$ be a crystal on $Y/W_n$, 
flat over $W_n$. Let $\sE^W_n$ be the evaluation of $\sE$ on $(Y,W_nY)$. Then 
%%%%%%%%%%%%%%%%
\eq{fTorind3}{  \tor_q^{W_n\sO_Y}(\sE^W_n,\WOm{r}{j}{Y}/p^i\WOm{r}{j}{Y}) = 
\tor_q^{W_n\sO_Y}(\sE^W_n,p^i\WOm{r}{j}{Y}) = 0 }
%%%%%%%%%%%%%%%%
for $q \geq 1$, $1 \leq r \leq n$, $i \geq 0$, and all $j$. 
\end{thm}
%%%%%%%%%%%%%%%%%%%%%%%%%%%%%%%%

\begin{proof}
Thanks to Theorem \ref{ssTorind2}, it suffices to prove the vanishing of the left hand 
side. 

The statement is trivial for $i = 0$. Let us prove it first for $i=1$. We may assume 
that $Y$ is affine, and we can choose a smooth formal scheme $\Y$ lifting $Y$ over 
$W$, together with a lifting $F$ of the absolute Frobenius endomorphism de $Y$. Then 
the isomorphism \eqref{fDerEWn} and the transitivity of the derived exension of scalars 
show as in the proof of Theorem \ref{ssTorind2} that it is equivalent to prove the 
relation 
%%%%%%%%%%%%%%%%
\eq{fRedTorind1}{ \tor_q^{\sO_{Y_n}}(\sE_n,\WOm{r}{j}{Y}/p\WOm{r}{j}{Y}) = 0, }
%%%%%%%%%%%%%%%%
where $\WOm{r}{j}{Y}/p\WOm{r}{j}{Y}$ is viewed as an $\sO_{Y_n}$-module thanks to 
$t_F$. As this module is annihilated by $p$, the flatness of $\sE_n$ relative to $W_n$ 
implies as in the proof of Lemma \ref{ssTorind1} that it is equivalent to prove that 
%%%%%%%%%%%%%%%%
\eq{fRedTorind2}{ \tor_q^{\sO_Y}(\sE_1,\WOm{r}{j}{Y}/p\WOm{r}{j}{Y}) = 0, }
%%%%%%%%%%%%%%%%
where $\sE_1 = \sE_n/p\sE_n$ and $\WOm{r}{j}{Y}/p\WOm{r}{j}{Y}$ is viewed as an
$\sO_Y$-module thanks to the reduction $\overline{t}_F : \sO_Y \to W_r\sO_Y/pW_r\sO_Y$ 
of $t_F$. Repeating again the proof of Lemma \ref{ssTorind1}, we can use the 
$p$-curvature filtration of $\sE_1$ to reduce \eqref{fRedTorind2} to the relation 
%%%%%%%%%%%%%%%%
\eq{fRedTorind3}{ \tor_q^{\sO_Y}(\sF,\WOm{r}{j}{Y}/p\WOm{r}{j}{Y}) = 0, }
%%%%%%%%%%%%%%%%
where $\sF$ is an $\sO_Y$-module and $\WOm{r}{j}{Y}/p\WOm{r}{j}{Y}$ is viewed as an
$\sO_Y$-module through the composition $\overline{t}_F \circ F$. As the latter is equal
to $\oF$ by Lemma \ref{ssRedtF}, relation \eqref{fRedTorind3} is then a consequence of
Lemma \ref{ssLoclib}.

This proves that, for $i = 1$,
$\tor_q^{W_n\sO_Y}(\sE^W_n, \WOm{r}{j}{Y}/p^i\WOm{r}{j}{Y})$ vanishes for $q \geq
1$, $1 \leq r \leq n$, and all $j$. Moreover, by Theorem \ref{ssTorind2}, it also
vanishes when $r = 1$ for $q \geq 1$, $i \geq 0$, and all $j$. Using the exact sequences
\eqref{fpseq} and the previous result for $i = 1$, one can then argue by induction on
$r$ to prove that the same vanishing holds for $1 \leq r \leq n$ and all $i \geq 0$.
\end{proof}

%%%%%%%%%%%%%%%%%%%%%%%%%%%%%%%%
\begin{cor}\label{ssupinj2}
Let $Y$ be a smooth $k$-scheme, and let $\sE$ be a crystal on $Y/W_n$ for some $n \geq
2$. If $\sE$ is flat over $W_n$, the homomorphism \eqref{fDefup} is an injection
%%%%%%%%%%%%%%%%
\eqn{ \up : \sE^W_{n-1} \otimes_{W_{n-1}\sO_Y} \WOm{n-1}{j}{Y} \inj 
\sE^W_n \otimes_{W_n\sO_Y} \WOm{n}{j}{Y}. }
%%%%%%%%%%%%%%%%
\end{cor}
%%%%%%%%%%%%%%%%%%%%%%%%%%%%%%%%

\begin{proof}
Since $\up : \WOm{n-1}{j}{Y} \to \WOm{n}{j}{Y}$ is injective, it follows from the 
theorem that 
%%%%%%%%%%%%%%%%
\eqn{ \Id \otimes \up : \sE^W_n \otimes_{W_n\sO_Y} \WOm{n-1}{j}{Y} \to 
\sE^W_n \otimes_{W_n\sO_Y} \WOm{n}{j}{Y} }
%%%%%%%%%%%%%%%%
is injective too. Using \eqref{fTransEWn}, we get an isomorphism 
%%%%%%%%%%%%%%%%
\eqn{ \sE^W_n \otimes_{W_n\sO_Y} \WOm{n-1}{j}{Y} \riso 
\sE^W_{n-1} \otimes_{W_{n-1}\sO_Y} \WOm{n-1}{j}{Y}, }
%%%%%%%%%%%%%%%%
which identifies $\Id \otimes \up$ to the homomorphism $\up$ defined by 
\eqref{fDefup} and completes the proof.
\end{proof}

%%%%%%%%%%%%%%%%%%%%%%%%%%%%%%%%
\subsection{}\label{ssComplTens}
%%%%%%%%%%%%%%%%%%%%%%%%%%%%%%%%
We now assume that $\sE$ is a crystal on $Y/W$, flat over $W$, and we consider the
inverse system of $W_n\sO_Y$-modules $\sE^W_n$ defined in \ref{ssCrys} by taking the
evaluation of $\sE$ at all thickenings $(Y,W_nY)$ when $n$ varies. We also assume that
$\sE$ is a \textit{quasi-coherent crystal}, i.e., that, for any PD-thickening $(U,T)$
in $\Cris(Y/W)$, the evaluation $\sE_T$ of $\sE$ on $T$ is a quasi-coherent
$\sO_T$-module. It is equivalent to ask that, for any $n \geq 1$, this condition be
verified for thickenings of the form $(U_\alpha,U_{\alpha,n})$, where $U_\alpha$ varies
in an open covering of $Y$ and $U_{\alpha,n}$ is a smooth lifting of $U_\alpha$ on
$W_n$. Thanks to \cite[I, Prop.\ 1.13.1]{Il79}, the $W_n\sO_Y$-module $\sE^W_n
\otimes_{W_n\sO_Y} \WOm{n}{j}{Y}$ is then quasi-coherent for all $n \geq 1$ and all
$j$.

For all $j$, we define 
%%%%%%%%%%%%%%%%
\eqn{ \EWOm{\sE}{j}{Y} := \varprojlim_n\,(\sE^W_n \otimes_{W_n\sO_Y} 
\WOm{n}{j}{Y}). }
%%%%%%%%%%%%%%%%
By construction, we have projections 
%%%%%%%%%%%%%%%%
\eqn{ \EWOm{\sE}{j}{Y} \to \sE^W_i \otimes_{W_i\sO_Y} \WOm{i}{j}{Y} }
%%%%%%%%%%%%%%%%
for each $i \geq 1$, and, since the inverse system has surjective transition maps and 
quasi-coherent terms, the Mittag-Leffler criterion implies that these projections are
surjective. We define the \textit{canonical filtration} of $\EWOm{\sE}{j}{Y}$ by
%%%%%%%%%%%%%%%%
\mln{ \Fil^i (\EWOm{\sE}{j}{Y}) = \\
\begin{cases}
\sEh\; \otimesh_{W\sO_Y} \WOm{}{j}{Y} & \text{if $i \leq 0$}, \\
\Ker( \EWOm{\sE}{j}{Y} \to \sE^W_i \otimes_{W_i\sO_Y} \WOm{i}{j}{Y}) & \text{if $i \geq
1$}.
\end{cases} }
%%%%%%%%%%%%%%%%
Note that, for $i \geq 1$,
%%%%%%%%%%%%%%%%
\mln{ \Fil^i (\EWOm{\sE}{j}{Y}) \riso \\
\varprojlim_n \Ker(\sE^W_n \otimes_{W_n\sO_Y} \WOm{n}{j}{Y} \to 
\sE^W_i \otimes_{W_i\sO_Y} \WOm{i}{j}{Y}), }
%%%%%%%%%%%%%%%%
hence the Mittage-Leffler criterion implies that, for any affine open subset $U \subset
Y$,
%%%%%%%%%%%%%%%%
\eq{fVanCoh}{ H^q(U, \Fil^i(\EWOm{\sE}{j}{Y})) = 0 }
%%%%%%%%%%%%%%%%
for $q \geq 1$ and all $i$, $j$. 

For each affine open subset $U \subset Y$, the canonical filtration endows 
$\Gamma(U, \EWOm{\sE}{j}{Y})$ with a topology, which will be called the \textit{canonical 
topology}. From \eqref{fVanCoh}, we deduce the isomorphism  
%%%%%%%%%%%%%%%%
\eq{fQuotcan}{ \Gamma(U, \EWOm{\sE}{j}{Y}) / \Fil^i\Gamma(U, \EWOm{\sE}{j}{Y}) \riso
\Gamma(U, \sE^W_i \otimes_{\W_i\sO_Y} \WOm{i}{j}{Y}) }
%%%%%%%%%%%%%%%%
for $i \geq 1$. It follows that $\Gamma(U, \EWOm{\sE}{j}{Y})$ is separated and complete
for the canonical topology.

%%%%%%%%%%%%%%%%%%%%%%%%%%%%%%%%
\begin{prop}\label{sspCompl}
Let $Y$ be a smooth $k$-scheme, and let $\sE$ be a crystal on $Y/W$. Assume that $\sE$
is flat over $W$ and quasi-coherent.

\romain For all $j$, multiplication by $p$ is injective on $\EWOm{\sE}{j}{Y}$. 

\romain If $U \subset Y$ is an affine open subset, then, for all $i \geq 0$ and all
$j$, $p^i\Gamma(U, \EWOm{\sE}{j}{Y})$ is closed in $\Gamma(U, \EWOm{\sE}{j}{Y})$ for 
the canonical topology, and $\Gamma(U, \EWOm{\sE}{j}{Y})$ is separated and complete 
for the $p$-adic topology.
\end{prop}
%%%%%%%%%%%%%%%%%%%%%%%%%%%%%%%%

\begin{proof}
As multiplication by $p$ on $\EWOm{\sE}{j}{Y}$ is the inverse limit of the maps $\up$ 
defined in \eqref{fDefup}, assertion (i) results from Corollary \ref{ssupinj2}.

Let $U$ be an affine open subset of $Y$, and let $i \in \N$. Using assertion (i) and
Corollary \ref{ssupinj2}, we can write a commutative diagram with exact columns
%%%%%%%%%%%%%%%%
\eqn{ \xymatrix@R=15pt{
0 \ar[d] & 0 \ar[d] \\
\Gamma(U,\EWOm{\sE}{j}{Y}) \ar[d]^(.43){p^i} \ar[r]^-{\sim} & 
\varprojlim_{n>i} \Gamma(U, \sE^W_{n-i} \otimes_{W_{n-i}\sO_Y} \WOm{n-i}{j}{Y}) 
\ar[d]^(.43){\up^i} \\
\Gamma(U,\EWOm{\sE}{j}{Y}) \ar[d] \ar[r]^-{\sim} & 
\varprojlim_{n>i} \Gamma(U, \sE^W_n \otimes_{W_n\sO_Y} \WOm{n}{j}{Y}) \ar[d] \\
\frac{\DS \Gamma(U,\EWOm{\sE}{j}{Y})}{\DS
p^i\Gamma(U,\EWOm{\sE}{j}{Y})} \ar[r] \ar[d] & 
\varprojlim_{n>i}\frac{\DS \Gamma(U, \sE^W_n\otimes_{W_n\sO_Y} \WOm{n}{j}{Y})}{\DS
\up^i\Gamma(U,\sE^W_{n-i}\otimes_{W_{n-i}\sO_Y} \WOm{n-1}{j}{Y})} \ar[d] \\
 0 & 0 
} }
%%%%%%%%%%%%%%%%
in which the surjectivity in the right hand side column results from the surjectivity
of the transition maps in the inverse system $(\Gamma(U, \sE^W_{n-i}
\otimes_{W_{n-i}\sO_Y} \WOm{n-i}{j}{Y}))_{n>i}$. It follows that the bottom horizontal arrow 
is an isomorphism. The quasi-coherence assumption on $\sE$ implies that 
%%%%%%%%%%%%%%%%
\eqn{ \up^i\Gamma(U,\sE^W_{n-i}\otimes_{W_{n-i}\sO_Y} \WOm{n-1}{j}{Y})) =
p^i\Gamma(U,\sE^W_n\otimes_{W_n\sO_Y} \WOm{n}{j}{Y}) }
%%%%%%%%%%%%%%%%
and that the map \eqref{fQuotcan}
%%%%%%%%%%%%%%%%
\eqn{ \Gamma(U, \EWOm{\sE}{j}{Y}) \to \Gamma(U,\sE^W_n\otimes_{W_n\sO_Y}
\WOm{n}{j}{Y}) }
%%%%%%%%%%%%%%%%
is surjective. Therefore the quotient in the right hand side column can
be rewritten as
%%%%%%%%%%%%%%%%
\eqn{ \varprojlim_{n>i}\,\frac{\DS \Gamma(U, \EWOm{\sE}{j}{Y})}{\DS
p^i\Gamma(U, \EWOm{\sE}{j}{Y})+ \Fil^n\Gamma(U, \EWOm{\sE}{j}{Y})}. }
%%%%%%%%%%%%%%%%
This proves the first part of assertion (ii). As $\Gamma(U, \EWOm{\sE}{j}{Y})$ is
separated and complete for the canonical topology, which is coarser than the $p$-adic
topology, the second part of assertion (ii) follows by \cite[Ch.~III, \S~3, \num
5, Cor.~2 to Prop.~10]{Bou}.
\end{proof}

%%%%%%%%%%%%%%%%%%%%%%%%%%%%%%%%%%%%%%%%%%%%%%%%%%%%%%%%%%%%%%%%
\medskip
\section{De Rham-Witt cohomology with coefficients}\label{dRWCoeff}
\smallskip
%%%%%%%%%%%%%%%%%%%%%%%%%%%%%%%%%%%%%%%%%%%%%%%%%%%%%%%%%%%%%%%%

In this section, we extend the classical comparison theorem between the crystalline and
de Rham-Witt cohomologies of a smooth $k$-scheme $Y$ to the case of cohomologies with
coefficients in a crystal on $Y/W_n$ that is flat over $W_n$. 

As indicated in our general conventions, we now assume until the end of the article
that all schemes are quasi-compact and separated. We first recall the construction of
the comparison morphism in the case of constant coefficients \cite[II, 1]{Il79}, 
starting at the level of complexes of Zariski sheaves on $Y$.

%%%%%%%%%%%%%%%%%%%%%%%%%%%%%%%%
\subsection{}\label{ssCohCris}
%%%%%%%%%%%%%%%%%%%%%%%%%%%%%%%%
Let $\P$ be a smooth formal scheme over $W$, with reduction $P_n$ over $W_n$ and
special fibre $P = P_1$. Let $X \subset P$ be a closed subscheme. We denote by $\sJ_n$
(\resp $\sJ$) the ideal of $X$ in $P_n$ (\resp $\P$) and by $\sP_n$ (\resp $\sPh$) the
divided power envelope of $\sJ_n$ (\resp the $p$-adic completion of the divided power
envelope of $\sJ$). In these constructions, we impose that all divided powers be
compatible with the natural divided powers of $p$, which implies that $\sP_n \simeq
\sP_{n+1}/p^n\sP_{n+1}$ for all $n$ and $\sPh = \varprojlim_n \sP_n$. Divided power
envelopes have a natural connection, which allows to define the de Rham complexes
$\sP_n \otimes \Omd{P_n}$ and $\sPh \otimes \Omd{\P}$; these complexes are supported in
$X$. Let $u_{X/W_n}$ (resp.\ $u_{X/W}$) be the projection from the crystalline topos of
$X$ relative to $W_n$ (resp.\ $W$) to the Zariski topos of $X$. In its local form, the
comparison theorem between crystalline and de Rham cohomologies \cite[(7.1.2)]{BO}
provides functorial isomorphisms
%%%%%%%%%%%%%%%%
\ga{fCrisdRLocn}{ \RR u_{X/W_n*}\,\sO_{X/W_n} \simeq \sP_n \otimes \Omd{P_n/W_n},  \\
\RR u_{X/W*}\,\sO_{X/W} \simeq \sPh \otimes \Omd{\P/W}. \label{fCrisdRLoc} }
%%%%%%%%%%%%%%%%
in $\Db(X,W_n)$ and $\Db(X,W)$. Taking sections on $X$, one gets in $\Db(W_n)$ and 
$\Db(W)$ the global comparison isomorphisms 
%%%%%%%%%%%%%%%%
\ga{fCrisdRn}{ \RR\Gamma\cris(X/W_n, \sO_{X/W_n}) \simeq 
\RR\Gamma(X, \sP_n \otimes \Omd{P_n/W_n}), \\
\RR\Gamma\cris(X/W, \sO_{X/W}) \simeq 
\RR\Gamma(X, \sPh \otimes \Omd{\P/W}). \label{fCrisdR}}
%%%%%%%%%%%%%%%%

These isomorphisms can be generalized to the case where the datum of the embedding $X
\inj \P$ is replaced by the data of an affine open covering $\fU = (U_{\alpha})$ of
$X$ and of closed embeddings $U_{\alpha} \inj \P_{\alpha}$ into smooth formal schemes 
\cite[0, 3.2.6]{Il79}. For each multi-index $\ualpha = (\alpha_0,\ldots,\alpha_i)$,
let $U_{\ualpha} = U_{\alpha_0} \cap \cdots \cap U_{\alpha_i}$, $\P_{\ualpha} =
\P_{\alpha_0} \times_W \cdots \times_W \P_{\alpha_i}$, let $\sJ_{\ualpha} \subset
\sO_{\P_{\ualpha}}$ be the ideal defining the diagonal embeddings $U_{\ualpha} \inj
\P_{\ualpha}$, $\sPh_{\ualpha}$ its completed divided power envelope, $P_{\ualpha,n}$, 
$\sJ_{\ualpha,n}$, $\sP_{\ualpha,n}$ the reductions mod $p^n$ of $\P_{\ualpha}$, 
$\sJ_{\ualpha}$, $\sP_{\ualpha}$, and let $j_{\ualpha} : U_{\ualpha} \inj X$ be the 
inclusion. One can define \v{C}ech double complexes $j_{\sbul\,*}(\sP_{\sbul,n} \otimes 
\Omega\hbul_{P_{\sbul,n}})$ (resp.\ $j_{\sbul\,*}(\sPh_{\sbul} \otimes 
\Omega\hbul_{\P_{\sbul}})$) with general term
%%%%%%%%%%%%%%%%
\eqn{ \prod_{\ualpha = (\alpha_0,\ldots,\alpha_i)} j_{\ualpha\,*}(\sP_{\ualpha,n} 
\otimes \Omega^j_{P_{\ualpha,n}}) \quad\quad 
(\text{resp.\ }
\prod_{\ualpha = (\alpha_0,\ldots,\alpha_i)} j_{\ualpha\,*}(\sPh_{\ualpha} 
\otimes \Omega^j_{\P_{\ualpha}})
) }
%%%%%%%%%%%%%%%%
in bidegree $(j,i)$. If one uses the subscript ``t'' to denote the total complex
associated to a multicomplex, one can then generalize the isomorphisms
\eqref{fCrisdRLocn} and \eqref{fCrisdRLoc} as
%%%%%%%%%%%%%%%%
\ga{fCrisdRLocn2}{ \RR u_{X/W_n*}\,\sO_{X/W_n} \simeq (j_{\sbul\,*}(\sP_{\sbul,n} \otimes 
\Omega\hbul_{P_{\sbul,n}}))\tot,  \\
\RR u_{X/W*}\,\sO_{X/W} \simeq (j_{\sbul\,*}(\sPh_{\sbul} \otimes 
\Omega\hbul_{\P_{\sbul}}))\tot, \label{fCrisdRLoc2} }
%%%%%%%%%%%%%%%%
and one gets similar genaralizations of \eqref{fCrisdRn} and \eqref{fCrisdR}.

%%%%%%%%%%%%%%%%%%%%%%%%%%%%%%%%
\begin{defn}\label{ssMaptoWn}
Let $\P$ be a formal scheme over $W$, endowed with a $\sigma$-semi\-linear morphism $F
: \P \to \P$ lifting the absolute Frobenius endomorphism of its special fibre $P$, and
let $X \subset P$ be a closed subscheme. For any $n \geq 1$, let $t_F : \sO_{P_n} \to
W_n\sO_P$ be the homomorphism \eqref{fDeftF} defined by $F$. Then the composition
%%%%%%%%%%%%%%%%
\eq{fComptF}{ \sO_{P_n} \xra{t_F} W_n\sO_P \to W_n\sO_X }
%%%%%%%%%%%%%%%%
maps $\sJ_n$ to $VW_{n-1}\sO_X \subset W_n\sO_X$. Using the natural divided powers of
the ideal $VW_{n-1}\sO_X$ (which are compatible with the divided powers of $p$), the
universal property of divided power envelopes provides a unique factorization of this
composition through a homomorphism denoted 
%%%%%%%%%%%%%%%%
\eq{fDefhF}{h_F : \sP_n \to W_n\sO_X,}
%%%%%%%%%%%%%%%%
which commutes with the divided powers. 
\end{defn}
%%%%%%%%%%%%%%%%%%%%%%%%%%%%%%%%

%%%%%%%%%%%%%%%%%%%%%%%%%%%%%%%%
\subsection{}\label{ssConstCase} 
%%%%%%%%%%%%%%%%%%%%%%%%%%%%%%%%
Let $Y$ be a smooth $k$-scheme, embedded through a closed immersion $Y \inj P$ into 
the special fibre of a smooth formal scheme $\P$ over $W$. Assume that $\P$ 
is endowed with a Frobenius lifting $F$ as in \ref{ssMaptoWn}, and keep for $Y$ the 
notation introduced for $X$ in \ref{ssCohCris}-\ref{ssMaptoWn}. Then the homomorphism
\eqref{fComptF} extends as a morphism of complexes
%%%%%%%%%%%%%%%%
\eq{fComptF2}{ \Omd{P_n} \to \WOmd{n}{P} \to \WOmd{n}{Y}. }
%%%%%%%%%%%%%%%%
Thanks to the structure of graded $\sP_n$-algebra defined by $h_F$ on $\WOmd{n}{Y}$,
this morphism defines by extension of scalars from $\sO_{P_n}$ to $\sP_n$ an
$h_F$-semi-linear morphism
%%%%%%%%%%%%%%%%
\eq{dRtodRWn}{\sP_n \otimes \Omd{P_n} \to \WOmd{n}{Y}, }
%%%%%%%%%%%%%%%%
which is still a morphism of complexes. When $n$ varies, these morphisms are
compatible, and their inverse limit gives a morphism of complexes
%%%%%%%%%%%%%%%%
\eq{dRtodRW}{\sPh \otimes \Omd{\P} \to \WOmd{}{Y}.}
%%%%%%%%%%%%%%%%
The morphisms \eqref{dRtodRWn} and \eqref{dRtodRW} are functorial with respect to the 
triple $(Y, \P, F)$ in the following sense. Let $\P'$ be a second smooth formal
$W$-scheme, with special fibre $P'$, let $Y' \subset P'$ be a smooth closed subscheme,
and let $F' : \P' \to \P'$ be a lifting of the Frobenius endomorphism of $P'$. If $u :
\P' \to \P$ is a $W$-morphism commuting with $F$ and $F'$, and inducing a $k$-morphism 
$f : Y' \to Y$, then we get commutative diagrams
%%%%%%%%%%%%%%%%
\ga{cstfunctn}{ \xymatrix@C=60pt{
f^{-1}(\sP_n \otimes \Omd{P_n}) \ar[r]^-{f^{-1}(\eqref{dRtodRWn})} \ar[d] & 
f^{-1}(\WOmd{n}{Y}) \ar[d] \\
\sP'_n \otimes \Omd{P'_n} \ar[r]-<22pt,0pt>^-{\eqref{dRtodRWn}} & 
\hspace{18pt}\WOmd{n}{Y'}\hspace{12pt},} \\
\xymatrix@C=60pt{
f^{-1}(\sPh \otimes \Omd{\P}) \ar[r]^-{f^{-1}(\eqref{dRtodRW})} \ar[d] & 
f^{-1}(\WOmd{}{Y}) \ar[d] \\
\sPh' \otimes \Omd{\P'} \ar[r]-<20pt,0pt>^-{\eqref{dRtodRW}} & 
\hspace{18pt}\WOmd{}{Y'}\hspace{12pt},} \label{cstfunct}}
%%%%%%%%%%%%%%%%
where the vertical maps are the functoriality morphisms. 

The morphisms \eqref{dRtodRWn} and \eqref{dRtodRW} are quasi-isomorphisms, and, via 
\eqref{fCrisdRLocn} and \eqref{fCrisdRLoc}, they define the comparison isomorphism
between crystalline and de Rham-Witt cohomologies \cite[II, 1.4]{Il79}. In the derived
category $\Db(Y,W_n)$ (\resp $\Db(Y,W)$) of sheaves of $W_n$-modules (\resp
$W$-modules) on $Y$, this comparison isomorphism does not depend on the choice of
$(\P,F)$ (by the standard argument comparing two embeddings $Y \inj (\P,F)$ and $Y \inj
(\P',F')$ to the diagonal embedding into $(\P \times_W \P', F \times F')$ via the
projection maps).

If $Y$ is quasi-projective, one can always find such an embedding $(\P,F)$, since it
suffices to choose for $\P$ an open subscheme of a projective space, endowed with the
morphism induced by some lifting of the Frobenius endomorphism of the projective space.
In the general case, one can choose an affine open covering $\fU = (U_{\alpha})$ of
$Y$ and closed immersions $U_{\alpha} \inj \P_{\alpha}$ into smooth formal schemes
endowed with a lifting of Frobenius. As in \ref{ssCohCris}, one can then identify $\RR
u_{Y/W_n\,*} \sO_{Y/W_n}$ with the total complex associated to the \v{C}ech double
complex $j_{\sbul\,*}(\sP_{\sbul,n} \otimes \Omega\hbul_{P_{\sbul,n}})$ (resp.\ $\RR
u_{Y/W\,*} \sO_{Y/W}$, $j_{\sbul\,*}(\sPh_{\sbul} \otimes \Omega\hbul_{\P_{\sbul}})$).
On the other hand, the total complex associated to the \v{C}ech double complex
$\check{C}\hbul(\fU, \WOmd{n}{Y})$ (resp.~$\check{C}\hbul(\fU, \WOmd{}{Y})$) provides a
resolution of $\WOmd{n}{Y}$ (resp.~$\WOmd{}{Y}$). Endowing each $\P_{\ualpha}$ with 
the product $F_{\alpha_0}\times \cdots \times F_{\alpha_i}$ and using the functoriality
\eqref{cstfunctn} (resp.~\eqref{cstfunct}), one gets between these double complexes a
morphism that is defined on each intersection $U_{\ualpha}$ by the corresponding
morphism $\eqref{dRtodRWn}$ (\resp \eqref{dRtodRW}). Since each of these is a
quasi-isomorphism, so is the morphism induced between the associated total complexes.
In the derived categories $D(Y, W_n)$ and $D(Y, W)$, this provides isomorphisms
%%%%%%%%%%%%%%%%
\ga{dRtodRWn2}{ \RR u_{Y/W_n\,*} \sO_{Y/W_n} \riso \WOmd{n}{Y}, \\
\RR u_{Y/W\,*} \sO_{Y/W} \riso \WOmd{}{Y}, }
%%%%%%%%%%%%%%%%
which extend the comparison isomorphisms \eqref{dRtodRWn} and \eqref{dRtodRW} to the general
case and do not depend on the choices (see \cite[II, 1.1]{Il79} for details).

%%%%%%%%%%%%%%%%%%%%%%%%%%%%%%%%
\subsection{}\label{ssCrys2}
%%%%%%%%%%%%%%%%%%%%%%%%%%%%%%%%
Let us keep the hypothesis of the existence of an embedding $Y \inj \P$ of $Y$ into a
smooth formal scheme, and let $\sE$ be a crystal on $Y/W_n$ for some $n \geq
1$ (\resp a crystal on $Y/W$).

We now denote by $\sE_n$ the evaluation of $\sE$ on the thickening $(Y, \Spec(\sP_n))$.
By construction, $\sE_n$ is a $\sP_n$-module, and, by restriction of scalars, it can
also be viewed as an $\sO_{P_n}$-module. As such, it inherits from the crystal
structure of $\sE$ an integrable connection $\nabla_n : \sE_n \to \sE_n \otimes
\Omega^1_{P_n/W_n}$, compatible with the natural connection on $\sP_n$. This connection
allows to define the de Rham complex $\sE_n \otimes \Omd{P_n/W_n}$. When $\sE$ is a
crystal on $Y/W$, it can be viewed as a compatible family of crystals on $Y/W_n$ for
all $n$. Thus, $\sE$ defines for all $n$ a $\sP_n$-module $\sE_n$ endowed with an
integrable connection $\nabla_n$, such that $\sE_n \simeq \sE_{n+1}/p^n\sE_{n+1}$ as a
module with connection. Then the $\sPh$-module $\sEh = \varprojlim_n \sE_n$ has a
connection $\nabla = \varprojlim_n \nabla_n$, and we get similarly the de Rham complex
$\sEh \otimes \Omd{\P/W}$. We recall that the datum of $(\sE_n, \nabla_n)$ (\resp
$(\sEh,\nabla)$) is equivalent to the datum of the crystal $\sE$.

As in the case of cohomology with constant coefficients, the comparison theorem between
crystalline and de Rham cohomologies yields functorial isomorphisms \cite[Th.~7.1]{BO}
%%%%%%%%%%%%%%%%
\ga{fustardRn}{ \RR u_{Y/W_n*}\,\sE \simeq \sE_n \otimes \Omd{P_n/W_n},  \\
\RR\Gamma\cris(Y/W_n,\sE) \simeq \RR\Gamma(Y, \sE_n
\otimes \Omd{P_n/W_n}), \label{fisocrysdRn}}
%%%%%%%%%%%%%%%%
and, when $\sE$ is a quasi-coherent crystal on $Y/W$ \cite[Th.~7.23]{BO},
%%%%%%%%%%%%%%%%
\ga{fustardR}{ \RR u_{Y/W*}\,\sE \simeq \sEh \otimes \Omd{\P/W}, \\
\RR\Gamma\cris(Y/W,\sE) \simeq \RR\Gamma(Y, \sEh \otimes 
\Omd{\P/W}). \label{fisocrysdR}}
%%%%%%%%%%%%%%%%
If $\fU =(U_{\alpha})$ is an affine open covering of $Y$ and $(U_{\alpha} \inj 
\P_{\alpha})$ a family of closed immersions into smooth formal schemes over 
$W$, these isomorphisms can be generalized by replacing $\sE_n \otimes \Omd{P_n/W_n}$ 
and $\sEh \otimes \Omd{\P/W}$ with the total complexes $(j_{\sbul\,*}(\sE_{P_{\sbul,n}} \otimes 
\Omd{P_{\sbul,n}/W_n}))\tot$ and $(j_{\sbul\,*}(\sEh_{\P_{\sbul}} \otimes 
\Omd{\P_{\sbul}/W}))\tot$ defined as in \ref{ssCohCris}.

%%%%%%%%%%%%%%%%%%%%%%%%%%%%%%%%
\subsection{}\label{ssMorphisms}
%%%%%%%%%%%%%%%%%%%%%%%%%%%%%%%%
We keep the notation $\sE^W_n$ used in \ref{ssCrys} for the evaluation of $\sE$ on the
thickening $(Y,W_nY)$. As shown by \'Etesse \cite[II, 2]{Et88}, the crystal structure
of $\sE$ also provides a differential $\nabla_n^W : \sE_n^W \to \sE_n^W
\otimes_{W_n\sO_Y} \WOm{n}{1}{Y}$, which extends so as to define a complex $\sE_n^W
\otimes_{W_n\sO_Y} \WOmd{n}{Y}$.

Let us assume again that $\P$ is endowed with a lifting $F$ of the absolute Frobenius
endomorphism of $P$. Then the PD-morphism $h_F : \sP_n \to W_n\sO_Y$ defined in
\eqref{fDefhF} corresponds to a morphism of thickenings $W_nY \to \Spec(\sP_n)$, and
the crystal structure of $\sE$ provides a natural $h_F$-semi-linear map
%%%%%%%%%%%%%%%%
\eqn{ h_{\sE,F} : \sE_n \lra \sE_n^W, }
%%%%%%%%%%%%%%%%
the linear factorization of which gives an isomorphism 
%%%%%%%%%%%%%%%%
\eq{fEPW}{ 1 \otimes h_{\sE,F} : W_n\sO_Y \otimes_{\sP_n} \sE_n \riso \sE_n^W }
%%%%%%%%%%%%%%%%
(generalizing \eqref{fEW}, where $Y = P$). We can then take the tensor product over
$h_F : \sP_n \to W_n\sO_Y$ of the morphism $h_{\sE,F}$ with the semi-linear morphism of 
graded modules $\sP_n \otimes \Omd{P_n} \to \WOmd{n}{Y}$ defined in \eqref{dRtodRWn},
and this defines an $h_F$-semi-linear morphism of graded modules 
%%%%%%%%%%%%%%%%
\eq{dREtodRWEn}{\sE_n \otimes_{\sO_{P_n}} \Omd{P_n} \riso 
\sE_n \otimes_{\sP_n} (\sP_n \otimes_{\sO_{P_n}} \Omd{P_n}) \lra 
\sE_n^W \otimes_{W_n\sO_Y} \WOmd{n}{Y}. }
%%%%%%%%%%%%%%%%
It is easy to deduce from the constructions of $\nabla_n$ and $\nabla_n^W$ that this
morphism is actually a morphism of complexes. When $\sE$ is a crystal on $Y/W$, the
inverse limit of these morphisms when $n$ varies defines a morphism
%%%%%%%%%%%%%%%%
\eq{dREtodRWE}{\sEh \otimes_{\sO_{\P}} \Omd{\P} \lra \sEh^W 
\otimesh_{W\sO_Y} \WOmd{}{Y},}
%%%%%%%%%%%%%%%%
where the completed tensor product is defined as in \ref{ssComplTens} by
%%%%%%%%%%%%%%%%
\eqn{ \sEh^W \otimesh_{W\sO_Y} \WOmd{}{Y} = 
\varprojlim_n\,(\sE_n^W \otimes_{W_n\sO_Y} \WOmd{n}{Y}). }
%%%%%%%%%%%%%%%%

As in the constant case, these morphisms are functorial with respect to $(Y,\P,F)$.
Keeping the same notation and hypotheses than in \ref{ssConstCase}, let $\sE' =
f^{\ast}\cris(\sE)$ be the inverse image of $\sE$, which is the crystal on $Y'$
corresponding to the inverse image connection $(\sE'_n, \nabla'_n) =
(\tu_n^{\ast}(\sE_n),\tu_n^{\ast}(\nabla_n))$ on $P'_n$, $\tu_n$ denoting the morphism 
of ringed spaces $(P'_n,\sP'_n) \to (P_n,\sP_n)$ defined by $u$. Then the following
square is commutative:
%%%%%%%%%%%%%%%%
\eq{fonct}{ \xymatrix@C=60pt{
f^{-1}(\sE_n \otimes_{\sO_{P_n}} \Omd{P_n}) \ar[r]^-{f^{-1}(\eqref{dREtodRWEn})}
\ar[d] & f^{-1}(\sE_n^W \otimes_{W_n\sO_Y} \WOmd{n}{Y}) \ar[d] \\
\sE'_n \otimes_{\sO_{P'_n}} \Omd{P'_n} \ar[r]-<54pt,0pt>^-{\eqref{dREtodRWEn}} &
\hspace{17pt}\sE'_n{}^W \otimes_{W_n\sO_{Y'}} \WOmd{n}{Y'}\hspace{9pt}.
} }
%%%%%%%%%%%%%%%%
When $\sE$ is a crystal on $Y/W$, there is a similar diagram 
based on \eqref{dREtodRWE}. 

When $f$ is the identity on $Y$, the right vertical arrow of diagram \eqref{fonct} is
the identity, and the left vertical arrow is a quasi-isomorphism which corresponds in
the derived category to the identity on crystalline cohomology, via the canonical
identifications \eqref{fustardRn} between the crystalline cohomology of $\sE$ and the
de Rham cohomologies of $\sE_n$ and $\sE'_n$ on $P_n$ and $P'_n$. It follows that, in
$\Db(Y,W_n)$ and $\Db(Y,W)$, the morphisms \eqref{dREtodRWEn} and \eqref{dREtodRWE} do
not depend on the choice of $(\P,F)$ when they are viewed as morphisms from crystalline
cohomology to de Rham-Witt cohomology via \eqref{fustardRn} and \eqref{fustardR}.

Finally, using \v{C}ech complexes as in the case of cohomology with constant
coefficients, these functorialities also allow to extend to the general case (i.e.,
without assuming the existence of an embedding $Y \inj (\P, F)$ as above) the
construction of the morphisms \eqref{dREtodRWEn} and \eqref{dREtodRWE} as morphisms 
%%%%%%%%%%%%%%%%
\ga{fCrisdRWLocn}{ \RR u_{Y/W_n\,*}\,\sE \lra \sE_n^W \otimes_{W_n\sO_Y} \WOmd{n}{Y}, \\
\RR u_{Y/W*}\,\sE \lra \sEh^W \otimesh_{W\sO_Y} \WOmd{}{Y} \label{fCrisdRWLoc} }
%%%%%%%%%%%%%%%%
in the corresponding derived categories. Taking global sections, one obtains morphisms
%%%%%%%%%%%%%%%%
\ga{morcrisdRWE}{\RR\Gamma\cris(Y/W_n,\sE) \lra 
\RR\Gamma(Y, \sE_n^W \otimes_{W_n\sO_Y} \WOmd{n}{Y}), \\
\ \ \RR\Gamma\cris(Y/W,\sE) \lra \RR\Gamma(Y, \sEh^W \otimesh_{W\sO_Y}
\WOmd{}{Y})\ \ \label{morcrisdRWEh}}
%%%%%%%%%%%%%%%%
(as usual, we assume in \eqref{fCrisdRWLoc} and \eqref{morcrisdRWEh} that $\sE$ is
quasi-coherent).

\medskip
Our basic comparison theorem between crystalline and de Rham-Witt cohomologies 
with coefficients is then the following:

%%%%%%%%%%%%%%%%%%%%%%%%%%%%%%%%
\begin{thm}\label{ssCrysnEtodRWE}
Let $k$ be a perfect field of characteristic $p$, and let $Y$ be a smooth $k$-scheme.
For some $n \in \N$, let $\sE$ be a crystal on $Y/W_n$. If $\sE$ is flat over $W_n$,
the comparison morphisms \eqref{fCrisdRWLocn} and \eqref{morcrisdRWE} are isomorphisms 
respectively in $\Db(Y,W_n)$ and $\Db(W_n)$.
\end{thm}
%%%%%%%%%%%%%%%%%%%%%%%%%%%%%%%%

This generalizes earlier results of \'Etesse \cite[Th\'eor\`eme
2.1]{Et88} and Langer-Zink \cite[Theorem 3.8]{LZ04}, where $\sE$ is
assumed to be flat over $\sO_{Y/W_n}$.
  
The first assertion of the theorem implies the second one. To prove the first one, one 
can make further reductions. Indeed, taking an affine open covering $(U_\alpha)$ of
$Y$, choosing closed immersions of each $U_\alpha$ in a smooth formal scheme
$\P_\alpha$ endowed with a Frobenius lifting, and using the definition of
\eqref{fCrisdRWLocn} as a morphism of $\Db(Y,W_n)$ thanks to \v{C}ech resolutions, it
suffices to prove that, on each intersection $U_{\ualpha}$, the corresponding morphism
\eqref{dREtodRWEn} of the category of complexes is a quasi-isomorphism. So we are
reduced to proving that, when $Y$ can be embedded in a smooth formal scheme $\P$
endowed with a Frobenius lifting, the morphism of complexes \eqref{dREtodRWEn} is a
quasi-isomorphism. This is a local statement on $Y$. As this morphism does not depend
in the derived category on the choice of $(\P,F)$, we may localize and assume that
$Y$, together with its Frobenius endomorphism, can be lifted as a smooth formal scheme
$\Y$ over $W$, with reduction $Y_n$ over $W_n$. Thus we may assume that $\P = \Y$, and
it suffices to prove the following local form of the theorem:

%%%%%%%%%%%%%%%%%%%%%%%%%%%%%%%%
\begin{thm}\label{ssCrysnEtodRWEloc}
Let $k$ be a perfect field of characteristic $p$, and let $Y$ be a smooth $k$-scheme. 
Assume that there exists a smooth formal scheme $\Y$ over $W$ lifting $Y$, with
reduction $Y_n$ over $W_n$, and a semi-linear endomorphism $F$ of $\Y$ lifting the
absolute Frobenius endomorphism of $Y$. For some $n \in \N$, let $\sE$ be a crystal on
$Y/W_n$. If $\sE$ is flat over $W_n$, the morphism of complexes
%%%%%%%%%%%%%%%%
\eq{fCrysnEtodRWEloc}{{\sE_n \otimes_{\sO_{Y_n}} \Omd{Y_n} \to \sE_n^W 
\otimes_{W_n\sO_Y} \WOmd{n}{Y}}, } % {dRYEtodRWE}
%%%%%%%%%%%%%%%%
defined by \eqref{dREtodRWEn}, is a quasi-isomorphism.
\end{thm}
%%%%%%%%%%%%%%%%%%%%%%%%%%%%%%%%

\begin{proof}
If we endow the complex $\sE \otimes \Omd{Y_n}$ with the $p$-adic 
filtration and the complex $\sE_n^W \otimes_{W_n\sO_Y} \WOmd{n}{Y}$ with the canonical 
filtration defined in \eqref{fDefFilCanE}, the morphism \eqref{dREtodRWEn} becomes a
morphism of filtered complexes, and we are reduced to proving that the morphisms
%%%%%%%%%%%%%%%%
\eq{fgrCrysnEtodRWEloc}{p^i\sE\otimes_{\sO_{Y_n}}\Omd{Y_n} / 
p^{i+1}\sE\otimes_{\sO_{Y_n}}\Omd{Y_n} \to
\gr^i(\sE_n^W \otimes_{W_n\sO_Y} \WOmd{n}{Y}) } 
%%%%%%%%%%%%%%%%
are quasi-isomorphisms for all $i$.

As the morphism \eqref{dREtodRWEn} is $h_F$-semi-linear, the morphism
\eqref{fCrysnEtodRWEloc} is $t_F$-semi-linear. If we denote as before by $\ot_F$ the
reduction modulo $p$ of $t_F$, we can view $\gr^i\WOmd{n}{Y}$ and $\gr^i(\sE_n^W 
\otimes_{W_n\sO_Y} \WOmd{n}{Y})$ as $\sO_Y$-modules through $\ot_F$. The morphism 
\eqref{fgrCrysnEtodRWEloc} is then $\sO_Y$-linear. 

Using \eqref{fFilCanEnW}, \eqref{fEW}, and the fact that $\gr^i\WOmd{n}{Y}$ is
annihilated by $p$, we obtain natural linear isomorphisms of complexes
%%%%%%%%%%%%%%%%
$$ \sE_1\otimes_{\sO_Y}\gr^i\WOmd{n}{Y} \riso 
\gr^i(\sE_n^W\otimes_{W_n\sO_Y}\WOmd{n}{Y}). $$ 
%%%%%%%%%%%%%%%%
Multiplication by $p^i$ induces a linear morphism of complexes
%%%%%%%%%%%%%%%%
$$ \underline{p}^i : \sE_1\otimes_{\sO_Y}\gr^0\WOmd{n}{Y} \to 
\sE_1\otimes_{\sO_Y}\gr^i\WOmd{n}{Y}. $$ 
%%%%%%%%%%%%%%%%
On the other hand, since $\sE_n$ is flat over $W_n$, multiplication by $p^i$ induces a 
linear isomorphism of complexes
%%%%%%%%%%%%%%%%
$$ \underline{p}^i : \sE_1\otimes_{\sO_Y}\Omd{Y} \xrightarrow{\ \sim\ \,} 
p^i\sE\otimes_{\sO_{Y_n}}\Omd{Y_n} / 
p^{i+1}\sE\otimes_{\sO_{Y_n}}\Omd{Y_n}. $$ 
%%%%%%%%%%%%%%%%
As in \cite[II 1.4]{Il79}, these morphisms fit in commutative diagrams
%%%%%%%%%%%%%%%%
\eqn{ \xymatrix{
\sE_1\otimes_{\sO_Y}\Omd{Y} \ar[r]^-{\sim} \ar[d]^-{\wr}_-{\underline{p}^i} & 
\sE_1\otimes_{\sO_Y}\gr^0\WOmd{n}{Y} \ar[d]_-{\underline{p}^i}\\
p^i\sE\otimes_{\sO_{Y_n}}\Omd{Y_n} / p^{i+1}\sE\otimes_{\sO_{Y_n}}\Omd{Y_n} \ar[r] & 
\sE_1\otimes_{\sO_Y}\gr^i\WOmd{n}{Y},
} }
%%%%%%%%%%%%%%%%
in which the upper horizontal isomorphism follows from the identification between
$\gr^0\WOmd{n}{Y}$ and $\Omd{Y}$. Hence it suffices to prove that the morphisms
%%%%%%%%%%%%%%%%
$$ \underline{p}^i : \sE_1\otimes_{\sO_Y}\Omd{Y} \to 
\sE_1\otimes_{\sO_Y}\gr^i\WOmd{n}{Y}
$$ 
%%%%%%%%%%%%%%%%
are quasi-isomorphisms for all $i$.

As in the proof of Lemma \ref{ssTorind1}, we can endow $\sE_1$ with an exhaustive
filtration by horizontal submodules $\sE_1^m$ such that $\sE_1^m/\sE_1^{m-1}$ has
$p$-curvature $0$. Since the filtration is exhaustive, it suffices to prove that the
above morphism is a quasi-isomorphism when $\sE_1$ is replaced by $\sE_1^m$ for any
$m$, and it is even sufficient to prove it for the quotients $\sE_1^m/\sE_1^{m-1}$. As
their connections have $p$-curvature $0$, they are of the form $F^{\ast}\sF^m$ for some
$\sO_Y$-modules $\sF^m$, as $\sO_Y$-modules with connection. So it is enough to check
that the morphisms of complexes
%%%%%%%%%%%%%%%%
\eq{fPerfQis}{ \underline{p}^i : \sF^m\otimes_{\sO_Y}F_{\ast}\Omd{Y} \to 
\sF^m\otimes_{\sO_Y}F_{\ast}\gr^i\WOmd{n}{Y} } 
%%%%%%%%%%%%%%%%
are quasi-isomorphisms. From \cite[I 3.14]{Il79}, we know that the morphisms
$\underline{p}^i:\Omd{Y}\to\gr^i\WOmd{n}{Y}$ are quasi-isomorphisms, and both complexes
are linear complexes of locally free finitely generated $\sO_Y$-modules when $\sO_Y$
acts by $F$ on $\Omd{Y}$ and by $\oF :\sO_Y \to W_n\sO_Y/pW_n\sO_Y$ on
$\gr^i\WOmd{n}{Y}$ \cite[I, 3.9]{Il79}. By Lemma \ref{ssRedtF}, $\oF = \ot_F \circ F$,
so the $\sO_Y$-action on $F_{\ast}\gr^i\WOmd{n}{Y}$ in \eqref{fPerfQis} is the one
given by $\oF$. Moreover, the connection on $F^{\ast}\sF^m$ is simply
$\Id_{\sF^m}\otimes d$, where the tensor product is taken through $F$. Therefore these
quasi-isomorphisms remain quasi-isomorphisms after tensorizing with $\sF^m$.
\end{proof}

%%%%%%%%%%%%%%%%%%%%%%%%%%%%%%%%
\begin{cor}\label{ssCrysEtodRWE}
Let $k$ be a perfect field of characteristic $p$, and let $Y$ be a smooth $k$-scheme.
Let $\sE$ be a crystal on $Y/W$, flat over $W$ and such that the $\sO_Y$-module $\sE_1$
defined by $\sE$ is quasi-coherent.

\romain If there exists a smooth formal scheme $\Y$ over $W$ lifting $Y$ and a
semi-linear endomorphism $F : \Y \to \Y$ lifting the absolute Frobenius endomorphism of
$Y$, then the morphism of complexes
%%%%%%%%%%%%%%%%
\eq{fCrysEtodRWE}{{\sEh \otimes_{\sO_{\Y}} \Omd{\Y} \to \sEh^W 
\otimesh_{W\sO_Y} \WOmd{}{Y}},}
%%%%%%%%%%%%%%%%
defined by \eqref{dREtodRWE}, is a quasi-isomorphism.

\romain In the general case, the comparison morphisms \eqref{fCrisdRWLoc} and 
\eqref{morcrisdRWEh} are isomorphisms in $\Db(Y,W)$ and $\Db(W)$.
\end{cor}
%%%%%%%%%%%%%%%%%%%%%%%%%%%%%%%%

\begin{proof}
The quasi-coherence of $\sE_1$ and the flatness of $\sEh$ over $W$ imply that $\sE_n
\simeq \sEh/p^n\sEh$ is quasi-coherent for all $n$. Then the terms of the two complexes
of inverse systems $\sE\lbul \otimes \Omd{Y\lbul}$ and $\sE^W\lbul
\otimes_{W\lbul\sO_Y} \WOmd{\lbul}{Y}$ have vanishing cohomology on open affine subsets
and surjective transition maps. Therefore, these are complexes with
$\varprojlim$-acyclic terms, and we obtain equalities
%%%%%%%%%%%%%%%%
\begin{eqnarray*}
\RR\varprojlim (\sE\lbul \otimes \Omd{Y\lbul}) & = & 
\varprojlim (\sE\lbul \otimes \Omd{Y\lbul}),\\
\RR\varprojlim (\sE^W\lbul \otimes_{W\lbul\sO_Y} \WOmd{\lbul}{Y}) & = & 
\varprojlim (\sE^W\lbul \otimes_{W\lbul\sO_Y} \WOmd{\lbul}{Y}).
\end{eqnarray*}
%%%%%%%%%%%%%%%%
When $n$ varies, the quasi-isomorphisms \eqref{fCrysnEtodRWEloc} provide a
quasi-isomorphism of complexes of inverse systems. Applying the
$\RR\varprojlim$ functor, it follows that \eqref{fCrysEtodRWE}
is a quasi-isomorphism.

As in the proof of Theorem \ref{ssCrysnEtodRWE}, assertion (ii) follows from assertion 
(i) by taking an affine covering of $Y$, choosing closed immersions of the open subsets of 
the covering into smooth formal schemes endowed with Frobenius liftings, and 
representing the morphism \eqref{dREtodRWE} in $\Db(Y,W)$ thanks to the associated 
\v{C}ech resolutions.
\end{proof}

%%%%%%%%%%%%%%%%%%%%%%%%%%%%%%%%
\subsection{}\label{ssImDirCris}
%%%%%%%%%%%%%%%%%%%%%%%%%%%%%%%%
As an example of application of Theorem \ref{ssCrysnEtodRWE}, we now show that, for 
local complete intersections, crystalline cohomology can be computed using the de
Rham-Witt complexes of smooth embeddings. We first recall a few facts about crystalline
direct images by a closed immersion \cite[Prop.\ 6.2]{BO}.

Let $Y$ be a $k$-scheme, and let $i : X \inj Y$ be a closed immersion defined by an 
ideal $\sI \subset \sO_Y$. For any PD-thickening $(U,T,\delta)$ of an open subset $U
\subset Y$, we denote by $\sI_T$ the ideal of $X \cap U$ in $T$. Let $i\cris{}_{\,*}$
be the direct image functor from the category of sheaves on $\Cris(X/W_n)$ (resp.\
$\Cris(X/W)$) to the category of sheaves on $\Cris(Y/W_n)$ (resp.\ $\Cris(Y/W)$). Then:

\alphab The sheaf $i\cris{}_{\,*}\,\sO_{X/W_n}$ (resp.\ $i\cris{}_{\,*}\,\sO_{X/W}$) is 
a crystal in $\sO_{Y/W_n}$-algebras (resp.\ $\sO_{Y/W}$-algebras) supported in $X$. Its
evaluation on a PD-thickening $(U,T,\delta)$ is given by
%%%%%%%%%%%%%%%%
\eq{fImDirCrisT}{ (i\cris{}_{\,*}\,\sO_{X/W_n})_T  = \sP_\delta(\sI_T)}
%%%%%%%%%%%%%%%%
(resp.\ $(i\cris{}_{\,*}\,\sO_{X/W})_T$), where $\sP_\delta$ denotes the divided power
envelope with compatibility with $\delta$.

\alphab For all $q \geq 1$, $R^q i\cris{}_{\,*}\,\sO_{X/W_n} = 0$ (resp.\ $R^q
i\cris{}_{\,*}\,\sO_{X/W} = 0$).

%%%%%%%%%%%%%%%%%%%%%%%%%%%%%%%%
\begin{thm}\label{ssApplCris}
%%%%%%%%%%%%%%%%%%%%%%%%%%%%%%%%
Let $Y$ be a smooth $k$-scheme, and let $i : X \inj Y$ be a regular immersion defined
by an ideal $\sI \subset \sO_Y$. For each $n \geq 1$, set $W_n\sI = \Ker(W_n\sO_Y \surj
W_n\sO_X)$. Let $\sP_n^W$ be the PD-envelope of $W_n\sI$ with compatibility with the
canonical divided powers of $VW_{n-1}\sO_Y$, and let $\sPh^W = \varprojlim_n \sP_n^W$.
There exists functorial isomorphisms
%%%%%%%%%%%%%%%%
\ga{fApplCrisLocn}{ \RR u_{X/W_n\,*}\,\sO_{X/W_n} \riso \sP_n^W \otimes_{W_n\sO_Y} 
\WOmd{n}{Y}, \\
\RR u_{X/W\,*}\,\sO_{X/W}  \riso \sPh^W \otimesh_{W\sO_Y} \WOmd{}{Y}, 
\label{fApplCrisLoc} }
%%%%%%%%%%%%%%%%
respectively in $\Db(X,W_n)$ and $\Db(X, W)$, and 
%%%%%%%%%%%%%%%%
\ga{fApplCrisn}{ \RR\Gamma\cris(X/W_n, \sO_{X/W_n}) \riso \RR\Gamma(X, \sP_n^W \otimes_{W_n\sO_Y} 
\WOmd{n}{Y}), \\
\RR\Gamma\cris(X/W, \sO_{X/W}) \riso \RR\Gamma(X, \sPh^W \otimesh_{W\sO_Y} \WOmd{}{Y}), 
\label{fApplCris} }
%%%%%%%%%%%%%%%%
respectively in $\Db(W_n)$ and $\Db(W)$. 
%%%%%%%%%%%%%%%%
\end{thm}
%%%%%%%%%%%%%%%%

\begin{proof}
As $u_{X/W_n} = u_{Y/W_n} \circ i\cris$, we get from \ref{ssImDirCris} b) a 
canonical isomorphism 
%%%%%%%%%%%%%%%%
\eq{fXonY}{ \RR u_{X/W_n\,*}\,\sO_{X/W_n} \riso \RR u_{Y/W_n\,*}(i\cris{}_{\,*}\,\sO_{X/W_n}). }
%%%%%%%%%%%%%%%%
To define \eqref{fApplCrisLocn}, it suffices to check that the crystal
$i\cris{}_{\,*}\,\sO_{X/W_n}$ is flat over $W_n$ and that its evaluation on $W_nY$ is 
canonically isomorphic to $\sP_n^W$: Theorem \ref{ssCrysnEtodRWE} will then provide 
an isomorphism 
%%%%%%%%%%%%%%%%
\eq{fIsoPnW}{ \RR u_{Y/W_n\,*}(i\cris{}_{\,*}\,\sO_{X/W_n}) \riso \sP_n^W 
\otimes_{W_n\sO_Y} \WOmd{n}{Y}, }
%%%%%%%%%%%%%%%%
and we will obtain \eqref{fApplCrisLocn} as the composition of \eqref{fXonY} and 
\eqref{fIsoPnW}. 

To check the first claim, consider an affine open subset $U \subset Y$ on which $\sI$ is
defined by a regular sequence $(t_1,\ldots,t_d)$. Let $U_n$ be a smooth lifting
of $U$ over $W_n$ and $(\tt_1,\ldots,\tt_d)$ a sequence of sections of $\sO_{U_n}$ 
lifting $(t_1,\ldots,t_d)$. It is easy to check that $(\tt_1,\ldots,\tt_d)$ is a 
regular sequence in $\sO_{U_n}$ and that $\sO_{U_n}/(\tt_1,\ldots,\tt_d)$ is flat over 
$W_n$. The ideal of $X$ in $U_n$ is $\sI_n = (p, \tt_1,\ldots,\tt_d)$, and, by
\eqref{fImDirCrisT}, the evaluation of $i\cris{}_{\,*}\,\sO_{X/W}$ on $U_n$ is the
divided power envelope $\sP_n$ of $\sI_n$ with compatibility with the divided powers 
of $(p)$. This divided power envelope is also the divided power envelope of the ideal 
$(\tt_1,\ldots,\tt_d)$ with compatibility with the divided powers of $(p)$, and, by 
\cite[2.3.3 and 2.3.4]{BBM}, it is flat over $W_n$. Therefore, the crystal
$i\cris{}_{\,*}\,\sO_{X/W}$ is flat over $W_n$. 

On the other hand, \eqref{fImDirCrisT} shows that the evaluation of
$i\cris{}_{\,*}\,\sO_{X/W_n}$ on $W_nY$ is the PD-envelope $\sP_{\can}(\sI_{W_nY})$,
where the subscript $\can$ denotes compatibility with the canonical divided powers of
$VW_{n-1}\sO_Y$. As the ideal $\sI_{W_nY}$ can be written
%%%%%%%%%%%%%%%%
\eqn{ \sI_{W_nY} = W_n\sI + VW_{n-1}\sO_Y, }
%%%%%%%%%%%%%%%%
the second claim follows. 

To define the isomorphism \eqref{fApplCrisLoc}, one constructs the family of
isomorphisms \eqref{fApplCrisLocn} for variable $n$ so as to obtain an isomorphism in
the derived category of inverse systems, and one can then apply the $\RR\varprojlim_n$
functor. Taking global sections, one gets \eqref{fApplCrisn} and \eqref{fApplCris}.
\end{proof}

%%%%%%%%%%%%%%%%%%%%%%%%%%%%%%%%%%%%%%%%%%%%%%%%%%%%%%%%%%%%%%%%
\medskip
\section{Rigid cohomology as a limit of crystalline cohomologies}\label{RigCoh}
\smallskip
%%%%%%%%%%%%%%%%%%%%%%%%%%%%%%%%%%%%%%%%%%%%%%%%%%%%%%%%%%%%%%%%

Our goal now is to apply the comparison theorems of the previous section to the
computation of rigid cohomology with compact supports in terms of de Rham-Witt
complexes, for a separated $k$-scheme of finite type. The key case is the case of a
proper $k$-scheme, and our method is based on the interpretation of the rigid
cohomology of such a scheme as a limit of suitable crystalline cohomologies (see
\cite[section 2]{Cr04} for a closely related point of view).

%%%%%%%%%%%%%%%%%%%%%%%%%%%%%%%%
\subsection{}\label{ssTubes} 
%%%%%%%%%%%%%%%%%%%%%%%%%%%%%%%%
We briefly recall some basic constructions used in the 
definition of rigid cohomology for proper $k$-schemes (see \cite{Be96}, \cite{Be97a}, 
\cite{LS07} for details and proofs).

Let $\P$ be a smooth formal $W$-scheme, and let $\P_K$ denote its generic fibre in
Raynaud's sense, which is a rigid analytic space over $K$. If $\P$ is affine, defined
by $\P = \Spf A$ where $A$ is topologically of finite type over $W$, then $\P_K$ is the
affinoid space $\Spm A_K$ defined by the Tate algebra $A_K$. In general, $\P_K$ is
endowed with a specialisation morphism $\spm : \P_K \to \P$, which is a continuous map
such that $\spm^{-1}(\U) = \U_K$ for any open subset $\U \subset \P$. By construction,
we have 
%%%%%%%%%%%%%%%%
\eqn{ \spm_{\ast}\sO_{\P_K}\ =\ \sO_{\P,K}\ :=\ \sO_{\P}\otimes K. }
%%%%%%%%%%%%%%%%

Let $P$ be the special fibre of $\P$, and let $X \inj P$ be a closed subscheme defined
by an ideal $\sJ \subset \sO_\P$. The tube of $X$ in $\P_K$ is the subset $]X[_\P\, :=
\spm^{-1}(X)$. If $\U = \Spf A \subset \P$ is an affine open subset, and if 
$f_1,\ldots,f_r \in A$ is a family of generators of $\Gamma(\U, \sJ)$, then
%%%%%%%%%%%%%%%%
\eq{desctube}{ ]X[_{\P} \cap \U_K\ = \{ x \in \U_K\ |\ \forall i, |f_i(x)| < 1 \}, }
%%%%%%%%%%%%%%%%
where, for a point $x \in \U_K$ corresponding to a maximal ideal $\mathfrak{m} \subset
A_K$, $|f_i(x)|$ is the absolute value of the class of $f_i$ in the residue field $K(x)
= A_K/\mathfrak{m}$. In particular, $]X[_\P$ is an admissible open subset of $\P_K$.

As the specialization morphism $\spm : \P_K\to\P$ maps $]X[_{\P}$ to $X$, one defines a
sheaf of $\sO_{\P,K}$-algebras supported on $X$ by setting
%%%%%%%%%%%%%%%%
$$ \sA_{X,\P}\ =\ \spm_{\ast}\sO_{]X[_{\P}}.$$
%%%%%%%%%%%%%%%%
For any affine open subset $\U \subset \P$ and any $j$, there is a canonical
isomorphism $\Gamma(\U_K,\Omega^j_{\P_K}) \simeq \Gamma(\U, \Omega^j_{\P})\otimes K$,
which is compatible with differentiation. It follows that the canonical derivation of
$\sO_{]X[}$ endows $\sA_{X,\P}$ with an integrable connection. Therefore, one
can define the de Rham complex $\sA_{X,\P}\otimes\Omd{\P}$, and, taking into account
that each $\Omega^j_{\P}$ is locally free of finite rank, one gets a canonical
isomorphism of complexes
%%%%%%%%%%%%%%%%
$$ \sA_{X,\P}\otimes_{\sO_{\P}} \Omd{\P}\ \simeq\ 
\spm_{\ast}(\Omd{]X[_{\P}}).$$
%%%%%%%%%%%%%%%%
If $\U \subset \P$ is affine, it follows from \eqref{desctube} that $]X[_\P \cap \U_K$
is quasi-Stein \cite[Definition 2.3]{Ki67}, hence Kiehl's vanishing theorem implies
that $R^i\spm_{\ast}(\Omega^j_{]X[_{\P}}) = 0$ for all $j$ and all $i \geq 1$.
Therefore, we obtain
%%%%%%%%%%%%%%%%
\eq{fDefAXP}{ \sA_{X,\P}\otimes_{\sO_{\P}} \Omd{\P}\ \simeq\ 
\RR\spm_{\ast}(\Omd{]X[_{\P}}) }
%%%%%%%%%%%%%%%%
in the derived category $\Db(X,K)$.

As an object of $\Db(X,K)$, the complex $\sA_{X,\P}\otimes \Omd{\P}$ does not depend
(up to canonical isomorphism) on the choice of the embedding $X \inj \P$, and it is
functorial with respect to $X$ \cite[1.5]{Be97a}. We will use the notation 
%%%%%%%%%%%%%%%%
\eq{defCohRigLoc}{ \RR\uGamma\rig(X/K) := \sA_{X,\P}\otimes \Omd{\P} }
%%%%%%%%%%%%%%%%
to denote $\sA_{X,\P}\otimes \Omd{\P}$ as an object of $\Db(X, K)$. Assuming that $X$
is a proper $k$-scheme and can be embedded in a smooth formal scheme $\P$ as above, its
rigid cohomology is defined by setting
%%%%%%%%%%%%%%%%
\eqa{defCohRig}{\RR\Gamma\rig(X/K) & = & \RR\Gamma(X, \RR\uGamma\rig(X/K)) \\ 
& = & \RR\Gamma(X, \sA_{X,\P}\otimes\Omd{\P}) \simeq \RR\Gamma(]X[_{\P}, \Omd{]X[_{\P}}). 
\notag}
%%%%%%%%%%%%%%%%

In the general case, where such an embedding might not exist, one can generalize this
definition using a \v{C}ech style construction analogous to the description of the 
complex $\RR u_{X/W*}\,\sO_{X/W}$ given in \ref{ssCohCris} (see also
\cite[pp.~28-29]{Ha75}, where the analogous construction is used for algebraic de Rham
cohomology in characteristic $0$). One can choose a covering of $X$ by affine open
subsets $U_\alpha$, and, for each $\alpha$, a closed immersion of $U_\alpha$ in a
smooth affine formal scheme $\P_\alpha$ over $W$. Then the diagonal immersions for the
finite intersections 
%%%%%%%%%%%%%%%%
\eqn{ U_{\ualpha} :=  U_{\alpha_0} \cap \ldots \cap U_{\alpha_i} \inj \P_{\ualpha} := 
\P_{\alpha_0} \times_W \cdots \times_W \P_{\alpha_i} }
%%%%%%%%%%%%%%%%
provide algebras $\sA_{U_{\ualpha},\P_{\ualpha}}$. From the de Rham complexes of the
products $\P_{\ualpha}$ with coefficients in these algebras, one can build as in
\ref{ssCohCris} a \v{C}ech double complex. As an object of $D(X, K)$, the associated
total complex does not depend on the choices, and it is functorial with respect to $X$. It
will still be denoted by $\RR\uGamma\rig(X/K)$. When $X$ is proper, its rigid cohomology is
again defined by $\RR\Gamma\rig(X/K) := \RR\Gamma(X, \RR\uGamma\rig(X/K))$.

%%%%%%%%%%%%%%%%%%%%%%%%%%%%%%%%
\subsection{}\label{ssFlatEnv}
%%%%%%%%%%%%%%%%%%%%%%%%%%%%%%%%
To construct a de Rham-Witt complex computing rigid cohomology for proper varieties, we
will use a description of the algebra $\sA_{X,\P}$ based on PD-envelopes. 

Let $\sJ \subset \sO_{\P}$ be a coherent ideal. We denote 
by $\sP(\sJ)$ the divided power envelope of $\sJ$ (with compatibility with the
natural divided powers of $p$). By definition, $\sP(\sJ)$ is the sheaf of
$\sO_{\P}$-algebras associated to the presheaf $\U \mapsto \sP(\Gamma(\U,\sJ))$, where
$\sP(\Gamma(\U,\sJ))$ is the PD-envelope of the ideal $\Gamma(\U,\sJ)
\subset \Gamma(\U,\sO_{\P})$. However, the formation of PD-envelopes
commutes with flat ring extensions \cite[3.21]{BO}. Since $\sJ$ is coherent over
$\sO_{\P}$ and $\Gamma(\V,\sO_{\P})$ is flat over $\Gamma(\U,\sO_{\P})$ for any couple
of affine open subsets $\V \subset \U$ in $\P$, it follows that the functor $\U \mapsto
\sP(\Gamma(\U,\sJ))$ satisfies the glueing condition on the category of affine open
subsets of $\P$. This implies that
%%%%%%%%%%%%%%%%
\eq{sectPDEnv}{\Gamma(\U, \sP(\sJ)) = \sP(\Gamma(\U, \sJ))}
%%%%%%%%%%%%%%%%
for any affine open subset $\U \subset \P$.

We denote by $\sPh(\sJ)$ the $p$-adic completion of $\sP(\sJ)$. As divided power
envelopes may have some $p$-torsion, we also introduce the ideal $\sT(\sJ)$ of 
$p$-torsion sections of $\sP(\sJ)$, the quotient $\sPb(\sJ) = 
\sP(\sJ)/\sT(\sJ)$, and its $p$-adic completion $\sPt(\sJ)$. Note that the
natural connection of the divided power envelope $\sP(\sJ)$ extends to the algebras
$\sPh(\sJ)$, $\sPb(\sJ)$ and $\sPt(\sJ)$.

We first observe that these algebras have quasi-coherent reduction mod $p^n$:

%%%%%%%%%%%%%%%%%%%%%%%%%%%%%%%%
\begin{lem}\label{ssQcoh}
Let $\sJ \subset \sO_{\P}$ be a coherent ideal, $n \geq 1$ an integer, and $P_n$ the 
reduction of $\P$ on $W_n$. Define 
\ga{}{\sP_n(\sJ) = \sP(\sJ)/p^n\sP(\sJ),\quad\quad 
\sPb_n(\sJ) = \sPb(\sJ)/p^n\sPb(\sJ),\notag\\
\sT_n(\sJ) = \Ker(p^n : \sP(\sJ) \to \sP(\sJ)).\notag}
Then:\\
\romain $\sP_n(\sJ)$ and $\sPb_n(\sJ)$ are quasi-coherent
$\sO_{P_n}$-algebras.\\
\romain $\sT_n(\sJ)$ is a quasi-coherent $\sO_{P_n}$-module.
\end{lem}
%%%%%%%%%%%%%%%%%%%%%%%%%%%%%%%%

\begin{proof}
As the divided powers of $\sP(\sJ)$ are compatible with the 
divided powers of $p$, there is a canonical isomorphism
%%%%%%%%%%%%%%%%
$$ \sP_n(\sJ) = \sP(\sJ)/p^n\sP(\sJ) \xrightarrow{\ \sim\ \,} 
\sP(\sJ_n),$$ % end
%%%%%%%%%%%%%%%%
where $\sJ_n = \sJ\sO_{P_n}$ \cite[3.20-8)]{BO}. The ideal
$\sJ_n$ is quasi-coherent over $\sO_{P_n}$, and therefore
the $\sO_{P_n}$-algebra $\sP(\sJ_n)$ is also quasi-coherent
\cite[3.30]{BO}, which proves the first claim.

Since $\sT(\sJ) = \cup_i \sT_i(\sJ)$, the commutative diagrams
%%%%%%%%%%%%%%%%
\eqn{ \xymatrix{ 
0 \ar[r] & \sT_i(\sJ) \ar[r] \ar@{^{(}->}[d] & \sP(\sJ) 
\ar[r]^-{p^i} \ar@{=}[d] & p^i\sP(\sJ) \ar[r] \ar[d]^-{p} & 0 \\
0 \ar[r] & \sT_{i+1}(\sJ) \ar[r] & \sP(\sJ) \ar[r]^-{p^{i+1}} 
& p^{i+1}\sP(\sJ) \ar[r] & 0 
} }
%%%%%%%%%%%%%%%%
provide an isomorphism
%%%%%%%%%%%%%%%%
$$ \varinjlim_i p^i\sP(\sJ) \xrightarrow{\ \sim\ \,} 
\sPb(\sJ),$$
%%%%%%%%%%%%%%%%
in which the transition maps of the direct system are defined by
multiplication by $p$. Therefore, we obtain an isomorphism of 
$\sO_{P_n}$-modules
%%%%%%%%%%%%%%%%
\begin{eqnarray*}
\sPb_n(\sJ)\ =\ \sPb(\sJ)/p^n\sPb(\sJ) 
& \simeq & \varinjlim_i (p^i\sP(\sJ)/p^{n+i}\sP(\sJ))\\
& \simeq & \varinjlim_i (p^i\sP_{n+i}(\sJ)).
\end{eqnarray*}
%%%%%%%%%%%%%%%%
As the $\sO_{P_{n+i}}$-algebra $\sP_{n+i}(\sJ)$ is quasi-coherent,
$p^i\sP_{n+i}(\sJ)$ is a quasi-coherent $\sO_{P_n}$-module. Taking
the direct limit of these for variable $i$, it follows that
$\sPb_n(\sJ)$ is quasi-coherent over $\sO_{P_n}$.

Let $\V \subset \U$ be two affine open subsets of $\P$. Using the 
flatness of $\Gamma(\V,\sO_{\P})$ over $\Gamma(\U,\sO_{\P})$, we 
deduce from \eqref{sectPDEnv} that the canonical morphisms
%%%%%%%%%%%%%%%%
\eqa{ssQcohP}{\Gamma(\V,\sO_{\P}) \otimes_{\Gamma(\U,\sO_{\P})} 
\Gamma(\U,\sP(\sJ)) \to \Gamma(\V,\sP(\sJ)),\\
\Gamma(\V,\sO_{\P}) \otimes_{\Gamma(\U,\sO_{\P})} 
\Gamma(\U,\sT_n(\sJ)) \to \Gamma(\V,\sT_n(\sJ)) \notag}
%%%%%%%%%%%%%%%%
are isomorphisms. Therefore, $\sT_n(\sJ)$ is quasi-coherent over 
$\sO_{P_n}$.
\end{proof}

%%%%%%%%%%%%%%%%%%%%%%%%%%%%%%%%
\subsection{}\label{ssLimitEnv}
%%%%%%%%%%%%%%%%%%%%%%%%%%%%%%%%

We now fix a notation. Let $A$ be a commutative ring, $\fa \subset A$ an ideal that 
contains $p$, and $m$ a positive integer. For any ideal $I \subset A$, we will denote
by $I^{(m)}_{\fa}$ the ideal of $A$ generated by $\fa$ and the $p^m$-th powers of
elements of $I$, or, equivalently, by $\fa$ and the $p^m$-th powers of a family of 
generators of $I$. When $\fa = pA$, we will simply use the notation $I^{(m)}$. Note 
that 
%%%%%%%%%%%%%%%%
\eq{fSymbPow}{ I^{(m)}_{\fa} = I^{(m)} + \fa. }
%%%%%%%%%%%%%%%%
We extend this definition in an obvious way to sheaves of rings on topological spaces.

As in \ref{ssTubes}, let $\P$ be a smooth formal scheme, and let $X \inj P$ be a
closed subscheme of its special fibre, defined by an ideal $\sJ \subset \sO_{\P}$. We
apply the previous definition to $\sJ$, using the ideal $\fa = p\sO_{\P} \subset
\sO_{\P}$. We obtain in this way coherent ideals $\sJm \subset \sO_{\P}$.

As $\sJ^{(m+1)} \subset \sJm$ for all $m \geq 0$, the $\sO_{\P}$-algebras $\sPh(\sJm)$
and $\sPt(\sJm)$ defined in \ref{ssFlatEnv} sit in inverse systems
%%%%%%%%%%%%%%%%
\gan{\cdots \to \sPh(\sJ^{(m+1)}) \to \sPh(\sJm) \to \cdots \to 
\sPh(\sJ^{(0)}), \\
\cdots \to \sPt(\sJ^{(m+1)}) \to \sPt(\sJm) \to \cdots \to 
\sPt(\sJ^{(0)}). }
%%%%%%%%%%%%%%%%
The following proposition is the key to comparisons between
crystalline and rigid cohomologies in the proper singular case (see
also \cite[Section 2]{Cr04}):

%%%%%%%%%%%%%%%%%%%%%%%%%%%%%%%%
\begin{prop}\label{ssCompInvSys}
With the previous notation, there exists functorial horizontal
isomorphisms
%%%%%%%%%%%%%%%%
\ga{fAPinvlim}{\sA_{X,\P} \xrightarrow{\ \sim\ \,} \varprojlim_m 
(\sPh(\sJm)_K)
\xrightarrow{\ \sim\ \,} \varprojlim_m (\sPt(\sJm)_K).}
%%%%%%%%%%%%%%%%
Moreover, for any coherent ideal $\sK \subset \sO_{\P}$, 
%%%%%%%%%%%%%%%%
\eq{fVanRlim1}{ R^i\varprojlim_m (\sPh(\sJm + \sK)_K) = 
R^i\varprojlim_m (\sPt(\sJm + \sK)_K) = 0 }
%%%%%%%%%%%%%%%%
for all $i \geq 1$.
\end{prop}
%%%%%%%%%%%%%%%%%%%%%%%%%%%%%%%%

\begin{proof}
For all $m \geq 0$, let $\eta_m = p^{-1/p^m}$. We recall that the
closed tube $[X]_{\P,\eta_m}$ is defined as the open subset
of $\P_K$ characterized by
%%%%%%%%%%%%%%%%
$$[X]_{\P,\eta_m} \cap \U_K = \{x\in\U_K\ |\ \forall i, |f_i(x)| \leq 
\eta_m \} $$
%%%%%%%%%%%%%%%%
for any affine open subset $\U \subset \P$ and any family of
generators $(f_i)$ of $\sJ$ over $\U$. Then $]X[_{\P}\ = \cup_m
[X]_{\P,\eta_m}$ is an admissible covering of $]X[_\P$ by an increasing family of
open subsets \cite[1.1.9]{Be96}. It follows that, for any
open subset $\U \subset \P$, the algebra of analytic functions on
$]X[_{\P} \cap \U_K$ is defined by
%%%%%%%%%%%%%%%%
$$ \Gamma(]X[_{\P}\cap\U_K, \sO_{]X[_{\P}}) =
\varprojlim_m \Gamma([X]_{\P,\eta_m}\cap\U_K, \sO_{]X[_{\P}}).$$
%%%%%%%%%%%%%%%%
Let $\sA^{(m)}_{X,\P} = \spm_{\ast}(\sO_{[X]_{\P,\eta_m}})$. Then we 
obtain
%%%%%%%%%%%%%%%%
$$ \sA_{X,\P} \xrightarrow{\ \sim\ \,} \varprojlim_m 
\sA^{(m)}_{X,\P}.$$
%%%%%%%%%%%%%%%%

We now define canonical morphisms
%%%%%%%%%%%%%%%%
$$ \cdots \to \sPh(\sJm)_K \to \sPt(\sJm)_K \to \sA^{(m)}_{X,\P} 
\to \sPh(\sJ^{(m-1)})_K \to \cdots $$
%%%%%%%%%%%%%%%%
such that the composition of any three consecutive morphisms is the 
transition morphism in the corresponding inverse system. The first one 
is simply defined by the canonical map $\sPh(\sJm) \to \sPt(\sJm)$. 

To define the two other ones, we work on affine open subsets and glue the 
local constructions. If $\U = \Spf A \subset \P$, then the above 
description of $[X]_{\P,\eta_m}\cap\U_K$ shows that this is an affinoid 
subset of $\U_K$. From the definition of rigid analytic functions on 
such an affinoid, we obtain 
%%%%%%%%%%%%%%%%
\begin{eqnarray*}
\Gamma(\U, \sA^{(m)}_{X,\P}) & = & \Gamma([X]_{\P,\eta_m}\cap\U_K, 
\sO_{\P_K})\\
& = & A_K\{T_1,\ldots,T_r\}/(f_1^{p^m}-pT_1,\ldots,f_r^{p^m}-pT_r),
\end{eqnarray*}
%%%%%%%%%%%%%%%%
where $f_1,\ldots,f_r$ are generators of $\Gamma(\U, \sJ)$,
$T_1,\ldots,T_r$ is a family of indeterminates, and
$A_K\{T_1,\ldots,T_r\} \subset A_K[[T_1,\ldots,T_r]]$ is the subring of
formal power series the coefficients of which tend to $0$. Since
$\sA^{(m)}_{X,\P}$ is a $\Q$-algebra, the morphism $\sO_{\P} \to
\sA^{(m)}_{X,\P}$ factors uniquely through a morphism 
$\sP(\sJ^{(m)})/\sT(\sJ^{(m)}) \to \sA^{(m)}_{X,\P}$, and we obtain ring 
homomorphisms 
%%%%%%%%%%%%%%%%
$$ \Gamma(\U,\sP(\sJm)) \to \Gamma(\U,\sPb(\sJm)) \xrightarrow{\lambda} 
\Gamma(\U, \sA^{(m)}_{X,\P}).$$
%%%%%%%%%%%%%%%%
For $i\geq 1$, the composed homomorphism maps a section $(f_j^{p^m})^{[i]}$ 
to $f_j^{p^mi}/i!= (p^i/i!)T_j^i$. On the other hand, 
\eqref{sectPDEnv} shows that $\Gamma(\U,\sP(\sJm))$ is
generated as a $\Gamma(\U,\sO_{\P})$-module by the products 
$(f_1^{p^m})^{[i_1]} \cdots (f_r^{p^m})^{[i_r]}$, and the first
homomorphism is surjective, since Lemma \ref{ssQcoh} (ii) and the
quasi-compacity of $\U$ imply that $H^1(\U,\sT(\sJ^{(m)}))=0$. As $p^i/i!
\in \Z_p$, it follows that $\lambda(\Gamma(\U,\sPb(\sJm)))$ sits
inside the image of $A\{T_1,\ldots,T_r\}$ in $\Gamma(\U,
\sA^{(m)}_{X,\P})$. This implies that $\lambda$ is a continuous
morphism when the source is endowed with the $p$-adic topology and the
target with its Tate algebra topology. Therefore $\lambda$ factorizes
uniquely through the $p$-adic completion of $\Gamma(\U,\sPb(\sJm))$.

We observe here that the natural map
%%%%%%%%%%%%%%%%
$$ \Gamma(\U,\sPb(\sJm))/p^n\Gamma(\U,\sPb(\sJm)) \to \Gamma(\U, 
\sPb(\sJm)/p^n\sPb(\sJm)) $$
%%%%%%%%%%%%%%%%
is an isomorphism for all $n$: since $\sPb(\sJm)$ is $p$-torsion free,
it suffices to show that $H^1(\U,\sPb(\sJm))= 0$. But
$H^1(\U,\sP(\sJm))=0$, because \eqref{sectPDEnv} implies that
$\sP(\sJm)$ is a direct limit of coherent $\sO_{\P}$-modules, and
$H^2(\U,\sT(\sJ^{(m)}))=0$, because Lemma \ref{ssQcoh} (ii) shows that
$\sT(\sJ^{(m)})$ is a direct limit of quasi-coherent $\sO_{P_n}$-modules.
So the $p$-adic completion of $\Gamma(\U,\sPb(\sJm))$ is isomorphic to
$\Gamma(\U,\sPt(\sJm))$, and the factorization of $\lambda$ provides
the morphism $\Gamma(\U,\sPt(\sJm)_K) \to \Gamma(\U,
\sA^{(m)}_{X,\P})$.

Finally, one can define a morphism 
%%%%%%%%%%%%%%%%
$$A[T_1,\ldots,T_r]/(f_1^{p^m}-pT_1,\ldots,f_r^{p^m}-pT_r) \to 
\Gamma(\U,\sP(\sJ^{(m-1)}))$$
%%%%%%%%%%%%%%%%
by sending $T_j$ to $(p-1)!(f_j^{p^{m-1}})^{[p]}$. Factorizing through
the $p$-adic completions and tensorizing with $K$ provides the 
morphism $\Gamma(\U,\sA^{(m)}_{X,\P})$ $\to \Gamma(\U,
\sPh(\sJ^{(m-1)})_K)$, which is also unique.

Note that all these morphisms are horizontal, since they are
continuous and induce the identity map on $\sO_{\P,K}$, which is dense
in every algebra under consideration.

Applying the functor $\varprojlim_m$, the previous morphisms of 
inverse systems of sheaves yield isomorphisms
%%%%%%%%%%%%%%%%
\eqn{ \varprojlim_m (\sPh(\sJm)_K) \simeq 
\varprojlim_m (\sPt(\sJm)_K) \simeq 
\varprojlim_m \sA^{(m)}_{X,\P} \simeq \sA_{X,\P}. }
%%%%%%%%%%%%%%%%

To prove the second part of the proposition, it suffices to show that 
$R^i\varprojlim_m (\sPh(\sJm + \sK)_K) = 0$ for $i \geq 1$. Observe first that, for 
any affine open subset $\U \subset \P$,  
%%%%%%%%%%%%%%%%
\eq{fVanishJm}{ H^j(\U, \sPh(\sJm + \sK)) = 0 }
%%%%%%%%%%%%%%%%
for $j \geq 1$, thanks to Lemma \ref{ssQcoh} and to the algebraic Mittag-Leffler 
criterion. As $\U$ is noetherian, taking sections on $\U$ commutes with tensorisation 
with $K$, and it follows that 
%%%%%%%%%%%%%%%%
\eq{fVanishJmK}{ H^j(\U, \sPh(\sJm + \sK)_K) = 0  }
%%%%%%%%%%%%%%%%
for $j \geq 1$. On the other hand, $\Gamma(\U, \sP(\sJ^{(m)} + \sK))_K = \Gamma(\U,
\sO_{\P})_K$ for any $m \geq 0$, so that the image of $\Gamma(\U,\sO_{\P})_K$ in
$(\Gamma(\U, \sP(\sJ^{(m)} + \sK))^{\widehat{\ }}\,)_K$ is dense for the $p$-adic topology.
Lemma \ref{ssQcoh} inmplies that $\Gamma(\U, \sP(\sJ^{(m)} + \sK))^{\widehat{\ }} =
\Gamma(\U, \sPh(\sJ^{(m)} + \sK))$, hence it follows that, for $m' \geq m$, the image
of $\Gamma(\U,\sPh(\sJ^{(m')} + \sK)_K)$ is dense in $\Gamma(\U,\sPh(\sJ^{(m)} +
\sK)_K)$. Therefore, the inverse system of topological vector spaces
$(\Gamma(\U,\sPh(\sJ^{(m)} + \sK)_K))_{m\geq 0}$ satifies the topological
Mittag-Leffler condition (ML$'$) \cite[$0_{\mathrm{III}}$, 13.2.4 (i)]{EGA}. It is easy
to see that, together with \eqref{fVanishJmK}, this property implies that the inverse
system $(\sPh(\sJm + \sK)_K)_{m \geq 0}$ is $\varprojlim$-acyclic.
\end{proof}

%%%%%%%%%%%%%%%%%%%%%%%%%%%%%%%%
\begin{cor}\label{ssCompRigdR}
Let $X$ be a proper $k$-scheme, and let $X\inj\P$ be a closed immersion into a smooth formal
scheme over $W$. There exists canonical isomorphisms
%%%%%%%%%%%%%%%%
\ml{fRigtoAPLoc}{\RR\uGamma\rig(X/K)  \xrightarrow{\ \sim\ \,} 
\varprojlim_m (\sPh(\sJm)_K) \otimes_{\sO_{\P}}\Omd{\P} \\
 \xrightarrow{\ \sim\ \,} 
\varprojlim_m (\sPt(\sJm)_K) \otimes_{\sO_{\P}}\Omd{\P},} 
\ml{fRigtoAP}{\RR\Gamma\rig(X/K)  \xrightarrow{\ \sim\ \,}  \RR\Gamma(X, 
\varprojlim_m (\sPh(\sJm)_K) \otimes_{\sO_{\P}}\Omd{\P}) \\
 \xrightarrow{\ \sim\ \,}  \RR\Gamma(X, 
\varprojlim_m (\sPt(\sJm)_K) \otimes_{\sO_{\P}}\Omd{\P}),}
%%%%%%%%%%%%%%%%
functorial with respect to the immersion $X\inj\P$. 
\end{cor}
%%%%%%%%%%%%%%%%%%%%%%%%%%%%%%%%

\begin{proof}
Thanks to the previous proposition, this follows immediately from
\eqref{defCohRigLoc} and \eqref{defCohRig}.
\end{proof}

We now observe that Proposition \ref{ssCompInvSys} allows to write the 
rigid cohomology of a proper scheme as the inverse limit of the crystalline 
cohomologies of its infinitesimal neighbourhoods in a smooth scheme (tensorized by $K$).

%%%%%%%%%%%%%%%%%%%%%%%%%%%%%%%%
\begin{thm}\label{ssCompRigCrys}
Let $Y$ be a smooth $k$-scheme, and let $X \inj Y$ be a closed subscheme that is proper over 
$k$. Let $\sI \subset \sO_Y$ be the ideal of $X$ in $Y$, and let $X^{(m)} \subset Y$ be the 
closed subscheme defined by $\sI^{(m)}$ \(as defined in \nref{ssLimitEnv}\). There exists 
functorial isomorphisms 
%%%%%%%%%%%%%%%%
\eqa{fCompRigCrysLoc}{ \RR\uGamma\rig(X/K) & \riso & \RR\varprojlim_m 
((\RR u_{X^{(m)}/W*}\,\sO_{X^{(m)}/W})_K), \\
\RR\Gamma\rig(X/K) & \riso & \RR\varprojlim_m 
(\RR \Gamma\cris(X^{(m)}/W, \sO_{X^{(m)}/W})_K) \label{fCompRigCrys} }
%%%%%%%%%%%%%%%%
in $\Db(X,K)$ and $\Db(K)$ respectively. 
\end{thm}
%%%%%%%%%%%%%%%%%%%%%%%%%%%%%%%%

\begin{proof}
For $m \leq m'$, let $i^{(m',m)}\cris : (X^{(m)}/W)\cris \to (X^{(m')}/W)\cris$ denote
the morphism of topos induced by the inclusion $X^{(m)} \inj X^{(m')}$. To define the
right hand side of \eqref{fCompRigCrysLoc} and \eqref{fCompRigCrys}, one can work with
the category of inverse systems $(E^{(m)}, \rho^{(m)})$, where $E^{(m)}$ is a sheaf on
the site $\Cris(X^{(m)}/W)$ and $\rho^{(m)}$ is a morphism of sheaves
$(i^{(m,m-1)}\cris)^{-1}E^{(m)} \to E^{(m-1)}$. These systems form a topos
$(X^{(\sbul)}/W)\cris$, and the family of structural sheaves $\sO_{X^{(m)}/W}$,
together with the functoriality morphisms, is an object $\sO_{X^{(\sbul)}/\W}$ in this
topos. Using the fact that $|X^{(m)}| = |X|$ for all $m$, the family of functors
$u_{X^{(m)}/W*}$ defines a functor $u_{X^{(\sbul)}/W*}$ from $(X^{(\sbul)}/W)\cris$ to
the topos $X^{(\sbul)}\Zar$ of inverse systems of sheaves on $X$. We denote by
$\RR\varprojlim_m ((\RR u_{X^{(m)}/W*}\,\sO_{X^{(m)}/W})_K)$ the complex obtained by
applying to $\sO_{X^{(\sbul)}/\W}$ the right derived functor of $u_{X^{(\sbul)}/W*}$,
tensorizing with $K$, and applying the $\RR\varprojlim_m$ functor. The complex
$\RR\varprojlim_m (\RR \Gamma\cris(X^{(m)}/W, \sO_{X^{(m)}/W})_K)$ is defined
similarly.

Let us assume first that $Y$ can be embedded in a smooth formal scheme $\P$.
Let $\sJ$ be the ideal of $X$ in $\P$, and let $\sK$ be the ideal of $Y$ in $\P$. Then
$X^{(m)}$ is the closed subscheme of $\P$ defined by the ideal $\sJ^{(m)} + \sK$.
Thanks to the functoriality properties of the linearization functor used in the
construction of the comparison isomorphism between crystalline and de Rham cohomologies
\cite[6.9]{BO}, the family of linearizations of $\Omd{\P}$ on the sites
$\Cris(X^{(m)}/W)$ provides a resolution of $\sO_{X^{(\sbul)}/\W}$ that is acyclic for
the functor $u_{X^{(\sbul)}/W*}$. This defines an isomorphism
%%%%%%%%%%%%%%%%
\eqn{ \RR u_{X^{(\sbul)}/W*}(\sO_{X^{(\sbul)}/\W}) \riso \sPh(\sJ^{(\sbul)}+\sK) 
\otimes \Omd{\P} }
%%%%%%%%%%%%%%%%
in the derived categoy of inverse systems of $W$-modules on $X$, which has the
isomorphism \eqref{fCrisdRLoc} as component of index $m$ for each $m$. Tensorizing 
with $K$, applying $\RR\varprojlim_m$, and taking \eqref{fVanRlim1} into account, one
gets an isomorphism
%%%%%%%%%%%%%%%%
\eq{fInvSysCristodR}{ \RR\varprojlim_m ((\RR u_{X^{(m)}/W*}\,\sO_{X^{(m)}/W})_K) 
\riso \varprojlim_m (\sPh(\sJ^{(m)}+\sK)_K) \otimes \Omd{\P}. }
%%%%%%%%%%%%%%%%
As 
%%%%%%%%%%%%%%%%
\eqn{ \RR\uGamma\rig(X/K) \riso \varprojlim_m (\sPh(\sJ^{(m)})_K) \otimes \Omd{\P} }
%%%%%%%%%%%%%%%%
by \eqref{fRigtoAPLoc}, it suffices to prove that the functoriality morphism 
%%%%%%%%%%%%%%%%
\eq{fJmtoJmK}{ \varprojlim_m (\sPh(\sJ^{(m)})_K) \otimes \Omd{\P} \lra 
\varprojlim_m (\sPh(\sJ^{(m)}+\sK)_K) \otimes \Omd{\P} }
%%%%%%%%%%%%%%%%
is a quasi-isomorphism.

This is a local statement on $X$, so we may assume that $\P$, $Y$, and $X$ are affine.
It is independent of the choice of the embedding $\P$ of $Y$, because it is functorial
with respect to $\P$, so the morphisms defined by any two embeddings $\P$, $\P'$ can be
compared via the diagonal embedding into $\P \times_W\P'$. But the left
hand side of \eqref{fJmtoJmK} is independent of $\P$ up to quasi-isomorphism because
it computes canonically $\RR\uGamma\rig(X/K)$, and so does the right hand side because,
for each $m$, it computes canonically $(\RR u_{X^{(m)}/W*}\sO_{X^{(m)}/W})_K$ and it
has $\varprojlim$-acyclic terms. This allows to reduce to the case where $\P$ is a 
smooth lifting of $Y$. Then the source and target of \eqref{fJmtoJmK} are identical, 
and the morphism is the identity.

This defines the isomorphism \eqref{fCompRigCrysLoc} when $Y$ can be embedded into a 
smooth formal scheme over $W$. In the general case, 
one chooses an affine open covering $\fU = (U_\alpha)$ of $Y$ and closed immersions 
$U_{\alpha} \inj \P_{\alpha}$ into smooth formal schemes over $W$. Then one
can use for each $X^{(m)}$ the \v{C}ech resolution introduced in \ref{ssCohCris} to
compute $\RR\varprojlim_m ((\RR u_{X^{(m)}/W*}\,\sO_{X^{(m)}/W})_K)$ and the \v{C}ech
resolution introduced in \ref{ssTubes} to define $\RR\uGamma\rig(X/K)$. The
functoriality of the previous construction provides a morphism between these \v{C}ech
resolutions, and this morphism is a quasi-isomorphism because the corresponding
morphism for each intersection $U_{\ualpha}$ of the covering is a quasi-isomorphism.

The isomorphism \eqref{fCompRigCrys} follows from \eqref{fCompRigCrysLoc} by taking 
global sections and using the commutations between $\RR\varprojlim$ and $\RR\Gamma(X, 
-)$, and between $\RR\Gamma(X,-)$ and tensorisation by $K$.
\end{proof}

%%%%%%%%%%%%%%%%%%%%%%%%%%%%%%%%
\begin{rmk}\label{ssCompRigCrys1}
%%%%%%%%%%%%%%%%%%%%%%%%%%%%%%%%
As $Y$ is noetherian, the ind-scheme $(X^{(m)})_{m \geq 0}$ is isomorphic to the
ind-scheme that consists in all infinitesimal neighbourhoods $X_n$ of $X$ in $Y$. 
Therefore, one can rewrite isomorphisms \eqref{fCompRigCrysLoc} and
\eqref{fCompRigCrys} as
%%%%%%%%%%%%%%%%
\eqa{fCompRigCrysLoc1}{ \RR\uGamma\rig(X/K) & \riso & \RR\varprojlim_n 
((\RR u_{X_n/W*}\,\sO_{X_n/W})_K), \\
\RR\Gamma\rig(X/K) & \riso & \RR\varprojlim_n 
(\RR \Gamma\cris(X_n/W, \sO_{X_n/W})_K). \label{fCompRigCrys1} }
%%%%%%%%%%%%%%%%
\end{rmk}
%%%%%%%%%%%%%%%%

%%%%%%%%%%%%%%%%%%%%%%%%%%%%%%%%
\medskip
\section{The comparison theorem}\label{sProofMain}
\smallskip
%%%%%%%%%%%%%%%%%%%%%%%%%%%%%%%%

Given a proper $k$-scheme $X$ enbedded as a closed subscheme in a smooth $k$-scheme
$Y$, we construct now a de Rham-Witt complex on $Y$ with coefficients supported in
$X$, and we prove that its cohomology is canonically isomorphic to the rigid cohomology
of $X$. This de Rham Witt complex can easily be defined using divided power envelopes
for the ideals of sections of $W_n\sO_Y$ that vanish on the infinitesimal
neighbourhoods $X^{(m)}$ of $X$ in $Y$. However, because of the flatness over $W_n$
required in Theorem \ref{ssCrysnEtodRWE}, we need to use also an alternate definition
based on $p$-torsion free quotients of PD-envelopes, analogous to the construction 
introduced in the previous section.

More generally, we prove a similar comparison theorem for the rigid cohomology
with compact supports of a separated $k$-scheme of finite type.

%%%%%%%%%%%%%%%%%%%%%%%%%%%%%%%%
\subsection{}\label{ssAXYdRW} 
%%%%%%%%%%%%%%%%%%%%%%%%%%%%%%%%
Let $Y$ be a smooth $k$-scheme of finite type, and let $i : X \inj Y$ be the inclusion
of a closed subscheme defined by an ideal $\sI_Y \subset \sO_Y$. For each PD-thickening
$(U,T,\delta)$ of an open subset $U$ of $Y$, we will denote by $\sK_T$ the ideal of $U$
in $T$ and by $\sI_T$ the ideal of $X \cap U$ in $T$.

For each $m \geq 0$, let $i^{(m)} : X^{(m)} \inj Y$ be the inclusion of the closed
subscheme defined by $\sI_Y^{(m)}$, where $\sI_Y^{(m)}$ is defined as in \ref{ssLimitEnv}.
The ideal of $X^{(m)}$ in $T$ is then equal to
%%%%%%%%%%%%%%%%
\eq{fImTK}{ (\sI_T)^{(m)}_{\sK_T} := (f^{p^m})_{f \in \sI_T} + \sK_T. }
%%%%%%%%%%%%%%%%

We will work with the crystalline direct images $i^{(m)}\cris{}_{\,*}\sO_{X^{(m)}/W}$.
As recalled in \ref{ssImDirCris}, these are crystals on $Y/W$, which can be described as
follows: for each $(U,T,\delta)$ as above, the evaluation on $T$ of
$i^{(m)}\cris{}_{\,*}\sO_{X^{(m)}/W}$ is given by
%%%%%%%%%%%%%%%%
\eq{fPmXY}{ (i^{(m)}\cris{}_{\,*}\sO_{X^{(m)}/W})_T =
\sP_\delta((\sI_T)^{(m)}_{\sK_T}), }
%%%%%%%%%%%%%%%%
where $\sP_\delta$ denotes the divided power envelope with compatiblity with the
divided powers $\delta$ on $\sK_T$. When $(U,T,\delta)$ varies, the transition
maps are defined by the functoriality of divided power envelopes.

To simplify the notation, we set $\sP^{(m)}_{X,Y} = i^{(m)}\cris{}_{\,*}\sO_{X^{(m)}/W}$. 
For $m \geq 1$, the inclusions $X^{(m-1)} \inj X^{(m)}$ define functoriality morphisms
%%%%%%%%%%%%%%%%
\eqn{ \rho^{(m)} : \sP^{(m)}_{X,Y} \lra \sP^{(m-1)}_{X,Y}, }
%%%%%%%%%%%%%%%%
which turn the family $(\sP^{(m)}_{X,Y})_{m\geq 0}$ into an inverse system of
crystals. As in \ref{ssCrys}, we can take the evaluations of these crystals on the
thickenings $Y \inj W_nY$, and let $n, m$ vary. In this way, we obtain an inverse
system of $W_n\sO_Y$-modules $(\sP^{(m)\,W}_{X,Y,n})^{\ }_{m,n}$ and an inverse system
of complexes $(\sP^{(m)\,W}_{X,Y,n} \otimes_{W_n\sO_Y} \WOmd{n}{Y})^{\ }_{m,n}$. We
define
%%%%%%%%%%%%%%%%
\gan{ \sPh^{(m)\,W}_{X,Y} := \varprojlim_n  \sP^{(m)\,W}_{X,Y,n}, \quad\quad 
\sA^W_{X,Y} := \varprojlim_m \,(\sPh^{(m)\,W}_{X,Y} \otimes K), \\
\sPh^{(m)\,W}_{X,Y} \otimesh_{W\sO_Y} \WOmd{}{Y} := 
\varprojlim_n (\sP^{(m)\,W}_{X,Y,n} \otimes_{W_n\sO_Y} \WOmd{n}{Y}), \\
\sA^W_{X,Y} \otimesh_{W\sO_Y} \WOmd{}{Y} := \varprojlim_m 
\,((\sPh^{(m)\,W}_{X,Y} \otimesh_{W\sO_Y} \WOmd{}{Y}) \otimes K).
}
%%%%%%%%%%%%%%%%
As $\WOm{n}{0}{Y} = W_n\sO_Y$, we get in degree 0 
%%%%%%%%%%%%%%%%
\gan{ \sPh^{(m)\,W}_{X,Y} \otimesh_{W\sO_Y} \WOm{}{0}{Y} = 
\varprojlim_n (\sP^{(m)\,W}_{X,Y,n} \otimes_{W_n\sO_Y} \WOm{n}{0}{Y}) 
\simeq \sPh^{(m)\,W}_{X,Y}, \\
\sA^W_{X,Y} \otimesh_{W\sO_Y} \WOm{}{0}{Y} = \varprojlim_m 
\,((\sPh^{(m)\,W}_{X,Y} \otimesh_{W\sO_Y} \WOm{}{0}{Y}) \otimes K)
\simeq \sA^W_{X,Y}. }
%%%%%%%%%%%%%%%%
We will denote the transition morphisms by 
%%%%%%%%%%%%%%%%
\gan{ \rhoh^{(m)} : \sPh^{(m)\,W}_{X,Y} \to \sPh^{(m-1)\,W}_{X,Y}, \\
\rhoh^{(m)\sbul} : \sPh^{(m)\,W}_{X,Y} \otimesh_{W\sO_Y} \WOmd{}{Y} \to 
\sPh^{(m-1)\,W}_{X,Y} \otimesh_{W\sO_Y} \WOmd{}{Y}. }
%%%%%%%%%%%%%%%%

Note that, since the crystals $\sP^{(m)}_{X,Y}$ are supported in $X$, so are the
complexes $\sPh^{(m)\,W}_{X,Y} \otimesh_{W\sO_Y} \WOmd{}{Y}$ and $\sA^W_{X,Y}
\otimesh_{W\sO_Y} \WOmd{}{Y}$. We can now state the main result of this article.

%%%%%%%%%%%%%%%%%%%%%%%%%%%%%%%%
\begin{thm}\label{ssMain}
%%%%%%%%%%%%%%%%%%%%%%%%%%%%%%%%
Let $X$ be a $k$-scheme of finite type, and let $i : X \inj Y$ be a closed immersion
into a smooth $k$-scheme.

\romain If there exists a closed immersion $Y \inj \P$ of $Y$ into a smooth 
formal $W$-scheme $\P$ endowed with a lifting $F$ of the absolute Frobenius endomorphism, 
these data define a quasi-isomorphism of complexes of sheaves of $K$-vector spaces
supported in $X$
%%%%%%%%%%%%%%%%
\eq{fMainF}{ \sA_{X,\P} \otimes_{\sO_{\P}} \Omd{\P} \lra 
\sA^W_{X,Y} \otimesh_{W\sO_Y} \WOmd{}{Y}, }
%%%%%%%%%%%%%%%%
functorial with respect to $(X,Y,\P,F)$. Via the identifications of \cite[1.5]{Be97a},
its image in $\Db(X, K)$ is independent of the choice of the embedding of $Y$ into
$(\P,F)$.

\romain Without extra assumption, there exists in $\Db(X, K)$ a canonical isomorphism 
%%%%%%%%%%%%%%%%
\eq{fMainLoc}{ \RR\uGamma\rig(X/K) \riso \sA^W_{X,Y} \otimesh_{W\sO_Y} \WOmd{}{Y}, }
%%%%%%%%%%%%%%%%
functorial with respect to the couple $(X,Y)$ and defined by \eqref{fMainF} 
whenever the assumptions of {\normalfont{(i)}} are satisfied. 

\romain If $X$ is proper, there exists in $\Db(K)$ a canonical isomorphism
%%%%%%%%%%%%%%%%
\eq{fMainGlob}{ \RR\Gamma\rig(X/K) \riso 
\RR\Gamma(X, \sA^W_{X,Y} \otimesh_{W\sO_Y} \WOmd{}{Y}), }
%%%%%%%%%%%%%%%%
functorial with respect to the couple $(X, Y)$. 
%%%%%%%%%%%%%%%%
\end{thm}
%%%%%%%%%%%%%%%%

The proof of the theorem will be given in subsection \ref{ssProofMain}. It requires
some additional constructions, which we detail in the next subsections.

%%%%%%%%%%%%%%%%%%%%%%%%%%%%%%%%
\subsection{}\label{ssAltPm}
%%%%%%%%%%%%%%%%%%%%%%%%%%%%%%%%
We first want to show that one can perform constructions similar to those of
\ref{ssAXYdRW} using the $p$-torsion free envelopes defined in \ref{ssLimitEnv}. 

We begin with a local construction. Let $U \subset Y$ be an affine open subset, and let
$\U$ be a smooth affine formal scheme on $W$ lifting $U$; for all $n \geq 1$, let $U_n$
be the reduction of $\U$ on $W_n$. We denote by $\sI \subset \sO_\U$ the ideal defining
$X \cap U$ in $\U$, and we define $\sI^{(m)}$, $\sP(\sI^{(m)})$, $\sPh(\sI^{(m)})$,
$\sT(\sI^{(m)})$, $\sPb(\sI^{(m)})$ and $\sPt(\sI^{(m)})$ as in \ref{ssLimitEnv}. We set
%%%%%%%%%%%%%%%%
\gan{ \sP_n(\sI^{(m)}) = \sP(\sI^{(m)})/p^n\sP(\sI^{(m)}) \simeq 
\sPh(\sI^{(m)})/p^n\sPh(\sI^{(m)}), \\
\sPb_n(\sI^{(m)}) = \sPb(\sI^{(m)})/p^n\sPb(\sI^{(m)}) \simeq 
\sPt(\sI^{(m)})/p^n\sPt(\sI^{(m)}). }
%%%%%%%%%%%%%%%%

Let us recall how the crystal $\sP^{(m)}_{X,Y}$ corresponds to the $\sO_{\U}$-module
$\sPh(\sI^{(m)})$ in the classical equivalence between crystals on $\Cris(U/W)$ and
$p$-adically complete $\sO_{\U}$-modules endowed with an integrable and topologically
quasi-nilpotent connection. If $(V,T)$ is a thickening in $\Cris(U/W)$, there exists
locally on $T$ an integer $n\geq 1$ such that $p^n\sO_T = 0$, and a $W$-morphism $h : T
\to U_n$ extending the immersion $V \inj U_n$. Then the $\sO_T$-algebra
$h^*(\sP_n(\sI^{(m)}))$ does not depend on the choice of $n$ and $h$, up to a canonical
isomorphism defined by the HPD-stratification of $\sP_n(\sI^{(m)})$, and it can be 
identified with the Zariski sheaf $(\sP^{(m)}_{X,Y})_T$ on $T$ defined by the crystal
$\sP^{(m)}_{X,Y}$ \cite[Th.~6.6 and claim p.~6.3]{BO}.

As we already observed in \ref{ssLimitEnv}, the quotient $\sPb_n(\sI^{(m)})$ of
$\sP(\sI^{(m)})$ is endowed with the quotient connection. Since the connection of
$\sP_n(\sI^{(m)})$ is quasi-nilpotent and $\sPb_n(\sI^{(m)})$ is a quotient of
$\sP_n(\sI^{(m)})$, the connection of $\sPb_n(\sI^{(m)})$ is also quasi-nilpotent. It
follows that $\sPb_n(\sI^{(m)})$ defines by the previous construction a crystal in
$\sO_{U/W}$-algebras $\sPb^{(m)}_{X\cap U,U}$, endowed with a surjective morphism
$\pi^{(m)} : \sP_{X\cap U,U}^{(m)} \surj \sPb^{(m)}_{X\cap U,U}$.

In order to be able to glue such local constructions and to define a crystal globally
on $Y$, we need to show that the crystal $\sPb^{(m)}_{X\cap U,U}$ does
not depend, up to canonical isomorphism, on the choice of the lifting $\U$. Let
$\U'$ be a second lifting of $U$, defining a crystal
$\sPbp{m}{X\cap U,U}$ by the previous construction. As $U$ is 
affine, there exists a (non-canonical) $W$-isomorphism $u : \U \riso \U'$ inducing the
identity on $U$, hence on each $X^{(m)}$. If $\sI'$ denotes the ideal of $X$ in $\U'$,
the functoriality homomorphism $u^*\sP(\sI'{}^{(m)}) \to \sP(\sI^{(m)})$ is
then an isomorphism compatible with the connections. It induces an isomorphism
$u^*\sT(\sI'{}^{(m)}) \riso \sT(\sI^{(m)})$ between the torsion subsheaves, and we get
a horizontal isomorphism
%%%%%%%%%%%%%%%%
\eqn{ \varepsilon_u : u^*\sPb(\sI'{}^{(m)}) \riso \sPb(\sI^{(m)}). }
%%%%%%%%%%%%%%%%
If $(V,T)$ is a thickening in $\Cris(U/W)$, if we choose a morphism $h : T \to U_n$ as
above to define $(\sPb^{(m)}_{X\cap U,U})_T$, and if we use the
composition $h' = u_n \circ h : T \to U'_n$ to define
$(\sPbp{m}{X\cap U,U})_T$ (denoting by $u_n$ the reduction of $u$
modulo $p^n$), then the evaluation on $T$ 
%%%%%%%%%%%%%%%%
\eqn{ \varepsilon_T : (\sPbp{m}{X\cap U,U})_T \riso
(\sPb{}^{(m)}_{X\cap U,U})_T }
%%%%%%%%%%%%%%%%
of the wanted canonical isomorphism $\varepsilon : \sPbp{m}{X\cap
U,U} \riso \sPb^{(m)}_{X\cap U,U}$ is defined to be the pullback of
the reduction of $\varepsilon_u$ mod $p^n$ by $h$. One observes that, if one changes
the choices made for $h$, $h'$, or $u$, then the pull-back of $\varepsilon_u$ is
modified by the same isomorphisms deduced from the HPD-stratification that occur in
the definition of the Zariski sheaves $(\sPb^{(m)}_{X\cap U,U})_T$ and
$(\sPbp{m}{X\cap U,U})_T$, so that $\varepsilon_T$ finally does
not depend on any choice. One also checks that, thanks to the cocycle property of the
HPD-stratification, these isomorphisms satisfy the necessary transitivity properties. 
This allows to glue the local definitions on affine open subsets and to construct the 
crystal $\sPb^{(m)}_{X,Y}$ globally on $Y$. 

For $m \geq 1$, the natural homomorphisms $\sPb(\sI^{(m)}) \to \sPb(\sI^{(m-1)})$
defined on smooth liftings of affine open subschemes of $Y$ can also be glued so as to
define homomorphisms of crystals $\rhob^{(m)} : \sPb^{(m)}_{X,Y} \to \sPb^{(m-1)}_{X,Y}$ on
$Y$, turning the family of crystals $(\sPb^{(m)}_{X,Y})_{m \geq 0}$ into an inverse system.
On the other hand, we have by construction a surjective morphism of crystals
$\pi^{(m)} : \sP^{(m)}_{X,Y} \to \sPb^{(m)}_{X,Y}$ for each $m$, and these define a
morphism of inverse systems. The next lemma will allow us to replace, up to isogeny,
the inverse system of crystals $(\sP^{(m)}_{X,Y})_{m\geq 0}$ by the inverse system
$(\sPb^{(m)}_{X,Y})_{m \geq 0}$.

%%%%%%%%%%%%%%%%%%%%%%%%%%%%%%%%
\begin{lem}\label{ssCompInv}
With the previous notation and hypotheses, let $m$ be fixed, and let $s \geq 0$ be an
integer. If $s$ is big enough, there exists a unique morphism of crystals $\sigma_s^{(m)} :
\sPb^{(m)}_{X,Y} \to \sP^{(m-1)}_{X,Y}$ such that
%%%%%%%%%%%%%%%%
\eq{isog}{ \sigma_s^{(m)} \circ \pi^{(m)} = p^s\rho^{(m)}, \quad\quad 
\pi^{(m-1)} \circ \sigma_s^{(m)} = p^s\rhob^{(m)}. }
%%%%%%%%%%%%%%%%
\end{lem}
%%%%%%%%%%%%%%%%%%%%%%%%%%%%%%%%

\begin{proof}
The unicity claim follows from the fact that each $\pi^{(m)}$ is a epimorphism of
crystals. Note also that, if there exists a couple $(s, \sigma_s^{(m)})$ that satifies
the conditions of the lemma, then, for any $t \geq 0$, the couple $(s+t,
\sigma_{s+t}^{(m)} := p^t\sigma_s^{(m)})$ also satifies these conditions. As $Y$ is
quasi-compact, this allows to check the lemma locally on $Y$.

So we may assume that $Y$ is affine and has a smooth formal lifting $\Y$ on $W$. Let 
us keep the notation of \ref{ssAltPm}, only replacing $U$ and $\U$ by $Y$ and $\Y$. 
Let $(f_1,\ldots,f_r)$ be a family of generators of $\sI$, and let $\sR = 
\sO_\Y[T_1,\ldots,T_r]$. We define an $\sO_{\Y}$-algebra $\sA^{(m)}(f_1,\ldots,f_r)$
(the completion of which is an integral model of the algebra $\sA^{(m)}_{X,\Y}$ used in
the proof of Proposition \ref{ssCompInvSys}) by setting
%%%%%%%%%%%%%%%%
\eqn{ \sA^{(m)}(f_1,\ldots,f_r) := 
\sR/(f_1^{p^m} - pT_1,\ldots,f_r^{p^m} - pT_r).}
%%%%%%%%%%%%%%%%
The canonical homomorphism $\sO_\Y \to \sA^{(m)}(f_1,\ldots,f_r)$ maps the ideal 
$\sI^{(m)}$ to the PD-ideal $p\sA^{(m)}(f_1,\ldots,f_r)$, hence it factorizes uniquely 
through a PD-morphism
%%%%%%%%%%%%%%%%
\eq{defalpham}{ \alpha^{(m)} : \sP(\sI^{(m)}) \to \sA^{(m)}(f_1,\ldots,f_r) }
%%%%%%%%%%%%%%%%
On the other hand, there is a unique homomorphism of $\sO_\Y$-algebras 
%%%%%%%%%%%%%%%%
\eq{defbetam}{ \beta^{(m)} : \sA^{(m)}(f_1,\ldots,f_r) \to \sP(\sI^{(m-1)}) }
%%%%%%%%%%%%%%%%
sending $T_i$ to $(p-1)!(f_i^{p^{m-1}})^{[p]}$ for $i = 1,\ldots,r$. Let us also denote
by $\rho^{(m)}$ the canonical PD-morphism $\sP(\sI^{(m)}) \to \sP(\sI^{(m-1)})$. For
all $i$ and all $k \geq 0$, we have
%%%%%%%%%%%%%%%%
\eqna{ \beta^{(m)} \circ \alpha^{(m)}((f_i^{p^m})^{[k]}) & 
= & \beta^{(m)}(p^{[k]} T_i^k) \\
& = & p^{[k]} ((p-1)!(f_i^{p^{m-1}})^{[p]})^k \\
& = & (p(p-1)!(f_i^{p^{m-1}})^{[p]})^{[k]} \\
& = & (\rho^{(m)}(f_i^{p^m}))^{[k]}\ =\ \rho^{(m)}((f_i^{p^m})^{[k]}). }
%%%%%%%%%%%%%%%%
As the products $(f_1^{p^m})^{[k_1]} \cdots (f_r^{p^m})^{[k_r]}$ generate
$\sP(\sI^{(m)})$ as an $\sO_\Y$-module, we get that
%%%%%%%%%%%%%%%%
\eq{betalpha}{ \beta^{(m)} \circ \alpha^{(m)} = \rho^{(m)}. }
%%%%%%%%%%%%%%%%

Recall that, $\Y$ being noetherian and separated, the functors $H^q(\U, -)$ commute 
with direct limits for any open subset $\U \subset \Y$ and any $q \geq 0$. In 
particular, we have 
%%%%%%%%%%%%%%%%
\eqn{ \Gamma(\U, \sR) =  \Gamma(\U, \sO_\Y)[T_1,\ldots,T_r] }
%%%%%%%%%%%%%%%%
for any open subset $\U \subset \Y$, and, by \cite[Prop.~3.1.1]{DMA1}, 
$\sR$ is a coherent sheaf of rings on $\Y$. As 
$\sA^{(m)}(f_1,\ldots,f_r)$ is finitely presented over $\sR$, it is 
coherent too, and the same holds for the kernel $\sK_n$ of multiplication by $p^n$ on 
$\sA^{(m)}(f_1,\ldots,f_r)$ for any $n$. Therefore, $\sK_n$ is a coherent 
$\sO_{Y_n}[T_1,\ldots,T_r]$-module, which implies that it is generated by its global 
sections, and so is the subsheaf $\sK = \cup_n \sK_n$ of $p$-torsion sections of 
$\sA^{(m)}(f_1,\ldots,f_r)$. 

On the other hand, the ideal $\sN = \sum_i \sR\cdot f_i$ is the increasing union of 
the $\sO_\Y$-coherent subsheaves $\sN_d = \sum_i \sR_d\cdot f_i$, where $\sR_d \subset
\sR$ is the subsheaf of sections of total degree $\leq d$. Therefore, $\H^1(\Y, \sN) = 
0$, and $\Gamma(\Y, \sA^{(m)}(f_1,\ldots,f_r))$ is a quotient of $\Gamma(\Y, \sR)$, 
hence is a noetherian ring. So there is an integer $s \geq 0$ such that $p^s g = 0$ for
any $p$-torsion section $g \in \Gamma(\Y, \sA^{(m)}(f_1,\ldots,f_r))$. Since $\sK$ is
generated by its global sections on $\Y$, we obtain that $p^s\sK = 0$.

The homomorphism $\alpha^{(m)}$ maps the torsion ideal $\sT(\sI^{(m)}) \subset
\sP(\sI^{(m)})$ to $\sK$. Using \eqref{betalpha}, it follows that $p^s\rho^{(m)}$
vanishes on $\sT(\sI^{(m)})$, which provides a factorization of $p^s\rho^{(m)}$ through
$\sPb(\sI^{(m)})$. Reducing mod $p^n$ for all $n$ and taking the associated morphism of
crystals, we get a morphism $\sigma_s^{(m)} : \sPb^{(m)}_{X,Y} \to \sP^{(m-1)}_{X,Y}$, 
which satisfies the first relation of \eqref{isog}. As $\pi^{(m)}$ is an epimorphism,
the second one follows.
\end{proof}

%%%%%%%%%%%%%%%%%%%%%%%%%%%%%%%%
\subsection{}\label{ssDefDRWt}
%%%%%%%%%%%%%%%%%%%%%%%%%%%%%%%%
We can now apply the constructions of \ref{ssAXYdRW} to the crystals $\sPb^{(m)}_{X,Y}$.
Taking their evaluations on the thickenings $Y \inj W_nY$, we obtain again an inverse
system of $W_n\sO_Y$-modules $(\sPb^{(m)\,W}_{X,Y,n})^{\ }_{m,n}$ and an inverse
system of complexes $(\sPb^{(m)\,W}_{X,Y,n} \otimes_{W_n\sO_Y} \WOmd{n}{Y})^{\ }_{m,n}$. 
We now define
%%%%%%%%%%%%%%%%
\gan{ \sPt^{(m)\,W}_{X,Y} := \varprojlim_n  \sPb^{(m)\,W}_{X,Y,n}, \quad\quad 
\sAt^W_{X,Y} := \varprojlim_m \,(\sPt^{(m)\,W}_{X,Y} \otimes K), \\
\sPt^{(m)\,W}_{X,Y} \otimesh_{W\sO_Y} \WOmd{}{Y} := 
\varprojlim_n (\sPb^{(m)\,W}_{X,Y,n} \otimes_{W_n\sO_Y} \WOmd{n}{Y}), \\
\sAt^W_{X,Y} \otimesh_{W\sO_Y} \WOmd{}{Y} := \varprojlim_m 
\,((\sPt^{(m)\,W}_{X,Y} \otimesh_{W\sO_Y} \WOmd{}{Y}) \otimes K). }
%%%%%%%%%%%%%%%%
In degree $0$, these complexes are respectively equal to $\sPt^{(m)\,W}_{X,Y}$ and
$\sAt^W_{X,Y}$. We will denote by $\rhot^{(m)} : \sPt^{(m)\,W}_{X,Y} \to
\sPt^{(m-1)\,W}_{X,Y}$ and $\rhot^{(m)\sbul} : \sPt^{(m)\,W}_{X,Y} \otimesh_{W\sO_Y} 
\WOmd{}{Y} \to \sPt^{(m-1)\,W}_{X,Y} \otimesh_{W\sO_Y} \WOmd{}{Y}$ the transition
morphisms.

By functoriality, the morphisms of crystals $\pi^{(m)} : \sP^{(m)}_{X,Y} \to 
\sPb^{(m)}_{X,Y}$ define homomorphisms of $W\sO_Y$-algebras 
%%%%%%%%%%%%%%%%
\eq{fMort}{ \pih^{(m)} : \sPh^{(m)\,W}_{X,Y} \to \sPt^{(m)\,W}_{X,Y}, \quad\quad  
\pi : \sA^W_{X,Y} \to \sAt^W_{X,Y}, }
%%%%%%%%%%%%%%%%
and morphisms of complexes on $Y$ (supported in $X$)
%%%%%%%%%%%%%%%%
\ga{fMorPdRWt}{ \pih^{(m)\sbul} : \sPh^{(m)\,W}_{X,Y} \otimesh_{W\sO_Y} \WOmd{}{Y} \to 
\sPt^{(m)\,W}_{X,Y} \otimesh_{W\sO_Y} \WOmd{}{Y} \\
\label{fMorAdRWt} \pi\hbul : \sA^W_{X,Y} \otimesh_{W\sO_Y} \WOmd{}{Y} \to 
\sAt^W_{X,Y} \otimesh_{W\sO_Y} \WOmd{}{Y}. 
 }
%%%%%%%%%%%%%%%%

%%%%%%%%%%%%%%%%%%%%%%%%%%%%%%%%
\begin{prop}\label{ssIsoDRWt}
Let $Y$ be a smooth $k$-scheme, and let $X \subset Y$ be a closed subscheme.

\romain The morphism $\pi\hbul$ defined by \eqref{fMorAdRWt} is an isomorphism.

\romain For all $m\geq 0$, $j \geq 0$ and $i \geq 1$, we have
%%%%%%%%%%%%%%%%
\ga{fVanRlimP}{ 
R^i\varprojlim_n (\sPb^{(m)\,W}_{X,Y,n} \otimes_{W_n\sO_Y} \WOm{n}{j}{Y}) = 0, \\ 
\label{fVanRlimA}
R^i\varprojlim_m ((\sPt^{(m)\,W}_{X,Y} \otimesh_{W\sO_Y} \WOm{}{j}{Y}) \otimes K) = 0.}
%%%%%%%%%%%%%%%%
\end{prop}
%%%%%%%%%%%%%%%%%%%%%%%%%%%%%%%%

In particular, the homomorphism $\pi$ of \eqref{fMort} is an isomorphism.

%%%%%%%%%%%%%%%%
\begin{proof}
Applying Lemma \ref{ssCompInv}, we can find an increasing sequence of integers
$(s(m))_{m\geq 1}$ such that, for each $m$, there exists a morphism of crystals
$\sigma^{(m)}_{s(m)} : \sPb^{(m)}_{X,Y} \to \sP^{(m-1)}_{X,Y}$ satisfying 
\eqref{isog} for $s = s(m)$. For any $m \geq 1$, these morphisms satisfy the relation
%%%%%%%%%%%%%%%%
\eq{compsigma}{ p^{s(m+1)-s(m)}\sigma_{s(m)}^{(m)} \circ \rhob^{(m+1)} = 
\rho^{(m)} \circ \sigma_{s(m+1)}^{(m+1)}, }
%%%%%%%%%%%%%%%%
which follows from \eqref{isog} because $\pi^{(m+1)}$ is an epimorphism. 

By functoriality, the morphisms $\sigma_{s(m)}^{(m)}$ define homomorphisms of
$W\sO_Y$-algebras $\sigmat_{s(m)}^{(m)} : \sPt^{(m)\,W}_{X,Y} \to \sPh^{(m-1)\,W}$ and
morphisms of complexes
%%%%%%%%%%%%%%%%
\eqn{ \sigmat_{s(m)}^{(m)\sbul} : \sPt^{(m)\,W}_{X,Y} \otimesh_{W\sO_Y} \WOmd{}{Y} \to 
\sPh^{(m-1)\,W}_{X,Y} \otimesh_{W\sO_Y} \WOmd{}{Y}. }
%%%%%%%%%%%%%%%%
They satisfy the analogs of relations \eqref{isog} and \eqref{compsigma}. If we set
%%%%%%%%%%%%%%%%
\eqn{ \psi^{(m)\sbul} = p^{-s(m)}\sigmat_{s(m)}^{(m)} : 
\sPt^{(m)\,W}_{X,Y} \otimesh_{W\sO_Y} \WOmd{}{Y} \otimes K \to 
\sPh^{(m-1)\,W}_{X,Y} \otimesh_{W\sO_Y} \WOmd{}{Y} \otimes K, }
%%%%%%%%%%%%%%%%
we obtain when $m$ varies a family of morphisms of complexes, which commute with the 
transition morphisms $\rhoh^{(m)\sbul}$ and $\rhot^{(m)\sbul}$ and satisfy 
%%%%%%%%%%%%%%%%
\eqn{ \psi^{(m)\sbul} \circ \pih^{(m)\sbul} = \rhoh^{(m)\sbul}, \quad\quad 
\pih^{(m-1)\sbul} \circ \psi^{(m)\sbul} 
= \rhot^{(m)\sbul}. }
%%%%%%%%%%%%%%%%
Taking the inverse limit of the morphisms $\psi^{(m)\sbul}$ when $m$ varies provides a
morphism of complexes
%%%%%%%%%%%%%%%%
\eqn{ \psi\hbul : \sAt^W_{X,Y} \otimesh_{W\sO_Y} \WOmd{}{Y} \to 
\sA^W_{X,Y} \otimesh_{W\sO_Y} \WOmd{}{Y}, }
%%%%%%%%%%%%%%%%
which is inverse to $\pi\hbul$. 

To prove assertion (ii), we observe first that, by \ref{ssComplTens} and Lemma
\ref{ssQcoh}, $\sPb^{(m)\,W}_{X,Y,n} \otimes_{W_n\sO_Y} \WOm{n}{j}{Y}$ is a
quasi-coherent $W_n\sO_Y$-module, for all $n \geq 1$, $m \geq 0$. So, for fixed $m$ and
variable $n$, the inverse system $(\sPb^{(m)\,W}_{X,Y,n} \otimes_{W_n\sO_Y}
\WOm{n}{j}{Y})_{n\geq 1}$ has surjective transition morphisms and terms with vanishing
cohomology in degree $\geq 1$ on any affine open subset of $Y$. Therefore
\eqref{fVanRlimP} follows from the algebraic Mittag-Leffler criterion \cite[Prop.\ 
(13.2.3)]{EGA}.

Note that this Mittag-Leffler criterion also implies that, for any affine open subset 
$U \subset Y$, any $i \geq 1$, and any $j$, 
%%%%%%%%%%%%%%%%
\eqn{ H^i(U, \sPt^{(m)\,W}_{X,Y} \otimesh_{W\sO_Y} \WOm{}{j}{Y}) = 0. }
%%%%%%%%%%%%%%%%
Since $U$ is quasi-compact and separated,  it follows that, for $i \geq 1$, 
%%%%%%%%%%%%%%%%
\eq{fVanHiK}{ H^i(U, (\sPt^{(m)\,W}_{X,Y} \otimesh_{W\sO_Y} \WOm{}{j}{Y}) \otimes K) = 0. }
%%%%%%%%%%%%%%%%
So, in order to prove \eqref{fVanRlimA}, it suffices to prove that each $K$-vector
space $\Gamma(U, (\sPt^{(m)\,W}_{X,Y} \otimesh_{W\sO_Y} \WOm{}{j}{Y}) \otimes K)$ can
be endowed with a metrizable topological vector space structure for which it is
separated and complete, and such that the transition maps for varying $m$ are
continuous and satisfy the topological Mittag-Leffler condition (ML$'$)
\cite[O$_\text{III}$, Rem.\ (13.2.4) (i)]{EGA}. Since, by construction, 
$\sPt^{(m)\,W}_{X,Y}$ is flat over $W$, Proposition \ref{sspCompl} allows to use for
this purpose the $p$-adic topology, defined by the lattice $\Gamma(U,
\sPt^{(m)\,W}_{X,Y} \otimesh_{W\sO_Y} \WOm{}{j}{Y})$, which is metrizable, separated, 
and complete. To finish the proof, it is then sufficient to prove that, for all $m \geq
0$, the image of $\Gamma(U, \WOm{}{j}{Y} \otimes K)$ in $\Gamma(U, (\sPt^{(m)\,W}_{X,Y}
\otimesh_{W\sO_Y} \WOm{}{j}{Y}) \otimes K)$ is dense for the $p$-adic
topology in the image of $\Gamma(U,
(\sPt^{(m+1)\,W}_{X,Y} \otimesh_{W\sO_Y} \WOm{}{j}{Y}) \otimes K)$. 
As multiplication by $p$ is an homeomorphism of $\Gamma(U,
(\sPt^{(m)\,W}_{X,Y}
\otimesh_{W\sO_Y} \WOm{}{j}{Y}) \otimes K)$ for the $p$-adic topology, it is enough to
show that, for any $n \geq 0$,
%%%%%%%%%%%%%%%%
\ml{fDense}{ \rhot^{(m+1)}(\Gamma(U, \sPt^{(m+1)\,W}_{X,Y} \otimesh_{W\sO_Y} 
\WOm{}{j}{Y})) \subset \\
\Gamma(U, \WOm{}{j}{Y}) \otimes K + p^n\Gamma(U, \sPt^{(m)\,W}_{X,Y} \otimesh_{W\sO_Y} 
\WOm{}{j}{Y}). }
%%%%%%%%%%%%%%%%

Let $f_1,\ldots,f_r \in \Gamma(U, \sO_Y)$ be a family of generators over $U$ of the
ideal $\sI_Y$ of $X$ in $Y$. Then, for all $i \geq 2$ and all $m \geq 0$, the kernel
$\sK^{(m)}_i$ of the homomorphism $W_i\sO_Y \to \sO_{X^{(m)}}$ is generated by
$VW_{i-1}\sO_Y$ and by the Teichm\"uller representatives
$[f_1^{p^{m}}],\ldots,[f_r^{p^m}]$. By construction, $\sPb^{(m)\,W}_{X,Y,i}$ is the
divided power envelope of $\sK^{(m)}_i$ with compatibility with the canonical divided
powers of $VW_{i-1}\sO_Y$. This implies that $\sPb^{(m)\,W}_{X,Y,i}$ is generated as a
$W_i\sO_Y$-module by the sections
%%%%%%%%%%%%%%%%
\eq{fPDgen}{ ([\uf]^{p^m})^{[\uk]} := ([f_1]^{p^m})^{[k_1]} \cdots 
([f_r]^{p^m})^{[k_r]}, }
%%%%%%%%%%%%%%%%
for $\uk = (k_1,\ldots,k_r) \in \N^r$. Therefore, an element $x \in \Gamma(U,
\sPt^{(m+1)\,W}_{X,Y} \otimesh_{W\sO_Y} \WOm{}{j}{Y})$ can be written as an infinite sum 
%%%%%%%%%%%%%%%%
\eqn{ x = \sum_{\uk} ([\uf]^{p^{m+1}})^{[\uk]}\otimes\omega_{\uk}, }
%%%%%%%%%%%%%%%%
where the sections $\omega_{\uk} \in \Gamma(U, \WOm{}{j}{Y})$ are such that, for any
$i$, $\omega_{\uk} \in \Gamma(U, \Fil^i\WOm{}{j}{Y})$ for all but a finite number of
$\uk$'s. As $\rhot^{(m+1)}$ is a PD-morphism, we can write 
%%%%%%%%%%%%%%%%
\eqn{ \rhot^{(m+1)}(x)  =  \sum_{\uk} (([\uf]^{p^m})^p)^{[\uk]}\otimes\omega_{\uk} 
 =  \sum_{\uk} \frac{(p\uk)!}{\uk!}([\uf]^{p^m})^{[p\uk]}\otimes\omega_{\uk}.
}
%%%%%%%%%%%%%%%%
Using the classical formula for the $p$-adic valuation of $v_p(k!)$ \cite[Lemma 3.3]{BO}, 
we get 
%%%%%%%%%%%%%%%%
\eqn{ v_p(\frac{(pk)!}{k!}) = k. }
%%%%%%%%%%%%%%%%
So we can write
%%%%%%%%%%%%%%%%
\eqn{ \rhot^{(m+1)}(x)  =  \sum_{\forall j,\, k_j < n} 
\frac{(p\uk)!}{\uk!}([\uf]^{p^m})^{[p\uk]}\otimes\omega_{\uk} + 
p^n \sum_{\exists j,\, k_j \geq n} 
\frac{(p\uk)!}{\uk!p^n}([\uf]^{p^m})^{[p\uk]}\otimes\omega_{\uk}, }
%%%%%%%%%%%%%%%%
a decomposition in which the first term belongs to $\Gamma(U, \WOm{}{j}{Y}) \otimes K$
and the second one to $p^n\Gamma(U, \sPt^{(m)\,W}_{X,Y} \otimesh_{W\sO_Y}
\WOm{}{j}{Y})$. This proves \eqref{fDense}, and ends the proof of the proposition. 
\end{proof}

%%%%%%%%%%%%%%%%%%%%%%%%%%%%%%%%
\subsection{}\label{ssProofMain}
%%%%%%%%%%%%%%%%%%%%%%%%%%%%%%%%
\textit{Proof of Theorem \ref{ssMain}}.  
Let us assume first that $Y$ can be embedded as a closed subscheme in a smooth 
formal scheme $\P$ endowed with a lifting $F$ of the absolute Frobenius endomorphism.
For $n \geq 1$, let $P_n$ be the reduction of $\P$ over $W_n$. We denote by $\sJ$ 
the ideal of $X$ in $\P$, $\sJ_n$ its ideal in $P_n$, and we keep the
notation of \ref{ssLimitEnv} and \ref{ssAXYdRW}. Let $\sK$ be the ideal of $Y$ in
$\P$, and $\sK_n$ its ideal in $P_n$.

Let $\sP(\sK)$, $\sP(\sK_n) \simeq \sP(\sK)/p^n\sP(\sK)$ be the divided powers
envelopes of $\sK$ and $\sK_n$. We denote by $T_n$ the PD-thickening of $Y$ given by
$T_n = \Spec \sP(\sK_n)$ and by $\sP^{(m)\,T}_{X,Y,n}$, $\sPb^{(m)\,T}_{X,Y,n}$ the
evaluations on $T_n$ of the crystals $\sP^{(m)}_{X,Y}$, $\sPb^{(m)}_{X,Y}$. Let 
$\sK_{T_n} = \Ker(\sO_{T_n} \surj \sO_Y)$ be the PD-ideal generated by $\sK_n$ in
$\sP(\sK_n)$, and let $\delta$ be its canonical PD-structure. Then the ideal
$\sI_{T_n}$ of $X$ in $T_n$ is given by
%%%%%%%%%%%%%%%%
\eqn{ \sI_{T_n}  =  \sJ_n\sO_{T_n} + \sK_{T_n}. }
%%%%%%%%%%%%%%%%
Therefore, 
%%%%%%%%%%%%%%%%
\eqn{ (\sI_{T_n})^{(m)}_{\sK_{T_n}} = \sJ^{(m)}_n\sO_{T_n} + \sK_{T_n}, }
%%%%%%%%%%%%%%%%
and the description \eqref{fPmXY} of $\sP^{(m)\,T}_{X,Y,n}$ shows that 
%%%%%%%%%%%%%%%%
\eqn{ \sP^{(m)\,T}_{X,Y,n} = \sP_{\delta}(\sJ^{(m)}_n \sO_{T_n}+ \sK_{T_n}) = 
\sP_{\delta}(\sJ^{(m)}_n\sO_{T_n}). }
%%%%%%%%%%%%%%%%
But the universal property of divided power envelopes provides formally a canonical 
isomorphism
%%%%%%%%%%%%%%%%
\eqn{ \sP_{\delta}(\sJ^{(m)}_n \sO_{T_n}+ \sK_{T_n}) \simeq \sP(\sJ^{(m)}_n + \sK_n). }
%%%%%%%%%%%%%%%%
So we finally get a canonical isomorphism
%%%%%%%%%%%%%%%%
\eq{fPmTn}{ \sP^{(m)\,T}_{X,Y,n} \simeq \sP(\sJ^{(m)}_n + \sK_n). }
%%%%%%%%%%%%%%%%

Let 
%%%%%%%%%%%%%%%%
\gan{ \sPh^{(m)\,T}_{X,Y} = \varprojlim_n \sP^{(m)\,T}_{X,Y,n} = \sPh(\sJ^{(m)} + \sK), \\
\sPt^{(m)\,T}_{X,Y} = \varprojlim_n \sPb^{(m)\,T}_{X,Y,n} = \sPt(\sJ^{(m)} + \sK). }
%%%%%%%%%%%%%%%%
We get by \eqref{dREtodRWE} canonical morphisms
of complexes that sit in a commutative square
%%%%%%%%%%%%%%%%
\eq{fMorm1}{ \xymatrix{
\sPh^{(m)\,T}_{X,Y} \otimes_{\sO_{P_n}} \Omd{\P} \ar[r] \ar[d] & 
\sPh^{(m)\,W}_{X,Y} \otimesh_{W\sO_Y} \WOmd{}{Y} \ar[d] \\
\sPt^{(m)\,T}_{X,Y} \otimes_{\sO_{P_n}} \Omd{\P} \ar[r] & 
\sPt^{(m)\,W}_{X,Y} \otimesh_{W\sO_Y} \WOmd{}{Y} . 
} }
%%%%%%%%%%%%%%%%
Tensorizing by $K$ and taking the inverse limit for variable $m$, we get a similar 
square 
%%%%%%%%%%%%%%%%
\eq{fMor1}{ \xymatrix{
\varprojlim_m (\sPh^{(m)\,T}_{X,Y} \otimes K) \otimes_{\sO_{\P}} \Omd{\P}
\ar[r] \ar[d]^-{\sim} & \sA^W_{X,Y} \otimesh_{W\sO_Y} \WOmd{}{Y} \ar[d]^-{\sim} \\ 
\varprojlim_m (\sPt^{(m)\,T}_{X,Y} \otimes K) \otimes_{\sO_{\P}} \Omd{\P}
\ar[r] & \sAt^W_{X,Y} \otimesh_{W\sO_Y} \WOmd{}{Y}, 
} }
%%%%%%%%%%%%%%%%
in which the right vertical arrow is an isomorphism thanks to Proposition
\ref{ssIsoDRWt}. The left vertical one is also an isomorphism, because Lemma
\ref{ssCompInv} implies by the same argument that the homomorphism
%%%%%%%%%%%%%%%%
\eqn{ \varprojlim_m (\sPh^{(m)\,T}_{X,Y} \otimes K) \to \varprojlim_m 
(\sPt^{(m)\,T}_{X,Y} \otimes K), } 
%%%%%%%%%%%%%%%%
which is induced by the morphism of inverse systems of crystals $\pi^{(m)} :
\sP^{(m)}_{X,Y} \to \sPb^{(m)}_{X,Y}$, is an isomorphism. On the other hand, Corollary
\ref{ssCrysEtodRWE} implies that the lower horizontal arrow of \eqref{fMorm1} is a
quasi-isomorphism, since the crystal $\sPb^{(m)}_{X,Y}$ is flat over $W$ and
quasi-coherent. For all $j$, the inverse system $(\sPt^{(m)\,T}_{X,Y} \otimes K \otimes
\Omega^{j}_{\P})_{m \geq 0}$ is $\varprojlim$-acyclic by \eqref{fPmTn} and
\eqref{fVanRlim1}, and the inverse system $((\sPt^{(m)\,W}_{X,Y} \otimesh_{W\sO_Y}
\WOm{}{j}{Y}) \otimes K)_{m \geq 0}$ is $\varprojlim$-acyclic by \eqref{fVanRlimA}.
Therefore, the lower horizontal arrow of \eqref{fMor1} is again a quasi-isomorphism,
and the upper horizontal arrow is a quasi-isomorphism too.

We now consider the canonical morphism of complexes defined by the commutative 
diagram 
%%%%%%%%%%%%%%%%
\eq{fMorRigtoCris}{ \xymatrix{
\sA_{X,\P} \otimes_{\sO_{\P}} \Omd{\P} \ar[d]^-{\sim}_-{\eqref{fAPinvlim}} \ar[r] & 
\varprojlim_m (\sPh^{(m)\,T}_{X,Y}\otimes K) \otimes_{\sO_{\P}} \Omd{\P} \\
\varprojlim_m \sPh(\sJ^{(m)})_K \otimes_{\sO_{\P}} \Omd{\P} \ar[r] &
\varprojlim_m \sPh(\sJ^{(m)} + \sK)_K \otimes_{\sO_{\P}} \Omd{\P}. \ar@{=}[u]
} }
%%%%%%%%%%%%%%%%
In this diagram, the lower horizontal arrow is the arrow \eqref{fJmtoJmK} induced by
the functoriality map for PD-envelopes. As shown in the proof of Theorem 
\ref{ssCompRigCrys}, this arrow is a quasi-isomorphism. Therefore, we can define the
quasi-isomorphism \eqref{fMainF} to be the composition of the morphism
\eqref{fMorRigtoCris} with the upper horizontal arrow of diagram \eqref{fMor1}. Each
of these two morphisms is functorial in an obvious way with respect to the data $(X, Y,
\P, F)$, and, as isomorphisms 
%%%%%%%%%%%%%%%%
\gan{ \RR\uGamma\rig(X/K) \riso \RR\varprojlim_m(\RR u_{Y/W\,*}\sP^{(m)}_{X,Y})_K, \\
\RR\varprojlim_m(\RR u_{Y/W\,*}\sP^{(m)}_{X,Y})_K \riso \sA^W_{X,Y}
\otimesh_{W\sO_Y} \WOmd{}{Y}}
%%%%%%%%%%%%%%%%
in the derived category $\Db(X,K)$, they are independent of the choice of the embedding
into $(\P,F)$ (by the usual diagonal embedding argument). This proves assertion (i) of
Theorem \ref{ssMain}.

To prove assertion (ii), we argue again as in \ref{ssConstCase}. We choose an affine open
covering $\fU = (U_{\alpha})$ of $Y$, and, for each $\alpha$, we choose a closed
immersion $U_{\alpha} \inj \P_{\alpha}$ into a smooth formal scheme endowed
with a lifting $F_{\alpha}$ of the absolute Frobenius endomorphism. For each
multi-index $\ualpha = (\alpha_0,\ldots,\alpha_i)$, we endow the product $\P_{\ualpha}
= P_{\alpha_0} \times_W \cdots \times_W P_{\alpha_i}$ with the endomorphism
$F_{\ualpha} = F_{\alpha_0} \times \cdots \times F_{\alpha_i}$, and we embedd
$U_{\ualpha} = U_{\alpha_0} \cap \ldots \cap U_{\alpha_i}$ diagonally into
$\P_{\ualpha}$. Let $j_{\ualpha}$ be the inclusion of $X \cap
U_{\ualpha}$ into $X$. As a representative of
$\RR\uGamma\rig(X/K)$, we use the total complex associated to the \v{C}ech bicomplex
described in \ref{ssTubes}: in bidegree $(j,i)$, it is given by
%%%%%%%%%%%%%%%%
\eq{fCechRig}{ \prod_{\ualpha = (\alpha_0,\ldots,\alpha_i)} j_{\ualpha\,*}
(\sA_{X \cap U_{\ualpha}, \P_{\ualpha}} \otimes 
\Omega^j_{\P_{\ualpha}}). }
%%%%%%%%%%%%%%%%
On the other hand, we can represent $\sA^W_{X,Y} \otimesh_{W\sO_Y} \WOmd{}{Y}$ by its
\v{C}ech resolution defined by the covering $\fU$, which is the total complex
associated to the bicomplex defined in bidegree $(j,i)$ by
%%%%%%%%%%%%%%%%
\eq{fCechdRW}{ \prod_{\ualpha=(\alpha_0,\ldots,\alpha_i)} j_{\ualpha\,*}
(\sA^W_{X \cap U_{\ualpha}, 
U_{\ualpha}} \otimesh_{W\sO_{U_{\ualpha}}} \WOm{}{j}{U_{\ualpha}}). }
%%%%%%%%%%%%%%%%
For fixed $\ualpha$ and variable $j$, we have a quasi-isomorphism of complexes
\eqref{fMainF} from $\sA_{X \cap U_{\ualpha}, \P_{\ualpha}} \otimes
\Omega^j_{\P_{\ualpha}}$ to $\sA^W_{X \cap U_{\ualpha}, U_{\ualpha}}
\otimesh_{W\sO_{U_{\ualpha}}} \WOm{}{j}{U_{\ualpha}}$, and its direct image by
$j_{\ualpha}$ remains a quasi-isomorphism, because
$j_{\ualpha}$ is an affine morphism, and, by Mittag-Leffler, the terms
of these complexes have no higher cohomology on affine subsets. The functoriality of
\eqref{fMainF} implies that these quasi-isomorphisms commute with the differential of
the \v{C}ech complexes when $\ualpha$ varies. Therefore they induce a quasi-isomorphism
between the associated total complexes, and \eqref{fMainLoc} is defined by its image in
$\Db(X,K)$.

To check that \eqref{fMainLoc} does not depend on the choices, one considers another
covering $\fU' = (U'_{\alpha'})$, with closed immersions $U'_{\alpha'} \inj
(\P'_{\alpha'},F'_{\alpha'})$. One can then refine both $\fU$ and $\fU'$ by the
covering by the intersections $U_{\alpha} \cap U'_{\alpha'}$, which can be embedded into
$(\P_{\alpha} \times_W \P'_{\alpha'}, F_{\alpha} \times F'_{\alpha'})$. Then the
functoriality of \eqref{fMainF} provides commutative diagrams showing that the
morphisms \eqref{fMainLoc} constructed by means of $\fU$ and $\fU'$ are equal in
$\Db(X,K)$. For a morphism of couples $(w, v) : (X',Y') \to (X,Y)$, a similar argument
using affine coverings $\fU'$, $\fU$ of $Y'$, $Y$, and the embeddings $U'_{\alpha'}
\cap v^{-1}(U_{\alpha}) \inj \P'_{\alpha'} \times \P_{\alpha}$ defined by the graph of
$v$, shows that the morphisms \eqref{fMainLoc} for $(X',Y')$ and $(X,Y)$ commute with
the functoriality morphisms on the sources and targets. This completes the proof of
assertion (ii).

Finally, in view of the definition of rigid cohomology for proper schemes recalled in
\ref{ssTubes}, assertion (iii) follows from assertion (ii) by taking sections on $X$.
\hfill$\Box$

%%%%%%%%%%%%%%%%%%%%%%%%%%%%%%%%
\subsection{}\label{ssCohRigc}
%%%%%%%%%%%%%%%%%%%%%%%%%%%%%%%%
We now briefly explain how to extend the previous results to rigid cohomology with
compact supports. 

Let $X$ be a $k$-scheme of finite type, $Z \subset X$ a closed subscheme, and $U = X
\setminus Z$. We generalize the complexes entering in Theorem \ref{ssMain} as
follows.

\alphab Assume that there exists a closed immersion $X \inj \P$ of $X$ into a smooth 
formal scheme $\P$. By functoriality, we obtain a canonical morphism of 
complexes 
%%%%%%%%%%%%%%%%
\eqn{ \sA_{X,\P}\otimes_{\sO_{\P}}\Omd{\P} \to \sA_{Z,\P}\otimes_{\sO_{\P}}\Omd{\P}, }
%%%%%%%%%%%%%%%%
which we view as bicomplex. We will use the shorter notation
%%%%%%%%%%%%%%%%
\eqn{ \sA^{\cc}_{X\setminus Z,\P} \otimes_{\sO_{\P}} \Omd{\P} := (\sA_{X,\P}
\otimes_{\sO_{\P}} \Omd{\P} \to \sA_{Z,\P} \otimes_{\sO_{\P}} \Omd{\P})\tot }
%%%%%%%%%%%%%%%%
for the associated total complex. From this definition, we obtain a short exact
sequence of complexes
%%%%%%%%%%%%%%%%
\eq{fExF}{ 0 \to \sA_{Z,\P}\otimes_{\sO_{\P}}\Omd{\P}[-1] \to 
\sA^{\cc}_{X\setminus Z,\P} \otimes_{\sO_{\P}} \Omd{\P} \to 
\sA_{X,\P}\otimes_{\sO_{\P}}\Omd{\P} \to 0. }
%%%%%%%%%%%%%%%%
We will denote by $\RR\uGamma\rigc(X\setminus Z/K)$ the image in $\Db(X,K)$ of the
complex $\sA^{\cc}_{X\setminus Z,\P} \otimes_{\sO_{\P}} \Omd{\P}$; up to canonical
isomorphism, it does not depend on the choice of the embedding of $X$ into $\P$.

\alphab Without embedding assumption, we can choose an affine covering $\fU = 
(U_{\alpha})$ of $X$ and closed immersions $U_{\alpha} \inj \P_{\alpha}$ into 
smooth formal schemes. We obtain a functoriality morphism between the 
\v{C}ech bicomplexes defined by \eqref{fCechRig} for $X$ and for $Z$, which we can 
view as a triple complex. We generalize the previous definition by denoting 
$\RR\uGamma\rigc(X\setminus Z/K)$ the image in $\Db(X,K)$ of the total complex
associated to this triple complex. Up to canonical isomorphism, it is still independent
of the choices. From the exact sequences \eqref{fExF} on intersections of open subsets 
of $\fU$, we get in $\Db(X,K)$ an exact triangle
%%%%%%%%%%%%%%%%
\eq{fExLoc}{ \RR\uGamma\rigc(X\setminus Z/K) \to \RR\uGamma\rig(X/K) \to 
\RR\uGamma\rig(Z/K) \xra{\,+1\,}. }
%%%%%%%%%%%%%%%%

\alphab When $X$ is proper, the rigid cohomology of $U$ with compact support can be 
defined by 
%%%%%%%%%%%%%%%%
\eqn{ \RR\Gamma\rigc(U/K) := \RR\Gamma(X, \RR\uGamma\rigc(X\setminus Z/K)). }
%%%%%%%%%%%%%%%%
Note that, under the assumptions of a), \eqref{fDefAXP} provides a canonical 
isomorphism 
%%%%%%%%%%%%%%%%
\eqn{ \RR\Gamma\rigc(U/K) \simeq \RR\Gamma(]X[_{\P}, (\Omd{]X[_{\P}} \to
u_{\ast}(\Omd{]Z[_{\P}}))\tot), }
%%%%%%%%%%%%%%%%
where $u$ denotes the inclusion $]Z[_{\P} \inj ]X[_{\P}$. Without embedding assumption, 
the complex $\RR\Gamma\rigc(U/K)$ does not depend on the compactification $X$ of $U$, 
and satisfies with respect to $U$ the usual functoriality properties of cohomology with
compact supports (see \cite[section 3]{Be86} and \cite[6.4]{LS07} for details). Taking
sections on $X$, the triangle \eqref{fExLoc} yields the usual triangle
%%%%%%%%%%%%%%%%
\eq{fExGlob}{ \RR\Gamma\rigc(U/K) \to \RR\Gamma\rig(X/K) \to 
\RR\Gamma\rig(Z/K) \xra{\,+1\,}.}
%%%%%%%%%%%%%%%%

\alphab Let now $i : X \inj Y$ be a closed immersion of $X$ into a smooth $k$-scheme.
By functoriality, we obtain a canonical morphism of complexes
%%%%%%%%%%%%%%%%
\eqn{ \sA^W_{X,Y} \otimesh_{W\sO_Y} \WOmd{}{Y} \to 
\sA^W_{Z,Y} \otimesh_{W\sO_Y} \WOmd{}{Y}, }
%%%%%%%%%%%%%%%%
which we consider again as a bicomplex. We will denote by
%%%%%%%%%%%%%%%%
\eqn{ \sA^{\cc,W}_{X\setminus Z,Y} \otimesh_{W\sO_Y} \WOmd{}{Y} := 
(\sA^W_{X,Y} \otimesh_{W\sO_Y} \WOmd{}{Y} \to 
\sA^W_{Z,Y} \otimesh_{W\sO_Y} \WOmd{}{Y})\tot}
%%%%%%%%%%%%%%%%
the associated total complex, which sits in a short exact sequence of complexes 
%%%%%%%%%%%%%%%%
\eq{fExW}{ 0 \to \sA^W_{Z,Y} \otimesh_{W\sO_Y} \WOmd{}{Y}[-1] \to 
\sA^{\cc,W}_{X\setminus Z,Y} \otimesh_{W\sO_Y} \WOmd{}{Y} \to 
\sA^W_{X,Y} \otimesh_{W\sO_Y} \WOmd{}{Y} \to 0. }
%%%%%%%%%%%%%%%%

%%%%%%%%%%%%%%%%%%%%%%%%%%%%%%%%
\begin{cor}\label{ssCorMain}
Let $X$ be a $k$-scheme of finite type, $Z \subset X$ a closed subscheme, $U = X 
\setminus Z$, and $i : X \inj Y$ a closed immersion in a smooth $k$-scheme. 

\romain If there exists a closed immersion $Y \inj \P$ of $Y$ into a smooth 
formal scheme $\P$ endowed with a Frobenius lifting, these data define a
quasi-isomorphism of complexes of sheaves of $K$-vector spaces supported in $X$
%%%%%%%%%%%%%%%%
\eq{fMaincF}{ \sA^{\cc}_{X\setminus Z,\P} \otimes_{\sO_{\P}} \Omd{\P} \lra 
\sA^{\cc,W}_{X\setminus Z,Y} \otimesh_{W\sO_Y} \WOmd{}{Y}, }
%%%%%%%%%%%%%%%%
functorial with respect to $(Z,X,Y,\P,F)$. Its image in $\Db(X, K)$ is independent of
the choice of the embedding of $Y$ into $(\P,F)$. Together with the quasi-isomorphisms
\eqref{fMainF} for $X$ and $Z$, it defines a morphism of exact sequences from the exact
sequence \eqref{fExF} to the exact sequence \eqref{fExW}.

\romain Without extra assumption, there exists in $\Db(X, K)$ a canonical isomorphism 
%%%%%%%%%%%%%%%%
\eq{fMaincLoc}{ \RR\uGamma\rigc(X\setminus Z/K) \riso 
\sA^{\cc,W}_{X\setminus Z,Y} \otimesh_{W\sO_Y} \WOmd{}{Y}, }
%%%%%%%%%%%%%%%%
functorial with respect to the triple $(Z,X,Y)$, which is defined by \eqref{fMaincF} 
whenever the assumptions of {\normalfont{(i)}} are satisfied. Together with the 
isomorphisms \eqref{fMainLoc} for $X$ and $Z$, it defines a isomorphism of exact 
triangles from \eqref{fExLoc} to the triangle defined by the exact sequence \eqref{fExW}.

\romain If $X$ is proper, there exists in $\Db(K)$ a canonical isomorphism
%%%%%%%%%%%%%%%%
\eq{fMaincGlob}{ \RR\Gamma\rigc(U/K) \riso 
\RR\Gamma(X, \sA^{\cc,W}_{X\setminus Z,Y} \otimesh_{W\sO_Y} \WOmd{}{Y}), }
%%%%%%%%%%%%%%%%
functorial with respect to the triple $(Z, X, Y)$. Together with the isomorphisms 
\eqref{fMainGlob} for $X$ and $Z$, it defines an isomorphism of exact triangles from 
\eqref{fExGlob} to the triangle obtained by applying the functor $\RR\Gamma(X, -)$ to 
the exact sequence \eqref{fExW}. 
\end{cor}
%%%%%%%%%%%%%%%%%%%%%%%%%%%%%%%%

\begin{proof}
Assertion (i) is a straightforward consequence of the functoriality properties of the 
quasi-isomorphism \eqref{fMainF}, applied to the immersion $Z \inj X$. Assertion (ii) 
follows from the same functorialities, applied at the level of \v{C}ech resolutions
after choosing an open covering of $Y$ by affine open subsets $U_{\alpha}$ and closed
immersions $U_\alpha \inj (\P_\alpha,F_\alpha)$, as in the proof of Theorem 
\ref{ssMain}~(ii) given in \ref{ssProofMain}. Finally, assertion (iii) follows from 
the previous one by taking global sections. 
\end{proof}

%%%%%%%%%%%%%%%%%%%%%%%%%%%%%%%%
\section{Relation with Witt vector cohomology}\label{sWittCoh}
%%%%%%%%%%%%%%%%%%%%%%%%%%%%%%%%

We end this article with the relation between Theorem \ref{ssMain} (with its corollary
\ref{ssCorMain}) and Theorem 1.1 of \cite{BBE}, which identifies the slope $< 1$ part
of rigid cohomology with compact supports with Witt vector cohomology with compact 
supports.

%%%%%%%%%%%%%%%%%%%%%%%%%%%%%%%%
\subsection{}\label{ssMortoW}
%%%%%%%%%%%%%%%%%%%%%%%%%%%%%%%%
Let $Y$ be a smooth $k$-scheme of finite type, and let $X \inj Y$ be a closed subscheme
defined by an ideal $\sI \subset \sO_Y$. Let $W_n\sI := \Ker(W_n\sO_Y \to W_n\sO_X)$,
$W\sI := \Ker(W\sO_Y \to W\sO_X)$ be the subsheaves of Witt vectors with coefficients
in $\sI$. By construction, $\sP^{(0)\,W}_{X,Y,n}$ is the divided power envelope of
$W_n\sI$ with compatibility with the canonical divided powers of $VW_{n-1}\sO_Y$ (see
\eqref{fPmXY}). Since these divided powers extend to $W_n\sO_Y/W_n\sI = W_n\sO_X$ (as
the canonical divided powers of $VW_{n-1}\sO_X$), there are canonical ring
homomorphisms
%%%%%%%%%%%%%%%%
\eqn{ \sP^{(0)\,W}_{X,Y,n} \lra W_n\sO_X, \quad\quad \sPh^{(0)\,W}_{X,Y} \lra W\sO_X. }
%%%%%%%%%%%%%%%%
By composition with the canonical projection, we obtain a ring homomorphism 
%%%%%%%%%%%%%%%%
\eq{fAWtoW}{ \sA^W_{X,Y} = \varprojlim_m (\sPh^{(m)\,W}_{X,Y} \otimes K) \to 
\sPh^{(0)\,W}_{X,Y} \otimes K \to W\sO_{X,K}.  }
%%%%%%%%%%%%%%%%
We can now compose with the augmentation morphism, and we obtain a morphism of 
complexes 
%%%%%%%%%%%%%%%%
\eq{fAWOmtoW}{ \sA^W_{X,Y} \otimesh_{W\sO_Y} \WOmd{}{Y} \lra W\sO_{X,K}, }
%%%%%%%%%%%%%%%%
where we view $W\sO_{X,K}$ as a complex concentrated in degree $0$. 

Let $Z \subset X$ be a closed subscheme defined by an ideal $\sH \subset \sO_X$, and 
let $U = X \setminus Z$. By functoriality, we obtain a morphism of complexes 
%%%%%%%%%%%%%%%%
\eq{fAcWtoW}{ \sA^{\cc,W}_{X\setminus Z,Y} \otimesh_{W\sO_Y} \WOmd{}{Y} \lra 
(W\sO_{X,K} \to W\sO_{Z,K})\tot. }
%%%%%%%%%%%%%%%%
Note that the target complex is actually a resolution of $W\sH_K$. 

Taking the image of \eqref{fAWOmtoW} (resp.\ \eqref{fAcWtoW}) in $\Db(X,K)$, and
composing with the isomorphism \eqref{fMainLoc} (resp.\ \eqref{fMaincLoc}), we obtain
canonical morphisms
%%%%%%%%%%%%%%%%
\ga{fGamtoW}{ \RR\uGamma\rig(X/K) \lra W\sO_{X,K}, \\ 
\label{fGamctoW} \RR\uGamma\rigc(X\setminus Z/K) \lra W\sH_K. }
%%%%%%%%%%%%%%%%
Assuming that $X$ is proper, and taking global sections on $X$, they define canonical
morphisms
%%%%%%%%%%%%%%%%
\ga{fRigtoW}{ \RR\Gamma\rig(X/K) \lra \RR\Gamma(X, W\sO_{X,K}), \\ 
\label{fRigctoWc} \RR\Gamma\rigc(U/K) \lra \RR\Gamma_{\cc}(U, W\sO_{U,K}), }
%%%%%%%%%%%%%%%%
where the complex $\RR\Gamma_{\cc}(U, W\sO_{U,K})$ is the Witt vector cohomology with 
compact supports defined by \cite[(2.12)]{BBE}.

%%%%%%%%%%%%%%%%%%%%%%%%%%%%%%%%
\subsection{}\label{ssBBE}
%%%%%%%%%%%%%%%%%%%%%%%%%%%%%%%%
Let $X$ be a $k$-scheme of finite type, $Z \subset X$ a closed subscheme, and $U = X
\setminus Z$. In \cite{BBE}, morphisms 
%%%%%%%%%%%%%%%%
\ga{fGamtoWBBE}{ a_X : \RR\uGamma\rig(X/K) \lra W\sO_{X,K}, \\
\label{fGamctoWBBE} a_{X,U} : \RR\uGamma\rigc(X \setminus Z/K) \lra W\sH_K }
%%%%%%%%%%%%%%%%
are defined in $D^b(X,K)$ without using the de Rham-Witt complex of a smooth ambient
scheme (see \cite[Prop.\ 4.4 and Thm.\ 4.5]{BBE}). To compare these morphisms with
\eqref{fGamtoW} and \eqref{fGamctoW}, we now recall their construction.

Assume first that there exists a closed immersion $X \inj \P$, where $\P$ is a smooth 
formal scheme endowed with a lifting $F$ of the absolute Frobenius 
endomorphism. Then $F$ defines a ring homomorphism \cite[(4.4)]{BBE} 
%%%%%%%%%%%%%%%%
\eq{fAtoWBBE}{ \sA_{X,\P} \to W\sO_{X,K}. }
%%%%%%%%%%%%%%%%
One can then define a morphism $\sA_{X,\P}\otimes \Omega\hbul_{\P} \to W\sO_{X,K}$ by
composition with the augmentation morphism of $\sA_{X,\P}\otimes \Omega\hbul_{\P}$
\cite[(4.5)]{BBE}. Using the identifications of \ref{ssTubes}, its image in $D^b(X,K)$
does not depend on the choice of the embedding $X \inj \P$, and defines a canonical
morphism $a_X$. Using \v{C}ech resolutions as usual, the construction of $a_X$ can be
extended to the general case where an embedding into some $(\P,F)$ might not exist.
Finally, the construction of $a_{X,U}$ follows using the functoriality of the
construction of $a_X$ with respect to $Z \inj X$.

%%%%%%%%%%%%%%%%%%%%%%%%%%%%%%%%
\begin{lem}\label{ssaXY}
Let $Y$ be a smooth $k$-scheme, and let $Z \inj X \inj Y$ be closed subschemes. Then
the morphisms \eqref{fGamtoWBBE} and \eqref{fGamctoWBBE} are respectively equal to
\eqref{fGamtoW} and \eqref{fGamctoW}.
\end{lem}
%%%%%%%%%%%%%%%%%%%%%%%%%%%%%%%%

\begin{proof}
We consider first the case where there exists a closed immersion $Y \inj \P$, where 
$\P$ is a smooth formal scheme over $W$ endowed with a Frobenius lifting $F$. Then
$a_{X}$ is defined as the image in $\Db(X,K)$ of the composition
%%%%%%%%%%%%%%%%
\eqn{ a_{X,\P} : \sA_{X,\P} \otimes_{\sO_{\P}} \Omega\hbul_{\P} \lra \sA_{X,\P} 
\xra{\eqref{fAtoWBBE}} W\sO_{X,K}. }
%%%%%%%%%%%%%%%%
On the other hand, \eqref{fGamtoW} is defined as the image of the composition 
%%%%%%%%%%%%%%%%
\eqn{ \sA_{X,\P} \otimes_{\sO_{\P}} \Omega\hbul_{\P} 
\xrightarrow[\text{qis}]{\eqref{fMainF}} 
\sA^W_{X,Y} \otimesh_{W\sO_Y} \WOmd{}{Y} \lra \sA^W_{X,Y} \xra{\eqref{fAWtoW}} W\sO_{X,K}, }
%%%%%%%%%%%%%%%%
where the middle arrow is the augmentation morphism. Thus it suffices to check the
commutativity of the diagram of morphisms of complexes 
%%%%%%%%%%%%%%%%
\eq{fCompBBEloc}{ \xymatrix{  
\sA_{X,\P} \otimes_{\sO_{\P}} \Omega\hbul_{\P} \ar[r] \ar[d]_-{\eqref{fMainF}} 
& \sA_{X,\P} \ar[d]_-{\eqref{fMainF}} \ar[rd]^-{\text{\eqref{fAtoWBBE}}} &  \\
\sA^W_{X,Y} \otimesh_{W\sO_Y} \WOmd{}{Y} \ar[r] & \sA^W_{X,Y} 
\ar[r]^-{\eqref{fAWtoW}} &  W\sO_{X,K}. 
} }
%%%%%%%%%%%%%%%%
Here the middle vertical arrow is the map induced in degree $0$ by \eqref{fMainF}, so  
that the square on the left hand side obviously commutes. 

To check the commutativity of the triangle, we keep the notation of \ref{ssProofMain}. 
Observe first that the composition of \eqref{fAWtoW} and \eqref{fMainF} 
can also be written as the composition 
%%%%%%%%%%%%%%%%
\eqn{ \sA_{X,\P} \lra \sPh(\sJ)_K  \lra \sPh(W(\sI))_K \lra W\sO_{X,K},  }
%%%%%%%%%%%%%%%%
where the first map is given by the isomorphism \eqref{fAPinvlim} composed with the
projection on the factor of index $0$, and the second one is induced by the morphism
\eqref{dREtodRWE} defined by the crystal structure of $i\cris{}_{\,*}\sO_{X/W}$. On the
other hand, \eqref{fAtoWBBE} is also defined as a composition of maps starting with the
homomorphism $\sA_{X,\P} \to \sPh(\sJ)_K$ (see \cite[(4.7)]{BBE}). Therefore, it
suffices to check the commutativity of the triangle 
%%%%%%%%%%%%%%%%
\eqn{ \xymatrix{
\sPh(\sJ) \ar[d]_-{\eqref{dREtodRWE}} \ar[dr] \\
\sPh(W(\sI)) \ar[r] & W\sO_X. } } 
%%%%%%%%%%%%%%%%
This is now a consequence of the universal property of PD-envelopes, as the oblique 
map is the canonical factorization of 
%%%%%%%%%%%%%%%%
\eqn{ \sO_{\P} \xra{t_F} W\sO_P \lra W\sO_X, }
%%%%%%%%%%%%%%%%
the vertical one is the canonical factorization of 
%%%%%%%%%%%%%%%%
\eqn{ \sO_{\P} \xra{t_F} W\sO_P \lra W\sO_Y \lra \sPh(W(\sI)), }
%%%%%%%%%%%%%%%%
and the horizontal one is the canonical factorization of $W\sO_Y \to W\sO_X$. 

In the general case, one chooses as usual an affine open covering $(U_{\alpha})$ of $Y$
and closed immersions $U_{\alpha} \inj \P_\alpha$ into smooth formal schemes
over $W$ endowed with Frobenius liftings $F_{\alpha}$. The previous argument can
then be repeated simplicially. Taking the image in $\Db(X,K)$ of the associated total
complexes, the first assertion of the lemma follows.

The assertion relative to $X \setminus Z$ follows similarly thanks to the functoriality
of diagram \eqref{fCompBBEloc} with respect to $Z \inj X$.
\end{proof}

%%%%%%%%%%%%%%%%%%%%%%%%%%%%%%%%
\begin{prop}\label{ssCompBBE}
Let $X$ be a proper $k$-scheme, $U = X \setminus Z$ an open subset of $X$, and $X \inj 
Y$ a closed immersion into a smooth $k$-scheme. For each $q \geq 0$, there is a 
commutative diagram 
%%%%%%%%%%%%%%%%
\eq{fCompBBE}{ \xymatrix{ 
H^q\rigc(U/K) \ar@{>>}[d]_-{\can} \ar[r]_-{\sim}^-{\eqref{fMaincGlob}} & 
H^q(X, \sA^{\cc,W}_{X\setminus Z,Y} \otimesh_{W\sO_Y} \WOmd{}{Y}) 
\ar[d]^-{\eqref{fAcWtoW}} \\
H^q\rigc(U/K)^{< 1} \ar[r]_-{\sim}^-{b^q_U} & H^q_{\cc}(U, W\sO_{U,K}),
} }
%%%%%%%%%%%%%%%%
where the left vertical arrow is the projection on the slope $< 1$ part, and the lower 
isomorphism is the canonical identification between the slope $< 1$ part of rigid 
cohomology with compact supports and Witt vector cohomology with compact supports 
\cite[Theorem 1.1]{BBE}.
\end{prop}
%%%%%%%%%%%%%%%%%%%%%%%%%%%%%%%%

\begin{proof}
Let us first recall the construction of the isomorphism $b^q_U$ (see \cite[5.2]{BBE}).
Applying the functor $\RR\Gamma(X, -)$ to the morphism $a_{U,X}$ and passing to
cohomology in degree $q$, we obtain a homomorphism
%%%%%%%%%%%%%%%%
\eqn{ a^q_U : H^q\rigc(U/K) \lra H^q_{\cc}(U, W\sO_{U,K}). }
%%%%%%%%%%%%%%%%
As the slopes of the Frobenius action on $H^q_{\cc}(U, W\sO_{U,K})$ belong to $[0, 
1[$, $a^q_U$ vanishes on the subspaces of slope $\lambda$ of $H^q\rigc(U/K)$ for any 
$\lambda \notin [0,1[$. Thanks to the slope decomposition of $H^q\rigc(U/K)$, it 
follows that $a^q_U$ factorizes through $H^q\rigc(U/K)^{<1}$, which provides the
homomorphism $b^q_U$. But the previous lemma implies that $a^q_U$ is equal to the 
homomorphism induced on $H^q$ by \eqref{fRigctoWc}. As the latter is the composition of
the homomorphisms induced by \eqref{fMaincGlob} and \eqref{fAcWtoW}, this proves the
proposition.
\end{proof}

%%%%%%%%%%%%%%%%%%%%%%%%%%%%%%%%
\begin{prop}\label{ssPartDegen}
Let $Y$ be a smooth $k$-scheme of finite type, and $X \inj Y$ a closed subscheme.

\romain The homomorphism $\sA^W_{X,Y} \to W\sO_{X,K}$ defined in \eqref{fAWtoW} 
is an epimorphism.

\romain Assume that $X$ is proper over $k$, and let $H^q\rig(X/K)^{\geq 1}$ denote the
component of slope $\geq 1$ of $H^q\rig(X/K)$. If we set
%%%%%%%%%%%%%%%%
\eq{fFilCan}{ \Fil^1 (\sA^W_{X,Y} \otimesh_{W\sO_Y} \WOmd{}{Y}) := 
\Ker(\sA^W_{X,Y} \otimesh_{W\sO_Y} \WOmd{}{Y} \xra{\eqref{fAWOmtoW}} W\sO_{X,K}), }
%%%%%%%%%%%%%%%%
the cohomology long exact sequence defined by \eqref{fFilCan} splits into short exact 
sequences 
%%%%%%%%%%%%%%%%
\eq{fShortSeq}{ 0 \to H^q(X, \Fil^1 (\sA^W_{X,Y} \otimesh_{W\sO_Y} \WOmd{}{Y})) \to 
H^q\rig(X/K) \to H^q(X, W\sO_{X,K}) \to 0, }
%%%%%%%%%%%%%%%%
which yield isomorphisms 
%%%%%%%%%%%%%%%%
\eq{fHighSlopes}{ H^q(X, \Fil^1 (\sA^W_{X,Y} \otimesh_{W\sO_Y} \WOmd{}{Y})) \riso 
H^q\rig(X/K)^{\geq 1}. }
%%%%%%%%%%%%%%%%
\end{prop}
%%%%%%%%%%%%%%%%%%%%%%%%%%%%%%%%

\begin{proof}
Let $\sI$ be the ideal of $X$ in $Y$. We first fix $m \geq 0$. For each $n \geq 1$,
$\sP^{(m)\,W}_{X,Y,n}$ is the PD-envelope of the ideal $W_n\sI^{(m)}$ with
compatibility with the canonical divided powers of $VW_{n-1}\sO_Y$, hence the canonical
homomorphism $\sP^{(m)\,W}_{X,Y,n} \to W\sO_{X^{(m)}}$ is surjective, and its kernel 
is the PD-ideal $\overline{W_n\sI^{(m)}}$ generated by $W_n\sI^{(m)}$. It follows from 
\cite[3.20, 3)]{BO} that the homomorphisms $\overline{W_{n+1}\sI^{(m)}} \to 
\overline{W_n\sI^{(m)}}$ are surjective. As each $\overline{W_n\sI^{(m)}}$ has 
vanishing cohomology over affine open subsets, the inverse system 
$(\overline{W_n\sI^{(m)}})_n$ is $\varprojlim$-acyclic, hence the homomorphism 
%%%%%%%%%%%%%%%%
\eqn{ \sPh^{(m)\,W}_{X,Y} \lra W\sO_{X^{(m)}} }
%%%%%%%%%%%%%%%%
is an epimorphism, with kernel $\overline{W\sI^{(m)}} := \varprojlim_n
\overline{W_n\sI^{(m)}}$.

Therefore, we obtain for variable $m$ an exact sequence of inverse systems 
%%%%%%%%%%%%%%%%
\eqn{ 0 \to \overline{W\sI^{(m)}} \otimes K \to \sPh^{(m)\,W}_{X,Y} \otimes K 
\to W\sO_{X^{(m)},K} \to 0. }
%%%%%%%%%%%%%%%%
By \cite[Prop.\ 2.1, (i)]{BBE}, the transition morphisms $W\sO_{X^{(m+1)},K} \to
W\sO_{X^{(m)},K}$ are isomorphisms, hence assertion (i) follows if we show that the
inverse system $(\overline{W\sI^{(m)}} \otimes K)_m$ is $\varprojlim$-acyclic. The
algebraic Mittag-Leffler criterion implies that the terms of this inverse system have
vanishing cohomology on affine open subsets. It is then sufficient to check that, when
$Y$ is affine, the system $(\overline{W\sI^{(m)}} \otimes K)_m$ statisfies the
topological Mittag-Leffler criterion. This can be done as in the proof of
\eqref{fVanRlimA}, the only difference being that, in the family of generators described
in \eqref{fPDgen}, we must have $\uk \neq (0,\ldots,0)$.

Using assertion (i), the definition \eqref{fFilCan} of $\Fil^1(\sA^W_{X,Y}
\otimesh_{W\sO_Y} \WOmd{}{Y})$ provides a long exact sequence of cohomology, in which 
the middle terms can be identified with $H^q\rig(X/K)$ using \eqref{fMainGlob}. 
This sequence splits because diagram \eqref{fCompBBE} implies that the homomorphisms 
$H^q\rig(X/K) \to H^q(X, W\sO_{X,K})$ are surjective for all $q$. Then the last
asseertion follows from \cite[Th.\ 1.1]{BBE}.
\end{proof}

%%%%%%%%%%%%%%%%%%%%%%%%%%%%%%%%
\begin{rmk}\label{ssLastRem}
The behaviour of the long exact sequence defined by \eqref{fFilCan} is of course
reminiscent of the degeneration of the slope spectral sequence for $H^*\rig(X/K) \simeq
H^*\cris(X/W)\otimes K$ when $X$ is proper and smooth over $k$ \cite[Th.\ 3.2, Cor.\
3.5]{Il79}.

It is therefore natural (but maybe too naive) to ask the following question: assuming 
$X$ proper, but maybe singular, can one generalize the definition \eqref{fFilCan} so as
to define a functorial filtration of $\sA^W_{X,Y} \otimesh_{W\sO_Y} \WOmd{}{Y}$ by
subcomplexes $\Fil^i (\sA^W_{X,Y} \otimesh_{W\sO_Y} \WOmd{}{Y})$, depending only on $X$
in $D^b(X,K)$ and quasi-isomorphic to $\WOm{}{\geq i}{X,K}$ when $X$ is smooth over
$k$, such that their cohomology computes the slope $\geq i$ component of $H^q\rig(X/K)$
for all $q$?

It should be pointed out here that the work of Nakkajima \cite{Nak} already provides a
construction of the slope spectral sequence for $H^*\rig(X/K)$, based on the de
Rham-Witt cohomology of a proper hypercovering of $X$ by smooth varieties. 
Nevertheless, a positive answer to the previous question would be interesting if the 
cohomology of the $\Fil^i$ was more accessible to computation. For example, one might 
hope to obtain generalizations of Katz's congruences mod $p^{a\kappa}$ for the number 
of rational points of $X$ when $X$ is a subvariety of the projective space over a
finite field with $p^a$ elements \cite{Ka71}. For $\kappa = 1$, these congruences are
indeed a consequence of the identification of the slope $< 1$ component of rigid
cohomology with Witt vector cohomology \cite[Cor.\ 1.5]{BBE}.
\end{rmk}
%%%%%%%%%%%%%%%%%%%%%%%%%%%%%%%%

\bibliographystyle{plain}
\renewcommand\refname{References}

\end{document}